\documentclass[11pt]{amsart}

\usepackage[letterpaper]{geometry}
\usepackage{mathabx}

\RequirePackage{amsthm,amsmath,amsfonts,amssymb}
\RequirePackage[authoryear]{natbib}
\setcitestyle{authoryear,open={(},close={)}} 

\RequirePackage[colorlinks,linkcolor=purple,citecolor=blue,urlcolor=blue]{hyperref}
\RequirePackage{graphicx}

\RequirePackage{amsthm,amsmath,amsfonts,amssymb}

\usepackage{amsmath, amssymb, amscd, amsthm, amsfonts,pbsi,aurical}
\usepackage{amssymb,amsthm,amsfonts,amsbsy,latexsym,amsxtra}
\usepackage{dsfont,todonotes}
\usepackage{graphicx,xcolor,color,bm}
\usepackage{mathtools,enumitem,todonotes,soul}
\usepackage[T1]{fontenc}
\usepackage[utf8]{inputenc}
\usepackage{float,mdframed}
\usepackage{mathrsfs} 

\usepackage{tikz}

\numberwithin{equation}{section}

\makeatletter
\def\namedlabel#1#2{\begingroup
    #2%
    \def\@currentlabel{#2}%
    \phantomsection\label{#1}\endgroup
}
\makeatother

\newtheorem{Theorem}{Theorem}[section]

\newtheorem{Remark}{Remark}[section]

\newtheorem{Corollary}{Corollary}[section]
\newtheorem{Definition}{Definition}[section]
\newtheorem{Lemma}{Lemma}[section]

\newcommand{\mbb}{\mathbb}
\newcommand{\mbf}{\mathbf}

\newcommand{\mcal}{\mathcal}

\newcommand{\E}{\mathbf{E}}

\newcommand{\Nb}{\mathbf{N}}

\newcommand{\cX}{\mathcal{X}}

\newcommand{\xd}{\tilde}
\newcommand{\XD}{\widetilde}
\newcommand{\xde}{\check}
\newcommand{\XDE}{\widecheck}

\newcommand{\X}{{\mathbb X}}

\DeclareMathOperator{\Var}{Var}

\def\Biggiven{\,\Big{|}\,}
\def\tr{\mathop{\text{tr}}\kern.2ex}

\def\tV{{\tilde V}}
\def\tW{{\tilde W}}

\def\E{{\mathrm E}}

\def\d{{\mathrm d}}

\allowdisplaybreaks

\title[Approximation Bounds for ATE Estimators]{\sf \bf Gaussian and Bootstrap Approximation for Matching-based Average Treatment Effect Estimators}

\author[Z. Shi]{Zhaoyang Shi}
\address{Department of Statistics, Harvard University.}
\email{zshi@fas.harvard.edu}

\author[C. Bhattacharjee]{Chinmoy Bhattacharjee}
\address{Department of Mathematics, University of Hamburg.}
\email{chinmoy.bhattacharjee@uni-hamburg.de}

\author[K. Balasubramanian]{Krishnakumar Balasubramanian}
\address{Department of Statistics, University of California, Davis.}
\email{kbala@ucdavis.edu}

\author[W. Polonik]{Wolfgang Polonik}
\address{Department of Statistics, University of California, Davis.}
\email{wpolonik@ucdavis.edu}

\begin{document}

\begin{abstract}
We establish Gaussian approximation bounds for covariate and rank-matching-based Average Treatment Effect (ATE) estimators. By analyzing these estimators through the lens of stabilization theory, we employ the Malliavin-Stein method to derive our results. Our bounds precisely quantify the impact of key problem parameters, including the number of matches and treatment balance, on the accuracy of the Gaussian approximation. Additionally, we develop multiplier bootstrap procedures to estimate the limiting distribution in a fully data-driven manner, and we leverage the derived Gaussian approximation results to further obtain bootstrap approximation bounds. Our work not only introduces a novel theoretical framework for commonly used ATE estimators, but also provides data-driven methods for constructing non-asymptotically valid confidence intervals. 
\end{abstract}

\maketitle

\section{Introduction}

Nearest neighbor matching estimators are non-parametric methods in causal inference used to estimate treatment effects by comparing \emph{treated} and \emph{untreated} units that are most similar in observable characteristics. This approach is commonly applied in observational studies where random assignment is not possible, necessitating statistical methods to estimate counterfactual outcomes (i.e., what would have happened to a treated unit if it had not received treatment). Matching involves pairing each treated unit with one or more untreated units that closely resemble it based on these characteristics, creating a comparison group that approximates the treated group but without the treatment. This process helps estimate the treatment effect by minimizing confounding due to observed differences. In particular, the aforementioned procedure is used to calculate the Average Treatment Effect (ATE) which reflects the treatment's effect across the entire population, including those who did and did not receive the treatment. Such ATE estimators have been widely used in various fields~\citep{imbens2004nonparametric,morgan2006matching,rosenbaum2010design,stuart2010matching}.

In two seminal works, \cite{abadie2006large,abadie2011bias} proposed ATE estimators based on nearest neighbor matching and provided their first asymptotic analysis. In particular~\cite{abadie2006large} showed that their proposed estimator has non-negligible bias if the dimension is greater than one. As a remedy,~\cite{abadie2011bias} proposed a bias-correction procedure and established asymptotic properties (including asymptotic normality) under the crucial assumption that the number of matches is fixed. More recently, \cite{lin2023estimation} established the consistency and asymptotic normality for matching-based ATE estimators allowing the number of matches to diverge with the number of observations. Through their analysis, they further showed that the estimator is doubly robust and semiparametrically efficient. 


Our primary objective in this work is to derive precise Gaussian approximation results for nearest-neighbor matching-based ATE estimators. Existing asymptotic normality results, commonly used for constructing confidence intervals, have significant limitations. In particular, they provide no information on when Gaussianity ``kicks in'', making the resulting confidence intervals valid only asymptotically. In addition, key parameters, such as the number of matches and the balance between treatment groups, become obscured in asymptotics. For instance, the rate at which the number of matches increases with the sample size directly impacts the accuracy of the Gaussian approximation, and consequently, the validity of the constructed confidence intervals. Establishing a fine-grained Gaussian approximation bound in this context would allow us to quantify these effects non-asymptotically, improving the reliability of the confidence intervals for the practitioner.



To this end, we introduce a novel approach for quantifying the accuracy of Gaussian approximations in matching-based ATE estimators. Our method combines stabilization theory with the \emph{Malliavin-Stein} method. A key observation for our analysis is that the leading term in the bias-corrected ATE estimator (specifically, the term \( E_n \) defined at \eqref{eq2:F}) exhibits a local geometric property termed as \emph{stabilization}, as illustrated in detail in Section \ref{sec:ATEstabilization} later. To the best of our knowledge, only~\cite{shi2022berry} has previously applied Stein's method in the context of causal inference problems. However, their work focused on leveraging classical results from the Stein's method literature on permutational statistics to derive Berry–Esseen bounds for design-based causal inference.

By leveraging the aforementioned insight and by refining the Gaussian approximation results in \cite{lachieze2019normal} (see our Theorem \ref{rateswithconstant} specifically), we present in Theorem \ref{thm:covariate} the first Gaussian approximation result for the bias-corrected matching-based ATE estimator. This result explicitly quantifies the approximation’s accuracy based on key parameters, including the number of matches and the treatment group balance. For example, a consequence of our result in the one-dimensional setting with balanced data, gives the 
 Gaussian approximation upper bound $M^5n^{-\frac{1}{2}}+M^{-\frac{1}{2}}$, where $M$ is the number of nearest neighbor matches and $n$ is the number of observations; see Corollary~\ref{coro:covariate} and~\eqref{dim1} for details. Similarly, in Theorem \ref{thm:rank}, we establish Gaussian approximation results for the \(\phi\)-transformed rank-based ATE estimator proposed and analyzed in~\cite{cattaneo2023rosenbaum}. 

On a more technical note, another contribution of our work is a refinement on Theorems B.3 and B.4 in \citet{lin2023estimation}. In this context, we derive a mathematically rigorous, fully non-asymptotic bound for the estimation error of the nearest-neighbor-based density ratio, as presented in Lemma \ref{lemma:l2convdensityratio}. In contrast, the error bounds in \cite{lin2023estimation} included asymptotic simplifications tailored to their purpose without providing a fully non-asymptotic expression. 

As an application of our main results, we analyze a multiplier bootstrap method to estimate the limiting distribution and establish bootstrap approximation rates for both the covariate-based and rank-based ATE estimators in Theorem \ref{thm:boot}. Notably, our bootstrap approximation results allow the number of matches to increase with the sample size. This contrasts with the findings in~\cite{abadie2008failure}, which demonstrate that the naive bootstrap procedure is inconsistent when the number of matches remains fixed. Our results, on the other hand, specify the exact rate at which the number of matches can diverge with the sample size for the multiplier bootstrap method to remain consistent. 

During the final stages of preparing this manuscript, we became aware of a concurrent work by~\cite{lin2024consistency} that establishes consistency results for the naive bootstrap procedure when the number of matches is allowed to grow with the sample size. We would like to point out three significant distinctions between this and our current work: (i) they employ the naive bootstrap procedure, similar to~\cite{abadie2008failure}, (ii) they do not provide rates for bootstrap approximation, and (iii) their proof techniques differ fundamentally, being more canonical, whereas our approach relies on the Malliavin-Stein method and stabilization theory.

\section{Notation}
Throughout the paper, we will use the following notation and conventions.
\begin{itemize}
    \item For an integer $n>0$, $[n]:=\{1,2,\ldots,n\}$.
    \item $\mathds{1}(A)$: the indicator function of $A$.
    \item $\text{Bin}(n,p)$: binomial random variable with parameters $n$ and $p$.
    \item $\mcal{N}(a,b)$: normal random variable with mean $a$ and variance $b$; When $a=0$ and $b=1$, we simply use $\mcal{N}$ to denote a standard normal random variable.
    \item For any two real sequences $\{a_n\}$ and $\{b_n\}$, we write $a_n\lesssim b_n$ (or equivalently, $b_n\gtrsim a_n$, $a_n=O(b_n)$) if there exists a constant $C>0$ such that $a_n\le Cb_n$ for $n$ large enough. We also write $a_n\asymp b_n$ if $a_n\lesssim b_n$ and $a_n\gtrsim b_n$.
    \item $d(\cdot,\cdot)$ represents the Euclidean distance in $\mbb{R}^{m}$ and $B(x,r)$ denotes the closed Euclidean ball centered at $x$ with radius $r$. For $m \in {\mathbb N}$, we let $V_m$ be the volume of the  unit ball in ${\mathbb R}^m$.
    \item $\X \subseteq \mbb{R}^m$ with $m \in \mbb{N}$ represents the support of the covariate $X$ in our model.
    \item For a set $A \subseteq \mbb{R}^m$, we denote $\XD{A}:=A\times \{0,1\}$. The set $\XD{\X}$ represents the support of the covariate pair $(X,D)$.
    \item For a set $A \subseteq \mbb{R}^m$, we denote $\XDE{A}:=A\times \{0,1\}\times \mbb{R}$. The set   $\XDE{\X}$ represents the support of the triplet $(X,D,\bm{\varepsilon})$.
    \item $\xd{x}:=(x,d)$ and $\xde{x}:=(x,d,\varepsilon)$ represent elements in $\XD{\X}$ and $\XDE{\X}$, respectively.
    \item $\cX_n$, $\XD{\cX}_n$ and $\XDE{\cX}_n$ stand for the (marked) point collections $\{X_{i}\}_{i=1}^{n}$, $\{\XD{X}_i\}_{i=1}^{n}$ and $\{\XDE{X}_i\}_{i=1}^{n}$, respectively.
    \item $\mbb{Q}$: the probability measure associated to the distribution of $X$; $\XDE{\mbb{Q}}$: the joint probability distribution of the triplet $(X,D,\bm{\varepsilon})$.
    \item Throughout the paper, $C$ stands for a generic finite positive constant whose actual value may vary from line to line in our computations. We do this to simplify many expressions, where the constants do not depend on the parameters of interest to us. For such parameters (such as $\eta,n$ and $M$), say $\Delta$, we specifically write $C(\Delta)>0$ to denote a finite positive constant that depends on $\Delta$.
\end{itemize}

\section{Matching-based Average Treatment Effect Estimators}

Following the framework by \cite{abadie2006large,abadie2011bias}, we are interested in estimating the average treatment effect (ATE) on outcomes in a binary treatment experiment. Consider $(X,Y,D) \in \mathbb{X} \times \mathbb{R} \times \{0,1 \}$, where $X$ corresponds to the unit, $Y$ corresponds to the response, and $D$ is the binary variable (possibly dependent on $X$) such that $D=1$ and $D=0$ corresponds to when the unit $X$ belongs to the treatment and the control groups, respectively. Under the framework of \cite{rubin1974estimating} (also see \cite{rosenbaum1995observationalstudies,imbens2009recent}), $X$ has two potential outcomes, $Y(0)$ and $Y(1)$, depending on whether $D=0$ or $D=1$, but we observe only one of them. In other words, $Y=DY(1)+(1-D)Y(0)$. The central goal is to estimate the population ATE defined as
\begin{equation}\label{eq:ATE}
    \tau=\mbb{E}\left(\mbb{E}(Y|X,D=1)-\mbb{E}(Y|X,D=0)\right),
\end{equation}
given observations $\{(X_i,Y_i,D_i)\}_{i=1}^{n}$ that are assumed to be independent and identically distributed copies of $(X,Y,D)$.

\subsection{Covariate-based matching}\label{sec:covariateATE}
We first discuss covariate-based matching estimators for ATE. Let $n_1=\sum_{i=1}^{n}D_i$ and $n_0=\sum_{i=1}^{n}(1-D_i) = n-n_1$ denote the number of treated and control units, respectively. Note here that while $n$ is a deterministic variable, $n_0$ and $n_1$ are \emph{random variables} depending on the specific instance of $\{D_i\}_{i=1}^n$.
For a point collection $\xd{\nu}_n=\{\xd{x}_{i} \in \XD{\X} : i \in [n]\}$, $\omega\in \{0,1\}$ and an integer $M>0$, we denote by $\mcal{J}^{\omega}_{M}(i, \xd{\nu}_n)$ the index set of $M$-NNs of ${x}_i$ within the set $\nu_n^\omega :=\{x \in \X : (x,\omega) \in \xd{\nu}_n\}$, namely, the set of all indices $j\in[n]$ such that $d_j=\omega$ and
\begin{align*}
    \sum_{l=1,d_{l}=\omega}^{n}\mathds{1}(d(x_l,x_i)\le d(x_j,x_i))\le M.
\end{align*}

Since only one of the potential outcomes $Y_i(0)$ and $Y_{i}(1)$ is observed, we impute the missing potential outcomes (see \cite{abadie2006large}) via nearest neighbor (NN) matching as
\begin{align*}
    \hat{Y}_{i}(0)=\left\{
    \begin{aligned}
        &Y_{i},\ \qquad\qquad \text{if}\ D_i=0;\\
        &\frac{1}{M}\sum_{j\in\mcal{J}^{0}_{M}(i,\XD{\cX_n})}Y_j,\ \text{if}\ D_{i}=1,
    \end{aligned}
    \right.
\end{align*}
and
\begin{align*}
    \hat{Y}_{i}(1)=\left\{
    \begin{aligned}
        &\frac{1}{M}\sum_{j\in\mcal{J}^{1}_{M}(i,\XD{\cX_n})}Y_j,\ \text{if}\ D_{i}=0;\\
        &Y_{i},\ \qquad\qquad \text{if}\ D_i=1.
    \end{aligned}
    \right.
\end{align*}
Then, the matching-based estimator for ATE is defined as the empirical counterpart of \eqref{eq:ATE} based on the $\hat Y_i(0)$ and $\hat Y_i(1)$, i.e., 
$$\hat \tau_M := \frac{1}{n}\sum_{i=1}^n \hat Y_i(1) - \frac{1}{n}\sum_{i=1}^n \hat Y_i(0).
$$

Let $K^{\omega}_{M}(i,\xd{\nu}_n)$ be the matched times for unit $i$ with $d_i=\omega \in \{0,1\}$, i.e.,
\begin{align}\label{matchedtimes}
    K^{\omega}_{M}(i,\xd{\nu}_n)=\sum_{j=1,d_j=1-\omega}^{n}\mathds{1}(i\in \mcal{J}^{\omega}_{M}(j,\xd{\nu}_n)).
\end{align}
In other words, $K^{\omega}_{M}(i,\xd{\nu}_n)$ denotes the total number of units $j$ in $\xd{\nu}_n$ with $d_j=1-\omega$ (i.e. with the opposite label to $x_i$) such that $x_i$ is one of its $M$-NNs in $\nu_n^\omega$. Then $\hat{\tau}_{M}$ can be further expanded as
\begin{align*}
    \hat{\tau}_{M} =\frac{1}{n}\bigg(\sum_{i=1,D_i=1}^{n}\bigg(1+\frac{K_{M}^{1}(i,\XD{\cX}_n)}{M}\bigg)Y_i-\sum_{i=1,D_i=0}^{n}\bigg(1+\frac{K_{M}^{0}(i,\XD{\cX}_n)}{M}\bigg)Y_i\bigg).
\end{align*}

The estimator $\hat{\tau}_{M}$ however suffers from an asymptotically non-negligible bias when the dimension $m$ is strictly larger than $1$ (\cite{abadie2006large}). To circumvent this, in a follow-up work, \cite{abadie2011bias} proposed a bias-corrected version $\hat{\tau}_{M}^{bc}$ (see \eqref{eq:F} below), defined as follows. Consider the regression model $Y_i=\mu_{D_i}(X_i)+\bm{\varepsilon}_{i}$, $i\in[n]$, define conditional expectations 
\begin{align*}
    \mu_0(x):=\mbb{E}(Y|X=x,D=0),\quad \mu_1(x):=\mbb{E}(Y|X=x,D=1),
\end{align*}
and let $\hat{\mu}_{0}(x)$ and $\hat{\mu}_{1}(x)$, respectively, be corresponding regression estimators. The estimator $\hat{\tau}_{M}^{bc}$ is then obtained by replacing the $\hat Y_i(0)$ and $\hat Y_i(1)$ in $\hat \tau_M$ by the corrected quantities $\hat Y_i(0) + \hat \mu_0(X_i) - \hat \mu_1(X_i)$ and $\hat Y_i(1) + \hat \mu_1(X_i) - \hat \mu_0(X_i),$ respectively. Writing the residuals as 
\begin{align*}
    \hat{R}_{i}:=Y_i-\hat{\mu}_{D_{i}}(X_i),\ i\in [n],
\end{align*}
and denoting the regression based estimator of the population ATE as
\begin{align*}
    \hat{\tau}_{\text{reg}}:=\frac{1}{n}\sum_{i=1}^{n}(\hat{\mu}_{1}(X_i)-\hat{\mu}_{0}(X_i)),
\end{align*}
the bias corrected estimator of ATE can be expressed as
\begin{align}\label{eq:F}
\hat{\tau}_{M}^{\text{bc}}:=\hat{\tau}_{\text{reg}}+\frac{1}{n}\bigg(\sum_{i=1,D_i=1}^{n}\bigg(1+\frac{K^{1}_{M}(i,\XD{\cX}_n)}{M}\bigg)\hat{R}_{i}-\sum_{i=1,D_i=0}^{n}\bigg(1+\frac{K^{0}_{M}(i,\XD{\cX}_n)}{M}\bigg)\hat{R}_{i}\bigg).
\end{align}

\cite{lin2023estimation} showed that the estimator \eqref{eq:F} is indeed \emph{doubly robust}. Consequently, it should also enjoy all the desirable properties of doubly robust estimators \citep[see][]{scharfstein1999adjusting,bang2005doubly}. Moreover, following \cite{lin2023estimation}, the estimator \eqref{eq:F} can be conveniently decomposed as
\begin{align}\label{eq2:F}
\begin{aligned}
    \hat{\tau}_{M}^{\text{bc}}&=\frac{1}{n}\sum_{i=1}^{n}(\hat{\mu}_{1}(X_i)-\hat{\mu}_{0}(X_i))+\frac{1}{n}\sum_{i=1}^{n}(2D_i-1)\bigg(1+\frac{K^{D_i}_{M}(i,\XD{\cX}_n)}{M}\bigg)(Y_i-\hat{\mu}_{D_i}(X_i))\\
    &=\frac{1}{n}\sum_{i=1}^{n}(\mu_{1}(X_i)-\mu_{0}(X_i))+\frac{1}{n}\sum_{i=1}^{n}(2D_i-1)\bigg(1+\frac{K^{D_i}_{M}(i,\XD{\cX}_n)}{M}\bigg)\bm{\varepsilon}_i+(B_{M}-\hat{B}_{M})\\
    &=:E_{n}+(B_{M}-\hat{B}_{M}),
\end{aligned}
\end{align}
where $\bm{\varepsilon}_i$ are the errors in our regression model, and 
\begin{align*}
    B_M&:=\frac{1}{n}\sum_{i=1}^{n}(2D_i-1)\bigg(\frac{1}{M}\sum_{m=1}^{M}(\mu_{1-D_i}(X_i)-\mu_{1-D_i}(X_{j_{m}^{1-D_i}(i,\XD{\cX}_n)}))\bigg),\\
    \hat{B}_{M}&:=\frac{1}{n}\sum_{i=1}^{n}(2D_i-1)\bigg(\frac{1}{M}\sum_{m=1}^{M}(\hat{\mu}_{1-D_i}(X_i)-\hat{\mu}_{1-D_i}(X_{j_{m}^{1-D_i}(i,\XD{\cX}_n)}))\bigg),
\end{align*}
with $j_{m}^{\omega}(i,\XD{\cX}_n)$ denoting the $m$-th nearest neighbor of the point $X_i$ in $\{X_j:D_j=\omega\}_{j=1}^{n}$ for $m\in[n]$. Here, $E_n$ can be viewed as the main contributing term and $(B_{M}-\hat{B}_{M})$ as the bias term. In Theorem \ref{thm:covariate} and Corollary \ref{coro:covariate} in Section \ref{sec:mainresult}, using stabilization theory and Malliavin-Stein method, we provide a quantitative estimate for the error in the Gaussian approximation of $\hat{\tau}_{M}^{\text{bc}}$ (appropriately centered and scaled).

\subsection{Rank-based matching}\label{sec:rank}

The above covariate-based matching uses the Euclidean distance for determining the nearest neighbor matching. It may however exhibit sensitivity to alterations in scale and to the existence of outliers or heavy-tailed distributions. Also, in practice distance metrics are often derived from a `standardized' representation of the data, and the selection of a metric is an important factor in causal inference because different metrics can lead to different conclusions (\cite{rosenbaum2010design}, chapter 9). Therefore, in two influential contributions, \cite{rosenbaum2005exact,rosenbaum2010design} advocated for using the distances between component-wise ranks, instead of the original data, to measure covariate similarity when constructing matching estimators of average treatment effects. This approach is called Rosenbaum’s rank-based matching estimator for ATE.

Compared to the covariate-based matching ATE estimator, Rosenbaum’s rank-based matching estimator is obtained by replacing the original values of the $X_i$'s with their component-wise ranks when performing nearest neighbor matching. The detailed construction is as follows.
    \vspace{0.1in}
    
\textbf{Step 1.} Write $X_i=(X_{i,1},\ldots,X_{i,m})^{T}$ for $i\in [n]$. Define the vector of the marginal empirical cumulative distribution functions, $\widehat{\mbf{F}}_{n}:\mbb{R}^{m}\rightarrow [0,1]^{m}$, as follows: for any $x=(x_1,\ldots,x_m)^{T}\in\mbb{R}^{m}$,
\begin{align*}
    \widehat{\mbf{F}}_{n}(x):=(\widehat{F}_{n,1}(x_1),\ldots,\widehat{F}_{n,m}(x_m))^{T},\ \text{with}\ \widehat{F}_{n,k}(x_k):=\frac{1}{n}\sum_{i=1}^{n}\mathds{1}(X_{i,k}\le x_k), k\in[m].
\end{align*}

For each $i\in[n]$, define $\hat{L}_{i}:=\widehat{\mbf{F}}_{n}(X_i)$ and note that for $k \in [m]$, the $k$-th component of $n\hat{L}_{i}$ is the corresponding rank of $X_{i,k}$ among $\{X_{j,k}\}_{j=1}^{n}$ (with ties broken arbitrarily). Also, let $\mbf{F}:\mbb{R}^{m}\rightarrow [0,1]^{m}$ be the vector of marginal population cumulative distribution functions, i.e. for any $x=(x_1,\ldots,x_m)^{T}\in\mbb{R}^{m}$,
\begin{align*}
    \mbf{F}(x):=(F_{1}(x_1),\ldots,F_{m}(x_m))^{T},\ \text{with}\ F_{k}(x_k):=\mbb{P}(X_{1,k}\le x_k), k\in[m].
\end{align*}
Write $L=\mbf{F}(X)$ with $L_i:=\mbf{F}(X_i)$, for $i\in[n]$.
    \vspace{0.1in}
    
\textbf{Step 2.} Similar to the covariate-based matching, regression adjustment is employed to correct the bias. Let $\hat{\mu}_{r,0}(\cdot)$ and $\hat{\mu}_{r,1}(\cdot)$ be estimators of the conditional means
\begin{align*}
    \mu_{r,0}(l):=\mbb{E}(Y|L=l,D=0),\quad \mu_{r,1}(l):=\mbb{E}(Y|L=l,D=1),
\end{align*}
respectively.
    \vspace{0.1in}
    
\textbf{Step 3.} The rank-based ATE estimator $\hat{\tau}_{r,M}^{bc}$ now is constructed by applying bias-correction and matching to $\{(\hat{L}_i,D_i,Y_i)\}_{i=1}^{n}$: Let $\mcal{J}_{r,M}(i)$ denote the index set of $M$-NNs of $\hat{L}_{i}$ in $\{\hat{L}_{j}:D_{j}=1-D_{i}\}_{j=1}^{n}$ with ties broken in some arbitrary way. The rank-based ATE estimator $\hat{\tau}_{r,M}^{bc}$ is then defined as
\begin{align}\label{rankATEest}
    \hat{\tau}_{r,M}^{bc}:=\frac{1}{n}\sum_{i=1}^{n}(\hat{Y}_i(1)-\hat{Y}_i(0)),
\end{align}
where, for $\omega\in\{0,1\}$,
\begin{align*}
    \hat{Y}_i(\omega):=\left\{
    \begin{aligned}
        &\frac{1}{M}\sum_{j\in\mcal{J}_{r,M}(i)}(Y_j+\hat{\mu}_{\omega}(\hat{L}_{i})-\hat{\mu}_{\omega}(\hat{L}_{j})),\ \text{if}\ D_{i}=1-\omega;\\
        &Y_{i},\ \qquad\ \qquad\qquad\quad\ \qquad\qquad\quad \quad\text{if}\ D_i=\omega.
    \end{aligned}
    \right.
\end{align*}

This can be further generalized by considering a functional transform of the data, as considered in \cite{cattaneo2023rosenbaum}. For the sake of completeness, we also explain this general case below. For $\omega\in\{0,1\}$, consider functions $\phi_{\omega}:\mbb{X}\rightarrow\mbb{X}_{\phi}$ with ${X}_{\phi}\subseteq \mbb{R}^{m'}$ for some $m'\in \mbb{N}$. Note here that $m'$ can indeed be different from $m$. Then, for possibly unknown $\phi_{\omega}$, let $\hat{\phi}_{0}$ and $\hat{\phi}_{1}$ be generic estimators based on the sample $\{(X_i,D_i,Y_i)\}_{i=1}^{n},$  and define

\begin{align*}
    L_{\phi,\omega}:=\phi_{\omega}(X),\ \text{and}\ \hat{L}_{\phi,\omega,i}:=\hat{\phi}_{\omega}(X_i),\ i\in[n].
\end{align*}
Note that when $\phi_0=\phi_1=\mbf{F}$ and $\hat{\phi}_{0}=\hat{\phi}_{1}=\hat{\mbf{F}}_n$, it recovers $L$ and $\hat{L}_i$ from \textbf{Step 1} above. Let $\mcal{J}_{\phi,M}(i)$ represent the index set of $M$-NN matches of $\hat{L}_{\phi,1-D_i,i}$ in $\{\hat{L}_{\phi,1-D_i,j}:D_j=1-D_i\}_{j=1}^{n}$ with ties broken in an arbitrary way. In other words, for determining the nearest
neighbors, this approach measures the similarity based on the Euclidean distance between transformed data points with the transformation function possibly also needing to be learned from the
same data. Let $K_{\phi}(i)$ stand for the number of matched times for the unit $i$, i.e.,
\begin{align}\label{def:kphii}
    K_{\phi}(i):=\sum_{j=1,D_j=1-D_i}^{n}\mathds{1}(i\in\mcal{J}_{\phi,M}(j)).
\end{align}
Moreover, let $\hat{\mu}_{\phi,\omega}(l)$ be  mappings from $\mbb{X}_{\phi}$ to $\mbb{R}$ that estimate the conditional means
\begin{align*}
    \mu_{\phi,\omega}(l):=\mbb{E}(Y|L_{\phi,\omega}=l,D=\omega)).
\end{align*}
The general $\phi$-transformed rank-based bias-corrected matching estimator $\hat{\tau}_{\phi,M}^{bc}$ is then given by
\begin{align}\label{phirankATEest}
    \hat{\tau}_{\phi,M}^{bc}:=\frac{1}{n}\sum_{i=1}^{n}(\hat{Y}_{\phi,i}(1)-\hat{Y}_{\phi,i}(0)),
\end{align}
where, for $\omega\in\{0,1\}$,
\begin{align*}
    \hat{Y}_{\phi,i}(\omega):=\left\{
    \begin{aligned}
        &\frac{1}{M}\sum_{j\in\mcal{J}_{\phi,M}(i)}(Y_j+\hat{\mu}_{\phi,\omega}(\hat{L}_{\phi,\omega,i})-\hat{\mu}_{\phi,\omega}(\hat{L}_{\phi,\omega,j})),\ \text{if}\ D_{i}=1-\omega;\\
        &Y_{i},\ \qquad\ \qquad\qquad\quad\ \qquad\qquad\qquad \text{if}\ D_i=\omega.
    \end{aligned}
    \right.
\end{align*}
Note again that when $\phi_0=\phi_1=\mbf{F}$ and $\hat{\phi}_{0}=\hat{\phi}_{1}=\hat{\mbf{F}}_n$, \eqref{phirankATEest} indeed recovers \eqref{rankATEest}. In the general setting, we provide a Gaussian approximation bound for $\hat{\tau}_{\phi,M}^{bc}$ in Theorem \ref{thm:rank} and Corollary \ref{coro:cdf} in Section \ref{sec:mainresult}.

\section{Assumptions}
In this paper, we work under the standard assumptions put forward in prior works by \cite{abadie2011bias}, \cite{lin2023estimation} and \cite{cattaneo2023rosenbaum}. We refer the reader to these works for further motivations for these assumptions. We do however make a few minor modifications to the assumptions, as required for our stabilization-based Gaussian approximation techniques.

\subsection{Covariate-based matching}\label{sebsec41}
For the covariate-based matching, we assume the following two sets of conditions.
\subsubsection{\bf Assumption set A}\label{AssumpA}
(Data Distribution)
\begin{itemize}
    \item[(1)] $X$ is supported on a compact, convex set $\mbb{X}\subset\mbb{R}^{m}$.
    \vspace{0.1in}
    \item [(2)] The distribution of $X$ is absolutely continuous w.r.t.\ the Lebesgue measure and its density $g$ is uniformly bounded from below and above, i.e., $0<g_{\min}\le g\le g_{\max}<\infty$ on $\mbb{X}$.
    \vspace{0.1in}
    \item [(3)] For almost all $x\in\mbb{X}$, $D$ is independent of $(Y(0),Y(1))$ conditional on $X=x$, and there exists a constant $\eta\in  (0,1/2]$ such that
    \begin{align*}
        \eta\le \mbb{P}(D=1|X=x)\le 1-\eta.
    \end{align*}
    \item [(4)] Denote by $g_0$ and $g_1$ the conditional densities of $X|D=0$ and $X|D=1$ with supports $S_0$ and $S_1$ (subsets of $\mbb{X}$), respectively. Both, $g_0$ and $g_1$ satisfy a Lipschitz-type condition, namely, for all $x,z\in S_0$ or all $x,z\in S_1$, $(|g_0(z)-g_0(x)|\vee |g_1(z)-g_1(x)|)\le L\|x-z\|$ for some constant $L>0$. They are uniformly bounded from above and below, i.e., for $i=0,1$, $0<g_{i,\min}\le g_{i}\le g_{i,\max}<\infty$ on $S_i$. We define the maximum of the `within' density ratios $r_{\text{ratio}}:=\max_{i=0,1}\left(\frac{g_{i,\min}}{g_{i,\max}}\right)$.
    \vspace{0.1in}
    \item [(5)] There exists a constant $0<a<1$ such that for any $0<\delta\le \text{diam}(S_0)$ and any $z\in S_1$, 
    \begin{align*}
        \lambda(B(z,\delta)\cap S_0)\ge a\lambda(B(z,\delta)),
    \end{align*}
    and for any $0<\delta\le\text{diam}(S_1)$ and for any $z\in S_0$
    \begin{align*}
        \lambda(B(z,\delta)\cap S_1)\ge a\lambda(B(z,\delta)).
    \end{align*}
    \item [(6)] There exists a constant $H>0$ such that the surface area of $S_0$ and $S_1$ is bounded by $H$.
\end{itemize}

\subsubsection{\bf Assumption set B}\label{AssumpB}
(Regression functions)
\begin{itemize}
    \item [(1)] $\mbb{E}\mu_{\omega}^{2}(X)$ is bounded for $\omega\in\{0,1\}$.
    \vspace{0.1in}
    \item[(2)] There exists $M_{l}>0$ such that $\mbb{E} \left[(Y-\mu_{D}(X))^2\right]\ge M_{l}>0$. Moreover, there exist $p>0$ and $0< M_{u,p}<\infty$ such that $\mbb{E}\left[|Y(\omega)-\mu_{\omega}(X)|^{4+p}|(X,D)=(x,\omega) \right]\le M_{u,p}$ for all $(x,\omega)\in\XD{\mbb{X}}$.
    \vspace{0.1in}
    \item [(3)] For $\omega=0,1$, $\mu_{\omega}$ is continuously differentiable up to order $\lfloor m/2\rfloor+1$, where $\lfloor\cdot\rfloor$ denotes the floor function. In particular, this implies that $\max_{t\in\Lambda_{\lfloor m/2\rfloor+1}}\|\partial^{t}\mu_{\omega}\|_{\infty}$ is bounded, where for any positive integer $k$, $\Lambda_{k}$ is the set of all vectors $t=(t_1,\ldots,t_m)\in \mbb{R}^{m}$ with non-negative integer coordinates such that $\sum_{i=1}^{m}t_i=k$.
    \vspace{0.1in}
    \item[(4)] There exists some constant $\epsilon_{\mu}>0$ such that for $\omega\in\{0,1\}$, the estimator $\hat{\mu}_{\omega}$ satisfies
    \begin{align*}
        \mbb{E}\underset{t\in\Lambda_{\lfloor m/2\rfloor+1}}{\max}\|\partial^{t}\hat{\mu}_{\omega}\|^{2}_{\infty}=O(1),\ \text{and}\ \mbb{E}\underset{l\in[\lfloor m/2\rfloor]}{\max}\underset{t\in\Lambda_{l}}{\max}\|\partial^{t}\mu_{\omega}-\partial^{t}\hat{\mu}_{\omega}\|^{2}_{\infty}=O(n^{-2\gamma_{l}}),
    \end{align*}
    with some constant $\gamma_{l}>\frac{1}{2}-\frac{l}{m}+\epsilon_{\mu}$ for $l=1,2,\ldots,\lfloor m/2\rfloor.$
\end{itemize}

\begin{Remark}
    Compared to \citet{lin2023estimation}, we do not require their assumptions 4.4 (i) and (ii) that $ \mbb{E} [(Y(\omega)-\mu_{\omega}(X))^2 | X=x]=\mbb{E} [\bm{\varepsilon}^2 | \XD{X}=(x,\omega)]$ is uniformly bounded away from zero for almost all $(x,\omega)\in \XD{\X}$, and that the $(2+\kappa)$-th conditional moments of the errors are uniformly bounded. The first assumption is needed in \citet{lin2023estimation} to invoke the Lindeberg-Feller central limit theorem, which they use for the asymptotic normality result. 
    
    The use of `stabilization' approach (see Section \ref{sec:stabilization}) to derive a non-asymptotic bound of Gaussian approximation however necessitates Assumption B.2. Note that we require a bounded $(4+p)$-th moment for our bound. Our assumption could potentially be relaxed to bounded $(2+p)$-th moment by using more sophisticated techniques as employed in~\cite{trauthwein2022quantitative}. 
\end{Remark}

\subsection{Rank-based matching} To accommodate the changes in the rank-based matching, some of the assumptions above need appropriate adjustments for the $\phi$-transformation; see \cite{cattaneo2023rosenbaum} for more details.

\subsubsection{\bf Assumption set C}\label{AssumpC}
(Data Distribution)
    \begin{itemize}
    \item [(1)] The image of $\phi_{\omega}$, i.e., $\mbb{X}_{\phi}\subset \mbb{R}^{m'},$ is a compact and convex set with bounded surface area.
    \vspace{0.1in}
    \item [(2)] The densities of $\phi_{\omega}(X), \omega \in \{0,1\},$ are continuous and uniformly bounded from above and below over $\mbb{X}_{\phi}.$ 
    \vspace{0.1in}
    \item[(3)] Same as A.3.
    \vspace{0.1in}
    \item [(4)] $\mbb{P}(D=1|\phi_{\omega}(X)=\phi_{\omega}(x))=\mbb{P}(D=1|X=x)$ for almost all $x\in\X$ and any $\omega\in\{0,1\}$.
    \vspace{0.1in}
    \item [(5)] For $\omega \in \{0,1\}$, let $g_{\phi,\omega,0}$ be the conditional densities of $\phi_{\omega}(X)|D=0$, both with support $S_{\phi,0},$ and similarly, let $g_{\phi,\omega,1}$ be the conditional densities of $\phi_{\omega}(X)|D=1$ with supports $S_{\phi,1}.$
 These conditional densities satisfy a Lipschitz-type condition, namely, for all $x,z\in S_{\phi,0}$ or all $x,z\in S_{\phi,1}$, $(|g_{\phi,\omega,0}(z)-g_{\phi,\omega,0}(x)|\vee |g_{\phi,\omega,1}(z)-g_{\phi,\omega,1}(x)|)\le L_{\phi,\omega}\|x-z\|$ for some constant $L_{\phi,\omega}>0$. They are also uniformly bounded from above and below, i.e., for $i=0,1$, $0<g_{\phi,\omega,i,\min}\le g_{\phi,\omega,i}\le g_{\phi,\omega,i,\max}<\infty$ on $S_{\phi,i}$. We define the maximum of `within' density ratios $r_{\text{ratio},\phi}:=\max_{i=0,1}\max_{\omega=0,1}\left(\frac{g_{\phi,\omega,i,\min}}{g_{\phi,\omega,i,\max}}\right)$.
    \vspace{0.1in}
    \item [(6)] There exists a constant $0<a_{\phi}<1$ such that for any $0<\delta\le \text{diam}(S_{\phi,0})$ and any $z\in S_{\phi,1}$, 
    \begin{align*}
        \lambda(B(z,\delta)\cap S_{\phi,0})\ge a_{\phi}\lambda(B(z,\delta)),
    \end{align*}
    and for any $0<\delta\le\text{diam}(S_{\phi,1})$ and for any $z\in S_{\phi,0}$
    \begin{align*}
        \lambda(B(z,\delta)\cap S_{\phi,1})\ge a_{\phi}\lambda(B(z,\delta)).
    \end{align*}
    \item [(7)] Both the surface areas of $S_{\phi,0}$ and $S_{\phi,1}$ are bounded by a constant $H_{\phi} > 0$.
\end{itemize}

\subsubsection{\bf Assumption set D}\label{AssumpD}
(Regression functions)
\begin{itemize}
    \item [(1)] $\mbb{E}\mu_{\phi,\omega}^{2}(L_{\phi,\omega})$ is bounded for $\omega\in\{0,1\}$.
    \vspace{0.1in}
    \item[(2)] There exists $M_{l,\phi}>0$ such that $\mbb{E} [Y-\mu_{\phi,D}(L_{\phi,D})]^2\ge M_{l,\phi}>0$. Moreover, there exist $p>0$ and $0< M_{u,\phi,p}<\infty$ such that $ \mbb{E}\left[|Y(\omega)-\mu_{\phi,\omega}(L_{\phi,\omega})|^{4+p}|(X,D) \right]=(x,\omega))\le M_{u,\phi,p}$ for all $(x,\omega)\in\XD{\mbb{X}}$.
    \vspace{0.1in}
    \item [(3)] For $\omega=0,1$, $\mu_{\phi,\omega}$ is continuously differentiable up to order $\lfloor m'/2\rfloor+1$. In particular, $\max_{t\in\Lambda_{\lfloor m'/2 \rfloor\vee 1 +1}}\|\partial^{t}\mu_{\phi,\omega}\|_{\infty}$ is bounded, where for any positive integer $k$, $\Lambda_{k}$ is the set of all vectors $t=(t_1,\ldots,t_m')\in \mbb{R}^{m'}$ with non-negative integer coordinates such that $\sum_{i=1}^{m'}t_i=k$.
    \vspace{0.1in}
    \item[(4)] For $\omega\in\{0,1\}$, the estimator $\hat{\mu}_{\phi,\omega}$ satisfies:
    \begin{align*}
        \mbb{E}\underset{t\in\Lambda_{\lfloor m'/2 \rfloor\vee 1 +1}}{\max}\|\partial^{t}\hat{\mu}_{\phi,\omega}\|^{2}_{\infty}=O(1),\ \text{and}\ \mbb{E}\underset{l\in[\lfloor m'/2\rfloor \vee 1]}{\max}\underset{t\in\Lambda_{l}}{\max}\|\partial^{t}\mu_{\phi,\omega}-\partial^{t}\hat{\mu}_{\phi,\omega}\|^{2}_{\infty}=O(n^{-2\gamma_{\phi,l}}),
    \end{align*}
    with some constant $\gamma_{\phi,l}>\left(\frac{1}{2}-\frac{l}{m'}\right)\vee 0$ for $l=1,2,\ldots,\lfloor m'/2 \rfloor\vee 1.$
\end{itemize}

Note again that similar to the case of covariate bases matching, compared to the assumptions in \citet{cattaneo2023rosenbaum}, we do not require their assumption 5.5 (i) that $ \mbb{E} [(Y(\omega)-\mu_{\phi,\omega}(U_{\phi,\omega}))^2 | U_{\phi,\omega}=u]$ is uniformly bounded away from zero for almost all $(u,\omega)$, which is needed there to apply the Lindeberg-Feller CLT, as well as the uniform boundedness of the $(2+\kappa)$-th moment in their assumption 5.5 (ii). 

\section{Main results}\label{sec:mainresult}
We now present our main results on Gaussian approximation bounds for the matching-based and rank-based ATE estimators. All our bounds are stated for the Kolmogorov distance $\mathsf{d}_{K}(\cdot,\cdot)$, which, for two real-valued random variables $Z_1,Z_2$, is defined as
\begin{align*}
    \mathsf{d}_{K}(Z_1,Z_2):=\sup_{t\in\mbb{R}}~|\mbb{P}(Z_1\le t)-\mbb{P}(Z_2\le t)|.
\end{align*}
For $r_{0}:=r_{\text{ratio}}\vee r_{\text{ratio},\phi}$ (with $r_{\text{ratio}}$ and $r_{\text{ratio},\phi}$ from Assumptions A.3 and C.4, respectively), define the quantities
\begin{align}\label{delta211}
\begin{aligned}
    \delta_{H_1}&:=\frac{1}{n^2\eta^4}+\left(\frac{n}{M\eta}\right)^2\left(e^{-(1-\log 2)M}+e^{M-r_0n\eta-M\log M+M\log (r_0n\eta)}\right)^2,\\
    \delta_{H_2}&:=\left(\frac{M}{n\eta}\right)^{1/m}+\frac{1}{n\eta}+\frac{n}{M}\left(e^{-(1-\log 2)M}+e^{M-r_0n\eta-M\log M+M\log (r_0n\eta)}\right),\\
    \delta_{H_3}&:=\left(\frac{n}{M\eta}\right)^2 e^{-(1-\log 2)M}.
\end{aligned}
\end{align}
For more details about the above terms, see the discussion following Theorem~\ref{thm:covariate}. All our results hold as long as $n \geq 9$. This is due to the fact that a certain tail bound on the radius of stabilization, required for the stabilization technique, holds as long as $n \geq 9$; see Section~\ref{sec:radius} and Lemma \ref{lem:tb} for additional details. 

\subsection{Rates for covariate-based ATE}
\begin{Theorem}[Gaussian approximation bound for $\hat{\tau}_{M}^{bc}$]\label{thm:covariate}
    Let Assumptions A and B in Sections \ref{AssumpA} and \ref{AssumpB} hold with $n\ge 9$, $M\in [n]$ and $\eta \in (0,1/2]$ such that, for constants $C_0,C_1 > 0$, $M\le C_0n\eta$  
    and $n\eta^2 \ge C_1$, where  $C_0 \le \max\limits_{i=1,2}\big(\frac{g_{i,\min}}{4}\big)^{m+1} \frac{V_m}{2L^m}$. 
    Then, for any $p\in (0,1]$ and $m\in\mbb{N}$, there exists a finite constant $C>0$ not depending on $n,M,\eta$ or $p$ such that
    \begin{align*}
        \mathsf{d}_{K}\left(\sqrt{n}(\hat{\tau}_{M}^{bc}-\tau),\mcal{N}(0,\sigma^2)\right)\le C(B_1+B_2+B_3),
    \end{align*}
    where
    \begin{align*}
        B_1&:=\frac{\alpha^{-1} \big((\frac{M}{\zeta\eta})^{\frac{20}{8+p}}\vee 1\big)\cdot \big((\frac{M}{\eta})^{\frac{16+3p}{16+2p}}\vee 1\big)}{n^{\frac{1}{2}}}+ \frac{(\frac{M}{\zeta\eta})^{\frac{40}{8+p}}\vee 1}{n^{\frac{1}{2}}},\\
    B_2&:=(\eta^{-k/(2m)}+\delta_{H_1}^{1/2})\left(M^{\frac{k}{2m}}n^{-\frac{k}{2m}+\frac{1}{4}}+\max_{l\in [k-1]}\left(n^{-\frac{\gamma_{l}}{2}-\frac{l}{2m}+\frac{1}{4}}M^{\frac{l}{2m}}\right)\right),\\
    B_3&:=\frac{1}{\eta}\left(\frac{M}{n\eta}\right)^{1/(2m)}+\delta_{H_1}^{1/2}+(\delta_{H_2}^{1/2}+1)\cdot\frac{1}{\eta M^{1/2}}+\delta_{H_3}^{1/2}+\frac{1}{\eta^{3}n^{1/3}},
    \end{align*}
    with $\alpha=p/(16+2p)$, $\zeta=p/(40+10p)$, $k=\lfloor m/2 \rfloor+1$, $\delta_{H_1}$-\,$\delta_{H_3}$ as in \eqref{delta211}, and $\gamma_l$ as defined in Assumption B.4 in Section \ref{AssumpB}. The limiting variance $\sigma^2$ is given by
    \begin{align}\label{def:sigma}
    \sigma^2:=\Var (\mu_1(X)-\mu_0(X))+\mbb{E}\left(\frac{\sigma_1^2(X)}{e(X)}+\frac{\sigma^2_0(X)}{1-e(X)}\right)>0,\ (\text{by Assumption B.2})
    \end{align}
    where $e(x):=\mbb{P}(D=1|X=x)$ and $\sigma^2_\omega(x):= \mbb{E} [(Y(\omega)-\mu_{\omega}(X))^2 | X=x]$.
\end{Theorem} 
The above Gaussian approximation bound consists of three parts. The term $B_1$ corresponds to the Gaussian approximation bound for $E_n$ in \eqref{eq2:F}, centered at the true ATE and scaled by the sample variance. Similar to the classical Berry-Essen Theorem (see, for example, \eqref{BE}), the polynomials involving $M$ in the numerators are from Assumption B.2 in Section \ref{AssumpB}, and the denominator $n^{1/2}$ in $B_1$ corresponds to a variance lower bound.

The term $B_2$ arises from the bias correction for $(B_M-\hat{B}_M)$. This can be further improved by assuming existence of higher order moments in Assumption B.4 as in Section \ref{AssumpB} instead of just $L^2$ moments. To ensure that the bias term $B_2\rightarrow 0$, one could pick $M\lesssim n^{\iota}$ with
    \begin{align*}
        \iota:=\min_{l\in\lfloor m/2\rfloor}\left\{\left[1-\left(\frac{1}{2}-\gamma_{l}+\epsilon_{\mu}\right)\frac{m}{l}\right]\wedge \left[1-\left(\frac{1}{2}+\epsilon_{\mu}\right)\frac{m}{\lfloor m/2\rfloor+1}\right]\right\},
    \end{align*}
    where $\epsilon_{\mu}$ and $\gamma_{l}$'s are defined in Assumption B.4 in Section \ref{AssumpB}.

The term $B_3$ relies on the convergence rate of the sample variance to its limiting variance $\sigma^2$. For the sake of generality, we keep track of the data balance parameter $\eta$ in Assumption A.3 in \ref{AssumpA}, and have made minimal assumptions on the relationship between $\eta$ and $n$ (or $M$); see also Lemma \ref{Variancelowerbound}. It is obtained via a modified and rigorous non-asymptotic convergence rate argument for the density ratio estimation in contrast to the crude asymptotic arguments in \cite{lin2023estimation}; see Lemma~\ref{lemma:l2convdensityratio}. This also necessitates the addition of the non-asymptotic error terms $\delta_{H_1}$-$\delta_{H_3}$. Moreover, the assumptions that there exists positive constants $C_0,C_1$ such that $M\le C_0n\eta$ and $C_1\le n\eta^2$ are mild in the sense that they are also required for  $B_3$ and $\delta_{H_1}$ to tend to zero.
 
The following corollary to Theorem \ref{thm:covariate} provides a user-friendly bound under some additional mild assumptions that ensure that all the error terms involving $\delta_{H_1}$-$\delta_{H_3}$ are negligible compared to the rest of the summands. 

\begin{Corollary}\label{coro:covariate}
    Let the assumptions of Theorem \ref{thm:covariate} prevail. Additionally, assume that $M^{-1}\log n=o(1), n^{-1}M\log n=o(1)$ and that $\eta$ is bounded away from $0$. Then, for any $p\in (0,1]$ and $m\in\mbb{N}$, there exists a finite constant $C>0$ not depending on $n,M,\eta$ or $p$ such that 
    \begin{align*}
        {\sf d}_{K}\left(\sqrt{n}(\hat{\tau}_{M}^{bc}-\tau),\mcal{N}(0,\sigma^2)\right)\le C(B_1'+B_2'+B_3'),
    \end{align*}
    where
    \begin{align*}
        B_1'&:=M^{\frac{40}{8+p}}n^{-\frac{1}{2}},\\
        B_2'&:=M^{\frac{k}{2m}}n^{-\frac{k}{2m}+\frac{1}{4}}+\max_{l\in [k-1]}\left(n^{-\frac{\gamma_{l}}{2}-\frac{l}{2m}+\frac{1}{4}}M^{\frac{l}{2m}}\right),\\
        B_3'&:=\left(\frac{M}{n}\right)^{1/(2m)}+\frac{1}{M^{1/2}}+\frac{1}{n^{1/3}}.
    \end{align*}
\end{Corollary}
A further simplified bound is provided in~\eqref{dim1} later, where we also compare covariate and rank-based ATE estimators in the univariate setting. We now make some remarks regarding Theorem \ref{thm:covariate} and Corollary \ref{coro:covariate}.

\begin{Remark}[Dependence on $M$]
    As for the dependence on $M$ in $B_1$, we show later in Section \ref{sec:ATEstabilization} that it is derived from viewing the ATE estimator as a sum of certain score functions whose dependencies are restrained within a ball. In particular, as $M$ increases, the radius of the ball becomes large resulting in increased dependency between the scores, deviating further away from an i.i.d.\ setup, which negatively affects the Gaussian convergence. This is due to the nature of the stabilization techniques, which also appear in many other nearest neighbor based estimators (see for instance the weighted entropy estimation in \cite{shi2024flexible} and the random forest estimation in \cite{shi2024multivariate}). 
\end{Remark}

\begin{Remark}[The balance of data]
    The parameter $\eta$ in Assumption A.3 in Section \ref{AssumpA} controls the balance of the data (i.e. the number of the treated and controlled individuals) ensuring sufficient individuals in both groups. According to $B_1$ (for fixed $p$), when $\eta=o(M)$, the bound tends to infinity. This regulates the choice of $M$ when the data is imbalanced. A phenomenon also occurs for $B_3$, where the choices of both $M$ and $n$ must be adjusted according to $\eta$.
\end{Remark}

The doubly robust estimator of ATE considered in \cite{lin2023estimation} actually uses a $K$-fold random partition of the data and averages the estimation on each subset to output a final estimator. We emphasize here that the stabilization technique could also be applied in a similar way as for the bound $B_1$, since both of these estimators use nearest neighbor matching. Carrying out this exercise is left as a future work. Furthermore, although both Theorem \ref{thm:covariate} and Corollary \ref{coro:covariate} are stated in the context of the Euclidean setting, the stabilization technique introduced later in Section \ref{sec:stabilization} used to obtain $B_1$ is valid for general metric spaces. Particularly, it can also be applied in the $m$-dimensional manifold setting, for example, as in \cite{penrose2013limit}.

\subsection{Rates for rank-based ATE}
\begin{Theorem}[Gaussian approximation bound for $\hat{\tau}_{\phi,M}^{bc}$]\label{thm:rank}
     Let Assumptions C and D in Section \ref{AssumpC} and \ref{AssumpD} hold with $n\ge 9$, $M\in [n]$ and $\eta \in (0,1/2]$ such that, for constants $C_0,C_1 > 0$, $M\le C_0n\eta$, 
     and $n\eta^2 \ge C_1 $ where $C_0 \le \max\limits_{i=1,2;\,\omega = 1,2}\big(\frac{g_{\phi,\omega,i,\min}}{4}\big)^{m'+1} \frac{V_{m'}}{2L_{\phi,\omega}^{m'}}.$  
     Then for any $p\in (0,1],m,m'\in \mbb{N}$, there exists a finite constant $C>0$ not depending on $n,M,\eta$ or $p$ such that
     \begin{align*}
        {\sf d}_{K}\left(\sqrt{n}(\hat{\tau}_{\phi,M}^{bc}-\tau),\mcal{N}(0,\sigma_{\phi}^2)\right)\le C(B_4+B_5+B_6),
    \end{align*}
    where
    \begin{align*}
        B_4&:=\frac{\alpha^{-1} \big((\frac{M}{\zeta\eta})^{\frac{20}{8+p}}\vee 1\big)\cdot \big((\frac{M}{\eta})^{\frac{16+3p}{16+2p}}\vee 1\big)}{n^{\frac{1}{2}}}+ \frac{\big((\frac{M}{\zeta\eta})^{\frac{40}{8+p}}\vee 1\big)}{n^{\frac{1}{2}}},\\
    B_5&:=(\eta^{-\frac{k}{2m'}}+\delta_{H_1}^{\frac{1}{2}})\Bigg(M^{k/(2m')}n^{-k/(2m')+1/4}+\max_{l\in [k-1]}\Big(n^{-\gamma_{\phi,l}/2+1/4}\Big(\Big(\frac{M}{n}\Big)^{l/(2m')}+n^{-l/4}\Big)\Big)\\
    &\quad +n^{-k/4+1/4}+n
    ^{1/4}(\sup_{\omega\in\{0,1\}}\lim_{\delta\rightarrow 0}\mbb{E}\sup_{x,y\in\mbb{X},\|\phi_{\omega}(x)-\phi_{\omega}(y)\|\le \delta}\|(\hat{\phi}_{\omega}-\phi_{\omega})(x)-(\hat{\phi}_{\omega}-\phi_{\omega})(y)\|_{\infty})^{1/2}\Bigg),\\
    B_6&:=\frac{1}{\eta}\left(\frac{M}{n\eta}\right)^{1/(2m')}+\delta_{H_1}^{1/2}+(\delta_{H_2}^{1/2}+1)\cdot\frac{1}{\eta M^{1/2}}+\delta_{H_3}^{1/2}+\frac{1}{\eta^{3}n^{1/3}}\\
    &\quad+\left(\frac{n}{M}\right)^{m/m'}\cdot \left(\frac{n^2}{M^2}\mbb{E}\Big(\sup_{\omega\in\{0,1\}}\sup_{x_1,x_2\in \X}\|\hat{\phi}_{\omega}(\cdot;x_1,x_2)-\phi_{\omega}\|_{\infty}^{2m}\Big)\right)^{1/4},
    \end{align*}
    with $\alpha=p/(16+2p)$, $\zeta=p/(40+10p)$, $k=\lfloor m'/2 \rfloor\vee 1 +1$, $\delta_{H_1}$-\,$\delta_{H_3}$ as in \eqref{delta211}, and $\gamma_{\phi,l}$ as defined in Assumption D.4 in Section \ref{AssumpD}. Here, $\hat{\phi}_{\omega}(\cdot;x_1,x_2)$ stands for the estimator constructed by inserting two new points $x_1,x_2\in \X$ into the point cloud with $D=1-\omega$. The limiting variance $\sigma_{\phi}^2$ is defined as
    \begin{align}\label{eq:sigmaphi}
        \sigma_{\phi}^2:=\Var (\mu_{\phi,1}(L_{\phi,1})-\mu_{\phi,0}(L_{\phi,0}))+\mbb{E}\left(\frac{\sigma_{\phi,1}^2(X)}{e(X)}+\frac{\sigma^2_{\phi,0}(X)}{1-e(X)}\right),
\end{align}
with $e(x):=\mbb{P}(D=1|X=x)$ and $\sigma_{\phi,\omega}(x)^2:= \mbb{E} [(Y(\omega)-\mu_{\phi,\omega}(L_{\phi,\omega}))^2 | X=x]$. 
\end{Theorem}

Similar to Theorem \ref{thm:covariate}, we aim to keep track of $\eta$, which measures the balance of the data. Consequently, in practice one should also be careful while picking $M$ and $n$ in relation to the speed of decay of $\eta$ to $0$. Different from Theorem~\ref{thm:covariate}, the Gaussian approximation bound for the $\phi$-transformation based estimator $\hat{\tau}_{\phi,M}^{bc}$ depends on the convergence of the estimator for $\phi_{\omega}$, that is $\hat{\phi}_{\omega}$, as appearing in $B_5$ and $B_6$. Comparing to the asymptotic normality result in \cite{cattaneo2023rosenbaum}, for $B_5$, they assumed a Donsker-type condition: for any $\omega\in\{0,1\}$ and any $\epsilon_e>0$,
\begin{align*}
    \lim_{\delta\rightarrow 0}\limsup_{n\rightarrow\infty}~\mbb{P}\Big(\sqrt{n}\sup_{x,y\in \X:\|\phi_{\omega}(x)-\phi_{\omega}(y)\|\le \delta}\|(\hat{\phi}_{\omega}-\phi_{\omega})(x)-(\hat{\phi}_{\omega}-\phi_{\omega})(y)\|\ge \epsilon_{e}\Big)=0,
\end{align*}
which is an asymptotic property. In our case, such a term directly appears in our non-asymptotic bound. Moreover, for $B_6$, \citet[Assumption 5.9]{cattaneo2023rosenbaum} requires that
\begin{align*}
    \lim_{n\rightarrow\infty}\frac{n^2}{M^2}\mbb{E}\Big(\sup_{\omega\in\{0,1\}}\sup_{x_1,x_2\in \X}\|\hat{\phi}_{\omega}(\cdot;x_1,x_2)-\phi_{\omega}\|_{\infty}^{2m}\Big)=0,
\end{align*}
which does not offer any rates of convergence. We again have this difference featuring in our bound instead.

From Theorem~\ref{thm:rank}, it is also seen that the Gaussian approximation bound depends on the choice of the transformation $\phi$ (including the dimension of the embedded space, $m'$) and its estimation $\hat{\phi}$. This selection is crucial as different transformations can lead to different conclusions for ATE estimation. We refer interested readers to (\cite{rosenbaum2010design}, chapter 9) and \cite{rosenbaum2005exact,rosenbaum2010design} for more details on the influence of choices for the transformation $\phi$. This general issue is beyond the scope of the current paper which focuses on Gaussian approximation bounds. However, focusing on the particular choice of $\phi$-transformation as the cumulative distribution function (CDF), we present the corresponding simplified result in Corollary \ref{coro:cdf}. For the sake of simplicity, we make further mild assumptions that $M^{-1}\log n=o(1), n^{-1}M\log n=o(1)$ and $\eta$ bounded away from $0$.

\begin{Corollary}[Gaussian approximation bound for CDF-rank-based estimator $\hat{\tau}_{r,M}^{\text{bc}}$]\label{coro:cdf}
    Let the assumptions in Theorem \ref{thm:rank} prevail. Assume in addition that $M^{-1}\log n=o(1), n^{-1}M\log n=o(1)$ and that $\eta$ is bounded away from $0$. If $\phi_0=\phi_1=\mbf{F}$ and $\hat{\phi}_0=\hat{\phi}_1=\hat{\mbf{F}}_n$ in Section \ref{sec:rank}, then for the corresponding rank-based ATE estimator $\hat{\tau}_{r,M}^{\text{bc}}$ in \eqref{rankATEest} and any $p\in (0,1],m\in \mbb{N}$, there exists a finite constant $C>0$ not depending on $n,M,\eta$ or $p$ such that
    \begin{align*}
        {\sf d}_{K}\left(\sqrt{n}(\hat{\tau}_{r,M}^{bc}-\tau),\mcal{N}(0,\sigma_{\phi}^2)\right)\le C(B_4'+B_5'+B_6'),
    \end{align*}
    where
    \begin{align*}
        B_4'&:=M^{\frac{40}{8+p}}n^{-\frac{1}{2}},\\
    B_5'&:=M^{k/(2m)}n^{-k/(2m)+1/4}+\max_{l\in [k-1]}\Big(n^{-\gamma_{\phi,l}/2+1/4}\Big(\Big(\frac{M}{n}\Big)^{l/(2m)}+n^{-l/4}\Big)\Big)+n
    ^{-1/4},
    \end{align*}
    and given $B_4'\le C_2$ for some $C_2>0$,
    \begin{align*}
    B_6'&:=\left(\frac{M}{n}\right)^{1/(2m)}+\frac{1}{M^{1/2}}+\left\{
        \begin{aligned}
            &\frac{1}{M^{1/4}},\ \qquad\qquad\  m=1,\\
            &\left(\frac{1}{Mn}\right)^{1/6},\ \quad\ \ m= 2,\\
            &M^{-3/2}n^{(-m+3)/2},\ m\ge 3,
        \end{aligned}
    \right.
    \end{align*}
    with $\alpha=p/(16+2p)$, $\zeta=p/(40+10p)$, $k=\lfloor m/2 \rfloor\vee 1 +1$, and $\gamma_{\phi,l}$ as defined in Assumption D.4 in Section \ref{AssumpD}. The limiting variance $\sigma_{\phi}^2$ is defined by plugging $\phi=\mbf{F}$ in \eqref{eq:sigmaphi}.
\end{Corollary}

It is instructive to compare the rates with covariate and CDF-base rank ATE. From Corollaries \ref{coro:covariate} and \ref{coro:cdf}, we note that $B_1'=B_4'$ and the additional term $n^{-1/4}$ in $B_5'$ is dominated by $n^{-1/6}$ in $B_6'$, thus not affecting the overall rate. Then, the terms that actually affect the overall rate are $B_3'$ and $B_6'$. When $m\geq 3$, the last term in $B_6'$ decays fast leading to the same overall rate with or without a CDF transformation. On the contrary, for example, when $m=1$, the last summand in $B_6'$ (which equals $M^{-1/6}$) dominates $B_5'$ as well as every other summand in $B_6'$. Hence, we have (in the case of CDF-based rank ATE and $m=1$)
\begin{align*}
    {\sf d}_{K}\left(\sqrt{n}(\hat{\tau}_{r,M}^{bc}-\tau),\mcal{N}(0,\sigma_{\phi}^2)\right)\lesssim M^{\frac{40}{8+p}}n^{-\frac{1}{2}}+M^{-\frac{1}{4}}.
\end{align*}
On the other hand, for the case without a CDF transformation and $m=1$ (i.e., covariate-based rank ATE), it holds that
\begin{align}\label{dim1}
    {\sf d}_{K}\left(\sqrt{n}(\hat{\tau}_{M}^{bc}-\tau),\mcal{N}(0,\sigma^2)\right)\lesssim M^{\frac{40}{8+p}}n^{-\frac{1}{2}}+M^{-\frac{1}{2}}.
\end{align}
We can then see that when $M^{\frac{40}{8+p}}n^{-\frac{1}{2}}\lesssim M^{-1/2}$, we expect a strictly worse rate with a CDF transformation.

\section{Bootstrap Approximation Bounds}\label{sec:boostrapresults} 

In the context of matching based estimators, the asymptotic normality results from~\cite{abadie2006large,lin2023estimation} could be used to construct confidence intervals for the ATE parameter $\tau$. Specifically, one claims that $\hat{\tau}_{M}^{\text{bc}}\pm z_{1-\frac{\alpha}{2}}\frac{\sigma}{\sqrt{n}}$, where $z_{1-\alpha/2}$ is the $(1-\alpha/2)$-quantile of the standard normal distribution, provides a $1-\alpha$ confidence interval in this context. However, there are two main shortcomings of such a claim: (i) The validity of the obtained confidence intervals holds only asymptotically as $n\to\infty$, and (ii) The limiting standard deviation $\sigma$ has to be consistently estimated (see \cite{abadie2006large}).

By the definition of the Kolmogorov metric, a direct application of Theorem \ref{thm:covariate} (or Corollary \ref{coro:covariate}) yields that, for any $0<\alpha<1$,
\begin{align*}
    \mbb{P}\Big(\tau\in \Big(\hat{\tau}_{M}^{\text{bc}}-z_{1-\frac{\alpha}{2}}\frac{\sigma}{\sqrt{n}},\hat{\tau}_{M}^{\text{bc}}+z_{1-\frac{\alpha}{2}}\frac{\sigma}{\sqrt{n}}\Big)\Big)\ge 1-\alpha-C(B_1+B_2+B_3).
\end{align*}
While this avoids the shortcoming in point (i) above, the unknown $\sigma$ makes it still impractical. To overcome both shortcomings (i) and (ii) simultaneously, we now provide an application of the Gaussian approximation results developed in the previous sections, for obtaining confidence intervals that are valid in a non-asymptotic sense. 

Bootstrap serves as one of the most important inferential techniques for non-parametric statistical analysis.   However, \cite{abadie2008failure} provided an example showing that the naive bootstrap (i.e., resampling from the empirical distribution of the observations) fails to provide an asymptotically valid standard error and quantiles for a matching-based ATE estimators. In addition, they argue that the main reason for this failure is that the naive bootstrap fails to reproduce the distribution of the number of matched times $K^{\omega}_{M}(i,\xd{\nu}_n)$ in \eqref{matchedtimes}, with fixed $M$.

Later, \cite{otsu2017bootstrap} proposed an alternative bootstrap called the weighted bootstrap to overcome this difficulty and showed validity of their procedure. They only considered the setting with fixed $M$, i.e., when the number of the nearest neighbors stays fixed. Other methodological studies include works by~\cite{abadie2022robust},~\cite{walsh2023nearest} and~\cite{kosko2024fast}. In the following, we adopt a multiplier bootstrap (or wild bootstrap) to not only bootstrap the distribution of the statistic $\sqrt{n}(\hat{\tau}_{M}^{bc}-\tau)$, but also provide a bound for the approximation accuracy in terms of the Kolmogorov distance, while allowing $M$ to diverge with $n$.


We now describe the details of the multiplier bootstrap procedure and present the approximation bounds. Recall the definition of the bias-corrected ATE estimator in \eqref{eq2:F}, given by
\begin{align*}
    \hat{\tau}_{M}^{\text{bc}}&:=\frac{1}{n}\sum_{i=1}^{n}(\hat{\mu}_{1}(X_i)-\hat{\mu}_{0}(X_i))+\frac{1}{n}\sum_{i=1}^{n}(2D_i-1)\Big(1+\frac{K^{D_i}_{M}(i,\XD{\cX}_n)}{M}\Big)(Y_i-\hat{\mu}_{D_i}(X_i)).
\end{align*}

Below, for notational ease, we denote
$$\Delta\hat{\mu}(X_i):=\hat{\mu}_1(X_i)-\hat{\mu}_0(X_i),\; i\in[n]\quad \text{ and }\quad \widebar{\Delta\hat{\mu}}:=\frac{1}{n}\sum_{i=1}^{n}(\hat{\mu}_1(X_i)-\hat{\mu}_0(X_i)).
$$
The multiplier bootstrap is constructed via the following steps:
\begin{itemize}
    \item Start with two sequences of i.i.d.\ random variables $\{V_i\}_{i=1}^{n}$ and $\{W_{i}\}_{i=1}^{n}$ following $\mcal{N}(0,1)$ and $\mcal{N}(1,1)$ distributions, respectively, as multipliers; these two sequences are also independent of each other.
    \vspace{0.1in}
    \item Based on the sample $\{(X_i,Y_i,D_i)\}_{i=1}^{n}$, compute the residuals
    $$\hat{R}_{i}:=Y_i-\hat{\mu}_{D_i}(X_i),\ i\in[n].$$
    \item For each $i\in [n]$, we obtain the bootstrap sample $\{(X_i,Y_{i}^{*},D_i)\}_{i=1}^{n}$ according to 
    \begin{align*}
        Y_{i}^{*}=\hat{\mu}_{D_i}(X_i)+W_i\hat{R}_{i}.
    \end{align*}
    \item Plugging in the bootstrap sample $\{(X_i,Y_{i}^{*},D_i)\}_{i=1}^{n}$ with the multipliers $\{W_i\}_{i=1}^{n},$ and using the additional multipliers $\{V_i\}_{i=1}^{n}$, the bootstrap estimator $\hat\tau_{M}^{\text{boot}}$ is then given by
    \begin{align*}    \hspace*{1.4cm}\hat \tau_{M}^{\text{boot}}&:=\widebar{\Delta\hat{\mu}}+\frac{1}{n}\sum_{i=1}^{n}(\Delta\hat{\mu}(X_i)-\widebar{\Delta\hat{\mu}})V_i+\frac{1}{n}\sum_{i=1}^{n}(2D_i-1)\Big(1+\frac{K^{D_i}_{M}(i,\XD{\cX}_n)}{M}\Big)(Y_i^{*}-\hat{\mu}_{D_i}(X_i)).
    \end{align*}
\end{itemize}

Turning to the rank-based ATE estimator, note that by definition \eqref{rankATEest}, we can rewrite the estimator in~\eqref{phirankATEest} as
\begin{align*}
    \hat{\tau}_{\phi,M}^{\text{bc}}=\frac{1}{n}\sum_{i=1}^{n}(\hat{\mu}_{\phi,1}(\hat{L}_{\phi,1,i})-\hat{\mu}_{\phi,0}(\hat{L}_{\phi,0,i}))+\frac{1}{n}\sum_{i=1}^{n}(2D_i-1)\left(1+\frac{K_{\phi}(i)}{M}\right)(Y_{i}-\hat{\mu}_{D_i}(\hat{L}_{\phi,D_i,i})).
\end{align*}
By replacing $X_i$ with the $\phi$-transformed sample $\hat{L}_{\phi,D_i,i}$, we can analogously construct the bootstrapped version of the rank-based ATE estimators, which we denote by $\hat\tau_{\phi,M}^{\text{boot}}$. The following result provides rate of convergence for the above multiplier bootstrapping procedures.

\begin{Theorem}\label{thm:boot}
Let
\begin{align*}
\mathsf{E}_1\coloneqq \max_{\omega=0,1}\|\mu_{\omega}-\hat{\mu}_{\omega}\|_{\infty}~~\;\;\text{and}\;\;\;~~\mathsf{E}_2\coloneqq  \max_{\omega=0,1}\|\mu_{\phi,\omega}-\hat{\mu}_{\phi,\omega}\|_{\infty}+\max_{\omega=0,1}\|\phi_{\omega}-\hat{\phi}_{\omega}\|_{\infty}.
\end{align*} 
Further, let
\begin{align*}
    L(\mu,\hat{\mu},n)&:=\bigg(M_{l}-\sqrt{M_{u,p}}\ n^{-1/3}-2\mathsf{E}_1\left(M_{u,p}+(2M_{u,p})^{1/4}n^{-5/12}\right)\bigg)\vee 0,
\end{align*}
and
\begin{align*}
    L(\mu,\hat{\mu},\phi,\hat{\phi},n)&:=\bigg(M_{l,\phi}-\sqrt{M_{u,\phi,p}}\ n^{-1/3}-2\mathsf{E}_2\left(M_{u,\phi,p}+(2M_{u,\phi,p})^{1/4}n^{-5/12}\right)\bigg)\vee 0.
\end{align*}
Under the assumptions of Theorem \ref{thm:covariate} and Theorem \ref{thm:rank}, respectively, there exists a finite constant  $C>0$ not depending on $n,M,\eta$ or $p$, such that
\begin{align*}
    &\mathsf{d}_{K}(\sqrt{n}(\hat\tau^{\text{boot}}_{M}-\hat{\tau}_{M}^{bc})|\XDE{\cX}_n,\sqrt{n}(\hat{\tau}_{M}^{bc}-\tau)) \le C\bigg( B_1+B_2+\frac{(1+\mathsf{E}_1^2) B_3}{L(\mu,\hat{\mu},n)}+\frac{\eta^{-1}\mathsf{E}_1+\eta^{-2}n^{-1/4}}{L(\mu,\hat{\mu},n) }\bigg),
   \shortintertext{and}
   &\mathsf{d}_{K}(\sqrt{n}(\hat \tau^{\text{boot}}_{\phi,M}-\hat{\tau}_{\phi,M}^{bc})|\XDE{\cX}_n,\sqrt{n}(\hat{\tau}_{\phi,M}^{bc}-\tau))\le C\bigg( B_4+B_5+\frac{(1+\mathsf{E}_2^2) B_6}{L(\mu,\hat{\mu},\phi,\hat{\phi},n)}+\frac{\eta^{-1}\mathsf{E}_2+\eta^{-2}n^{-1/4}}{L(\mu,\hat{\mu},\phi,\hat{\phi},n) }\bigg),
\end{align*}
where the two statements hold with probabilities at least $1-16B_3\wedge 1$ (in the covariance-based case) and $1-16B_6\wedge 1$ (in the rank-based case), respectively. The terms $B_3$ and $B_6$ are given in the statements of Theorem \ref{thm:covariate} and Theorem \ref{thm:rank} respectively.

\end{Theorem}

    To interpret the bounds in Theorem \ref{thm:boot}, first note that the terms appearing in the numerators are similar to those appearing in Theorems \ref{thm:covariate} and \ref{thm:rank}, except $\mathsf{E}_1$ and $\mathsf{E}_2$, which estimate the quality of the approximation of $\mu_\omega, \mu_{\phi,\omega}$ and $\phi_\omega$. The quantity $L$ appearing in the denominators serve as a lower bound of the conditional variance of the bootstrap estimators. For its eventual positivity, one needs that $\mathsf{E}_1$ and $\mathsf{E}_2$ tend to zero, i.e., the construction of the regression estimator $\hat{\mu}$ and the transformation estimator $\hat{\phi}$ need to be consistent with high probability. One then should expect with high probability that $L(\mu,\hat{\mu},n),L(\mu,\hat{\mu},\phi,\hat{\phi},n)\ge L_{l}>0$ for a strictly positive constant $L_{l}>0$ for $n$ large enough, which can be explicitly determined from the convergence rates of $\hat \mu$ and $\hat \phi$. The bound, in the covariate-based case for instance, then simplifies to the following: with high probability one has
    \begin{align*}
        \mathsf{d}_{K}(\sqrt{n}(\hat\tau^{\text{boot}}_{M}-\hat{\tau}_{M}^{bc})|\XDE{\cX}_n,\sqrt{n}(\hat{\tau}_{M}^{bc}-\tau)) \le C\left( B_1+B_2+(1+E_1^2) B_3+\eta^{-1}E_1+\eta^{-2}n^{-1/4}\right)
    \end{align*}
    for some finite constant $C>0$.

In the proof of Theorem \ref{thm:boot}, we use the limiting Gaussian distribution as the bridge to bound the distributions between the bootstrap and the original estimators. While showing the consistency of bootstrap does not necessarily need a Gaussian limit, our aim here is to highlight a general procedure to derive the rate of convergence of bootstrap via the rate of convergence for the Gaussian approximation using stabilization theory (see Section \ref{sec:stabilization}), which could potentially be useful for many geometric statistics having some form of local dependency structure.

\section{Gaussian approximation of stabilizing statistics}\label{sec:stabilization}
In this section, we discuss briefly our approach to prove the results in Section \ref{sec:mainresult} based on the notion of stabilization and Malliavin-Stein method. Our proof of Theorem \ref{thm:covariate} relies on Theorem \ref{rateswithconstant}, that we introduce here. It is a refinement of the seminal work of \cite{lachieze2019normal}, providing a quantitative bound for Gaussian approximation of Poisson functionals, and serves as the key step towards obtaining Theorems \ref{thm:covariate} and \ref{thm:rank}. Before we can state the result, we first explicitly introduce the setting for functionals of point processes, and the notion add-one cost operators acting on such functionals.

Recall that $\XDE{\X}:=\X\times \{0,1\}\times \mbb{R}$, with $\X \subseteq \mbb{R}^m$. As before, we will often denote an $\XDE{\X}$-valued random vector $(X,D,\bm{\varepsilon})$ or a the nonrandom vector $(x,d,\varepsilon) \in \XDE{\X}$ by $\XDE{X}$ and $\xde{x}$, respectively. While the concepts and results in this section can be extended to more general spaces, we will stick to the space $\XDE{\X}$ related to our ATE estimation (see Section \ref{sec:ATEstabilization}). We refer to  \cite{lachieze2019normal} for results in more general spaces. Let $(\XDE{\X},\mcal{F})$ be a measure space with the joint probability measure $\XDE{\mbb{Q}}$ of $(X,D,\bm{\varepsilon})$, with $\mbb{Q}$ denoting the marginal distribution of $X$. We also define a semi-metric $d_S(\cdot,\cdot)$ on $\XDE{\X}$ as
$$
d_S(\xde{x},\xde{y})= d(x,y), \quad \xde{x},\xde{y} \in \XDE{\X},
$$
where $d(\cdot,\cdot)$ denotes the Euclidean metric on $\mbb{R}^m$. Let \textbf{N} be the set of $\sigma$-finite counting measures on $(\Omega,\mcal{F})$, which can be interpreted as point configurations in $\Omega$. The set \textbf{N} is equipped with the smallest $\sigma$-field $\mathscr{N}$ such that the maps $m_{A}:\textbf{N}\rightarrow \mbb{N}\cup\{0,\infty\},\mcal{M}\mapsto\mcal{M}(A)$ are measurable for all $A\in\mcal{F}$. A point process is a random element in \textbf{N}. For $\mu \in \textbf{N}$, we write $x\in\mu$ if $\mu(\{x\})\ge 1$. Denote by $\textbf{F}(\textbf{N})$ the class of all measurable functions $f:\textbf{N}\rightarrow \mbb{R}$, and by $L^{0}(\XDE{\X}):=L^{0}(\XDE{\X},\mcal{F})$ the class of all real-valued, measurable functions $F$ on $\Omega$. Note that, as $\mcal{F}$ is the completion of $\sigma(\mu)$, each $F\in L^{0}(\XDE{\X})$ can be written as $F=f(\mu)$ for some measurable function $f\in \textbf{F}(\textbf{N})$. Such a mapping $f$, called a \textit{representative} of $F$, is $\XDE{\mbb{Q}}\circ \mu^{-1}$-a.s. uniquely defined. In order to simplify the discussion, we make the following convention: whenever a general function $F$ is introduced, we will select one of its representatives and denote such a representative mapping by the same symbol $F$.

\begin{Definition}[Cost/Difference Operators]\label{def:cost}
Let $F$ be a measurable function on $\Nb$. The family of add-one cost operators, $\mathsf{D}=(\mathsf{D}_{\xde{x}})_{\xde{x}\in \XDE{\mbb{X}}}$, are defined as 
\begin{align*}
	\mathsf{D}_{\xde{x}}F(\mu):=F(\mu\cup\{\xde{x}\})-F(\mu), \quad \xde{x} \in \XDE{\mbb{X}},\, \mu \in \Nb.
\end{align*}
Similarly, we can define a second-order cost operator (also called iterated add-one cost operator): for any  $\xde{x_1},\xde{x_{2}}\in\XDE{\mbb{X}}$ and $ \mu \in \Nb$,
\begin{align*}
	\mathsf{D}^2_{\xde{x_1},\xde{x_2}}F(\mu):=F(\mu\cup\{\xde{x_1}\}\cup\{\xde{x_2}\})-F(\mu\cup\{\xde{x_1}\})-F(\mu\cup\{\xde{x_2}\})+F(\mu).
\end{align*}
\end{Definition}

\begin{Theorem}\label{rateswithconstant}
    For $p\in (0,1]$ and $n\ge 9$, there exist a constant $C>0$ and a quantity $c(M,\eta,p)>0$ such that, for $E_n$ as defined in (\ref{eq2:F}),
	\begin{equation*}
	\mathsf{d}_{K}\left(\frac{E_n-\mbb{E}E_n}{\sqrt{\Var E_n}},\mcal{N}(0,1)\right)\le C(S_{1}+S_{2}+S_{3}+S_{4}+S_{5})
	\end{equation*}
	with 
	\begin{align*}
	S_{1}&:=c(M,\eta,p)^{\frac{2}{4+p/2}}\frac{1}{n\Var E_n}\sqrt{\int_{\XDE{\X}^{2}}\psi_{n}(\xde{x},\xde{x'})\XDE{\mbb{Q}}^{2}(d(\xde{x},\xde{x'}))},\\
    S_{2}&:=c(M,\eta,p)^{\frac{2}{4+p/2}}\frac{1}{n^{\frac{1}{2}}\Var E_n}\sqrt{\int_{\XDE{\mbb{X}}}\left(\int_{\XDE{\mbb{X}}}\psi_{n}(\xde{x},\xde{x'})\XDE{\mbb{Q}}(d\xde{x'})\right)^{2}\XDE{\mbb{Q}}(d\xde{x})},\\
    S_{3}&:=c(M,\eta,p)^{\frac{2}{4+p/2}}\frac{\sqrt{\Gamma_{n}}}{n^2\Var E_n},\\
	S_{4}&:=\bigg(\max\left\{c(M,\eta,p)^{\frac{1}{4+p/2}}\frac{\Gamma_{n}^{\frac{1}{2}}}{(n^2\Var E_n)^{\frac{1}{2}}},c(M,\eta,p)^{\frac{1}{4+p/2}}\frac{\Gamma_{n}^{\frac{1}{4}}}{(n^2\Var E_n)^{\frac{1}{2}}}+1\right\}\\& \quad \quad\quad \quad+c(M,\eta,p)^{\frac{1}{4+p/2}}\frac{\Gamma_{n}^{\frac{1}{4}}}{n^{\frac{1}{4}}(n^2\Var E_n)^{\frac{1}{2}}}\bigg)c(M,\eta,p)^{\frac{3}{4+p/2}}\frac{\Gamma_{n}}{(n^2\Var E_n)^{\frac{3}{2}}}\\&
 \quad \quad\quad \qquad\qquad+c(M,\eta,p)^{\frac{4}{4+p/2}}\frac{\Gamma_{n}}{(n^2\Var E_n)^{2}},\\
	S_{5}&:=c(M,\eta,p)^{\frac{3}{4+p/2}}\frac{\Gamma_{n}}{(n^2\Var E_n)^{\frac{3}{2}}},
	\end{align*}
	where
	\begin{align*}
	\Gamma_{n}&:=n\int_{\XDE{\mbb{X}}}\mbb{P}(\mathsf{D}_{\xde{x}}E_n(\XDE{\cX}_{n-1})\neq 0)^{\frac{p}{16+2p}}\XDE{\mbb{Q}}(d\xde{x}),\\
\psi_{n}(\xde{x},\xde{x'})&:=\underset{\XDE{\mcal{A}}\subset\XDE{\mbb{X}}:|\XDE{\mcal{A}}|\le 1}{\sup}\mbb{P}(\mathsf{D}^2_{\xde{x},\xde{x'}}E_n(\XDE{\mcal{X}}_{n-2-|\check{\mcal{A}}|}\cup\XDE{\mcal{A}})\neq 0)^{\frac{p}{16+2p}}.
\end{align*}
\end{Theorem}

\begin{Remark}\label{rmk:relationwithgeneral}
As mentioned above, although the above theorem is stated in terms of the ATE estimation $E_n$, 
under a finite $4+p$-moment condition, \citet[Theorem 4.2]{lachieze2019normal} provides a Gaussian approximation bound for functionals of Binomial point processes. Compared to their general result however, since we are interested in some key parameters such as $M$, $\eta$ and $p$, we keep track of the dependency on $M$, $\eta$ and $p$ of the generic constants in their proof resulting in the constant $c(M,\eta,p)$. In particular, from the proof of Theorem~\ref{rateswithconstant}, it will follow that
\begin{align}\label{def:cMetap}
    c(M,\eta,p)\asymp \Big(\frac{M}{\zeta\eta}\Big)^{5}\vee 1,
\end{align}
with $\zeta:=p/(40+10p)$.
\end{Remark}

Notice that to apply Theorem \ref{rateswithconstant}, one needs to find a lower bound for the variance of $E_n$, which we present in the following result, along with quantitative bounds for the variance approximation. Below, we denote 
$\sigma_{\bm{\varepsilon}}^2 = \Var(\bm{\varepsilon})$.

\begin{Lemma}\label{Variancelowerbound}
    Under the assumptions of Theorem \ref{thm:covariate}, for $n\ge 1$,
    \begin{align*}
        \Var E_n\ge \frac{\sigma_{\bm{\varepsilon}}^2 + \Var (\mu_1(X)-\mu_0(X))}{n}.
    \end{align*}
    Furthermore, for $\sigma^2$ as in \eqref{def:sigma},
    \begin{align*}
        |n\Var E_n-\sigma^2|\lesssim \frac{1}{\eta}\left(\frac{M}{n\eta}\right)^{1/(2m)}+\delta_{H_1}^{1/2}+(\delta_{H_2}^{1/2}+1)\cdot\frac{1}{\eta M^{1/2}}+\delta_{H_3}^{1/2}+\frac{1}{\eta^{3}n^{1/3}},
    \end{align*}
    where $\delta_{H_1}$-$\delta_{H_3}$ are defined in \eqref{delta211}. In addition, if we assume $M^{-1}\log n=o(1), n^{-1}M\log n=o(1)$ and $\eta$ bounded away from $0$, then
    \begin{align*}
        |n\Var E_n-\sigma^2|\lesssim \left(\frac{M}{n}\right)^{1/(2m)}+\frac{1}{M^{1/2}}+\frac{1}{n^{1/3}}.
    \end{align*}
\end{Lemma}
We note here that the lower bound in Lemma \ref{Variancelowerbound} from our Assumptions A(1) and B(2)-(3) in Section \ref{sebsec41} implies that $n \Var E_n \ge C$ for some constant $C>0$.

While a result such as Theorem \ref{rateswithconstant} is a very powerful first step in providing (optimal) rates for Gaussian convergence of functionals of Binomial processes, often the integrals appearing in the bound involving the functions $\Gamma_n$ and $\psi_{n}$ are very difficult to directly bound. An assumption on the functional $F$ that can be very effectively used to simplify such computations, is that the $F$ can be expressed as a sum of \emph{local contributions} from each point of the underlying process. This phenomenon is often referred to as \emph{stabilization} in the relevant literature. There is an ever-growing literature on the application of stabilization in combination with result such as Theorem \ref{rateswithconstant} arising from the Malliavin-Stein method. We refer to works by \cite{lachieze2019normal,lachieze2022quantitative,shi2024flexible,bhattacharjee2022gaussian,shi2024multivariate} for additional details and its applications in various statistical problems. In the following section, we describe how stabilization helps us to obtain the Gaussian approximation results in Section \ref{sec:mainresult} from Theorem \ref{rateswithconstant}.

\subsection{Stabilizing functionals of binomial point processes}

Let $\XDE{X}_1,\XDE{X}_2,...,\XDE{X}_n$ be i.i.d.\ random variables sampled from $\XDE{\mbb{Q}}$. The binomial point process $\XDE{\cX}_n$ associated with $\{\XDE{X}_1,\XDE{X}_2,...,\XDE{X}_n\}$ is defined as $\XDE{\cX}_n:=\sum_{i=1}^{n}\delta_{\XDE{X}_1},$ where $\delta$ is the Dirac measure. Given its association with the i.i.d.\ sample $\{\XDE{X}_i\}_{i=1}^{n}$, with a slight abuse of notation, we will often interchangeably use the binomial process $\XDE{\cX_n}$ and the i.i.d.\ sample $\{\XDE{X}_1\}_{i=1}^{n}$. In this paper, we concern ourselves with functionals $F_{n}$ of the binomial process $\XDE{\cX_n}$ that can be represented as a sum of the form
\begin{align}\label{FoB}
F_{n}(\XDE{\cX}_n):=\sum_{i=1}^{n}f_{n}(\XDE{X}_i,\XDE{\cX}_n),
\end{align}
where $f_{n}$ is called a {\em score function}. For $i \in [n]$, when contributions $f_n(\XDE{X}_i,\XDE{\cX}_n)$ are sufficiently `local', one can expect a Gaussian limit for the sequence of functionals $F_n$ as $n \to \infty$.

\subsubsection{Radius of Stabilization}\label{sec:radius}
For $n\ge 1$, the score function $f_n$ is said to be stabilizing if there exists an almost surely finite random variable $R_{n}:\XDE{\X}\times \mbf{N}\rightarrow \mbb{R}_{+}$ such that for all $\mu\in\mbf{N}$, $\xde{x}\in\mu$, and all finite $\XDE{\mcal{A}}\subset \XDE{\X}$ with $|\XDE{\mcal{A}}|\le 7$, 
\begin{align*}
f_n\left(\xde{x},(\mu\cup \XDE{\mcal{A}})\cap B(\xde{x},R_n(\xde{x},\mu))\right)=f_n(\xde{x},\mu\cup \XDE{\mcal{A}}),
\end{align*}
where the random variable $R_{n}$ is called the \textit{radius of stabilization}.

Informally speaking, the above definition states that the value of the score function $f_{n}$ at a point $\xde{x}$ is completely determined by all those point in the configuration that lie in the ball centered at $\xde{x}$ with radius $R_{n}$. Note that in \eqref{FoB}, the score functions $f_n$ are possibly dependent on all the points $\XDE{\cX}_n$. Stabilization is a geometrical localization of the dependence between all the score functions.

We further need the radius of stabilization $R_n$ to satisfy certain tail decay conditions. It is said to decay exponentially if there exist constants $C_1,C_2>0$ such that for $\xde{x}\in\XDE{\X}$, $n\ge 9$ and $r\ge 0$,
\begin{align}\label{tailcondition}
    \mbb{P}\Big(R_n(\xde{x},\XDE{\cX}_{n-8} \cup \{\xde{x}\})\ge r \Big)\le C_1e^{-C_2 nr^{m}}.
\end{align}
The presence of $n-8$ in the above definition is for certain technical reasons, see \cite[Eqn.\ (2.5)]{lachieze2019normal} for more details. An exponential decay as above for the tail probability of the radius of stabilization ensures that the dependence between score functions remain relatively local. Instead of exponentially decaying, \cite{penrose2005normal} also proposed the polynomially decaying condition.


\subsubsection{Connection to matching based ATE estimators}\label{sec:ATEstabilization}
We now connect the terminology above with the matching-based ATE estimator notation introduced in Section~\ref{sec:covariateATE}. The same, of course, also applies to the rank-based ATE estimator in \ref{sec:rank}.

Recall the definition of $\hat{\tau}_{M}^{bc}$ in \eqref{eq2:F} given by
\begin{align}\label{eq:E_n}
    \hat{\tau}_{M}^{bc}&=\frac{1}{n}\sum_{i=1}^{n}(\mu_{1}(X_i)-\mu_{0}(X_i))+\frac{1}{n}\sum_{i=1}^{n}(2D_i-1)\Big(1+\frac{K^{D_i}_{M}(i,\XD{\cX}_n)}{M}\Big)\bm{\varepsilon}_i+(B_{M}-\hat{B}_{M})\nonumber\\
    &=: E_n+(B_{M}-\hat{B}_{M}).
\end{align}
Our focus is the main term $E_n.$ Write 
\begin{align*}
    nE_n=\sum_{i=1}^{n}(\mu_{1}(X_i)-\mu_{0}(X_i))+\sum_{i=1}^{n}(2D_i-1)\Big(1+\frac{K^{D_i}_{M}(i,\XD{\cX}_n)}{M}\Big)\bm{\varepsilon}_i.
\end{align*}

Since $\bm{\varepsilon}$ can possibly depend on the covariate pair $(X,D)$, we are naturally led to consider a binomial point process in the aforementioned space $\XDE{\X}=\X\times\{0,1\}\times \mbb{R}$. Also, as mentioned before, we see that $\XDE{\cX}_n=\{(\XDE{X}_{i},i\in[n])\}$, as the collection of the triplets $\{(X_i,D_{i},\bm{\varepsilon}_{i})\}_{i=1}^{n}$, can be viewed as a binomial point process of size $n$ in the space $\XDE{\X}$ distributed as $(X,D,\bm{\varepsilon})$. Note that the collection $\{(X_i,D_{i},\bm{\varepsilon}_{i})\}_{i=1}^{n}$ and the sample $\{(X_i,Y_i,D_i)\}_{i=1}^{n}$ are linked through the true regression functions (conditional means) $\mu_0(x)$ and $\mu_1(x)$, in particular, one has $\bm{\varepsilon}_i=Y_i - \mu_{D_i}(X_i)$ for $i \in [n]$.

Therefore, $nE_n$ can be viewed as a functional of the binomial process $\XDE{\cX}_n$ and be represented as a sum of score functions as
\begin{align*}
    nE_n(\XDE{\cX}_n) &:= \sum_{i=1}^{n}(\mu_{1}(X_i)-\mu_{0}(X_i))+\sum_{i=1}^{n}(2D_i-1)\Big(1+\frac{K^{D_i}_{M}(i,\XD{\cX}_n)}{M}\Big)\bm{\varepsilon}_i\\
    &= \sum_{j=1}^{n}(\mu_{1}(X_j)-\mu_{0}(X_j))+\sum_{j=1}^{n}(2D_j-1)\bm{\varepsilon}_j \\
    & \qquad + \frac{1}{M}\sum_{i=1}^{n}(2D_i-1)\bm{\varepsilon}_i \sum_{j=1,D_j=1-D_i}^{n}\mathds{1}(i\in \mcal{J}^{D_i}_{M}(j,\XD{\cX}_n))\\
    &= \sum_{j=1}^{n}(\mu_{1}(X_j)-\mu_{0}(X_j))+\frac{1}{n}\sum_{j=1}^{n}(2D_j-1)\bm{\varepsilon}_j \\
    & \qquad + \frac{1}{M}\sum_{j=1}^{n}(1-2D_j) \sum_{i=1,D_i=1-D_j}^{n}\bm{\varepsilon}_i \mathds{1}(i\in \mcal{J}^{1-D_j}_{M}(j,\XD{\cX}_n)) =: \sum_{j=1}^n \xi_n(\XDE{X}_j, \XDE{\cX}_n),
\end{align*}
where the score function $\xi_n$ for $\xde{\nu}_k =\{\xde{x}_i\}_{i=1}^k \in \mbf{N}, k \in \mbb{N}$ is given by
\begin{align}\label{xin}
    \xi_n(\xde{x}_j, \xde{\nu}_k) \equiv \xi(\xde{x}_j, \xde{\nu}_k)&:=(\mu_1(x_j)-\mu_0(x_j))+(2d_j-1) \varepsilon_j \nonumber\\
    & \quad+\frac{1}{M}(1-2d_j)\sum_{i=1,d_i=1-d_j}^{k}\bm{\varepsilon}_i \mathds{1}(i\in \mcal{J}^{1-d_j}_{M}(j,\xde{\nu}_k)),\ j\in[k].
\end{align}
Note that all terms in the score $\xi_n(\XDE{X}_j, \XDE{\cX}_n)$ corresponding to the $j$-th sample $\XDE{X}_j$ is determined by $\XDE{X}_j$ except for 
$$
\sum_{i=1,D_i=1-D_j}^{n}\bm{\varepsilon}_i \mathds{1}(i\in \mcal{J}^{D_i}_{M}(j,\XD{\cX}_n)) = \sum_{i=1}^{n}\bm{\varepsilon}_i \mathds{1}(i\in \mcal{J}^{1-D_j}_{M}(j,\XD{\cX}_n)) .
$$ 
This is a function of $\XDE{X}_j$ and all those points in $\XDE{\cX}_n$ that are $M$-NNs of $\XDE{X}_j$ (in the metric $d_S$) when considering only those points in $\XDE{\cX}_n^{1-D_j}$, where for a point collection $\xde{\nu}_n=\{\xde{x}_i\}_{i=1}^n$, we write
\begin{align*}
\xde{\nu}_n^\omega:= \{\xde{x}_i \in \xde{\nu}_n: d_i=\omega\},
\end{align*} 
So, it is straightforward to see that for $j \in [n]$, the score function $\xi(\XDE{X}_j, \XDE{\cX}_n)$ is stabilizing with radius of stabilization (in the metric $d_S$) given by the $M$-NN distance from $\XDE{X}_j$ among the points in $\XDE{\cX}_n^{1-D_j}$. In other words, for a point collection $\xde{\nu}_k = \{\xde{x}_i\}_{i=1}^k$, we can take for $j \in [k]$,
\begin{equation}\label{eq:RoS}
    R_n(\xde{x}_j, \xde{\nu}_k)\equiv R(\xde{x}_j, \xde{\nu}_k) =\max_{i \in  \mcal{J}^{1-d_j}_{M}(j,\xd{\nu}_k)} d(x_i,x_j).
\end{equation}
The radius is also non-increasing in the point collection $\xde{\nu}_k$, so the above holds true even if we add additional points $\mcal{A}$ with $|\mcal{A}| \le 7$ to $\xde{\nu}_k$ (see the discussion in the beginning of the proof of \cite[Theorem 3.1]{lachieze2019normal}).

\subsubsection{Moment condition}
We say that the score function $f_n$ satisfies the $(4+p)$-moment condition if there exists $p\in (0,1]$ such that for all $n\ge 9$, $\xde{x}\in\XDE{\X}$, $\XDE{\mcal{A}}\subset \XDE{\X}$ with $|\XDE{\mcal{A}}|\le 7$,
\begin{align}\label{momentcondition}
    \left(\mbb{E}|f_{n}(\xde{x},\XDE{\cX}_{n-8} \cup\{\xde{x}\}\cup\XDE{\mcal{A}})|^{4+p}\right)^{\frac{1}{4+p}}\le M_{n,p}(\xde{x},\XDE{A}).
\end{align}
In the sequel, we will usually write $M_{n}(\xde{x}, \XDE{A})$ instead for notational convenience.

The above moment condition could be motivated by connections to the classical Berry-Esseen theorem, which provides Gaussian approximation bounds under the assumption of i.i.d.\ score functions, i.e, in \eqref{FoB}, $f_{n}(\XDE{X}_i,\XDE{\cX}_n)\equiv f_{n}(\XDE{X}_i)$ for $i\in [n]$ with bounded third moments $\mbb{E}|f_n(\XDE{X}_1)|^{3}<\infty$. Specifically, it states that
\begin{align}\label{BE}
\mathsf{d}_{K}\left(\frac{F_n-\mbb{E}F_n}{\sqrt{\Var F_n}},\mcal{N}(0,1)\right)
\lesssim&~~ \frac{\mathbb{E}|f_n(\XDE{X}_1)-\mathbb{E}f_n(\XDE{X}_1)|^{3}}{\Var f_n(\XDE{X}_1) }\frac{1}{\sqrt{\Var F_n}},
\end{align}
where $\mcal{N}$ is the standard normal random variable. Note here that our score functions are far from independent and can have significant local dependencies. The motivation behind~\eqref{momentcondition} is to go beyond independence. Assumptions of the radius of stabilization in Section~\ref{sec:radius} restricts the dependence between the scores within balls with radii that decay exponentially. However, in contrast to the classical Berry-Esseen bound that require finite third moments, due to the dependence in our model, we require a slightly stronger $(4+p)$-moment condition.

\section{Road-map for the proofs}\label{sec:roadmap}
We now outline the high-level ideas behind the proofs of our Gaussian and bootstrap approximation results.

\subsection{Gaussian Approximation Results}\label{sec:gauss}
Recall from \eqref{eq2:F} the bias corrected ATE estimator $\hat{\tau}_{M}^{bc}=E_n+(B_M-\hat{B}_{M})$ for the covariate based matching. The proof of Theorem \ref{thm:covariate} follows in three steps:
\begin{itemize}
    \item[(1)] We apply stabilization theory to bound 
    \begin{align*}
        I_0:={\sf d}_{K}\left(\frac{E_n-\mbb{E}E_n}{\sqrt{\Var E_n}},\mcal{N}(0,1)\right).
    \end{align*}
    For example, to prove Theorem \ref{thm:covariate}, in Lemma~\ref{lem:tb}, we show that the radius of stabilization \eqref{eq:RoS} for the scores given by \eqref{xin} associated to $E_{n}$ does indeed have an exponentially decaying tail and satisfies \eqref{tailcondition} with $C_1=C$ and $C_2=C\eta M^{-1}$ for some $C>0$ depending only on $m$ and $g_{\min}$ (see Remark \ref{rem:tail}). This, in addition to a moment bound as in \eqref{momentcondition}, leads to a Gaussian limit for $E_n$ (appropriately centered and scaled) via an application of our general bound in Theorem \ref{rateswithconstant}.

    \item [(2)] Next, from the bound in \textbf{Step (1)}, by scaling with the factor $\sqrt{n\Var E_n}/\sigma$, which is close to $1$ by Lemma \ref{Variancelowerbound} with the help of auxiliary Lemma \ref{lemma:l2convdensityratio}, we obtain a bound on
    \begin{align*}
    I_1:={\sf d}_{K}\left(\frac{\sqrt{n}\, (E_{n}-\mbb{E}E_n)}{\sigma},\mcal{N}(0,1)\right),
\end{align*}
where $\sigma^{2}$ is defined in \eqref{def:sigma}.
\vspace{0.1in}
\item [(3)] Lastly, noting that $\mbb{E}E_n=\tau$ and bounding $\hat{\tau}_{M}^{bc}-E_n=B_M-\hat{B}_{M}$ through its moment in \eqref{biasboundcovariate}, we obtain a bound for
\begin{align*}
    I_2:={\sf d}_{K}\left(\frac{\sqrt{n}\,(\hat{\tau}_{M}^{bc}-\tau)}{\sigma},\mcal{N}(0,1)\right).
\end{align*}
\end{itemize}
The proof of Theorem \ref{thm:rank} follows similar steps suitably adapted to the rank-based matching setting.

\subsection{Bootstrap Approximation Results}
In order to bound the distributional distance between $\sqrt{n}(\tau_{M}^{\text{boot}}-\hat{\tau}_{M}^{\text{bc}})$ and $\sqrt{n}(\hat{\tau}_{M}^{\text{bc}}-\tau)$ appearing in the first assertion in Theorem \ref{thm:boot}, we use the Gaussian limit in Theorem \ref{thm:covariate} as the bridge. Roughly, The proof of Theorem \ref{thm:boot} broadly follow in three steps:
\begin{itemize}
    \item [(1)] We first show that conditional on the original data $\XDE{\cX}_n$, the random variable $\sqrt{n}(\tau_{M}^{\text{boot}}-\hat{\tau}_{M}^{\text{bc}})$ can be expressed as an average of independent normal random variables. Thus, it also has a Gaussian distribution conditional on $\XDE{\cX}_n$.
    \vspace{0.1in}
    \item [(2)] On the other hand, by Theorem \ref{thm:covariate}, $\sqrt{n}(\hat{\tau}_{M}^{\text{bc}}-\tau)$ can be quantitatively approximated by a Gaussian distribution with variance $\sigma^2$ defined at \eqref{def:sigma}.
    \vspace{0.1in}
    \item [(3)] Finally, it suffices to bound the distance between these two Gaussian distributions, which we achieve by bounding the difference between the sample variance of $\sqrt{n}(\tau_{M}^{\text{boot}}-\hat{\tau}_{M}^{\text{bc}})$ conditional on $\XDE{\cX}_n$ and the limiting variance $\sigma^2$ with high probability, using standard concentration techniques.
\end{itemize}
A similar procedure also applies for the bootstrap approximation in the case of the rank based ATE estimator.

\subsection*{Acknowledgement}
The first, third and the fourth listed authors were supported by the NSF Grant DMS-2053918. The second author was supported in part by the German Research Foundation (DFG) Project 531540467.

\bibliographystyle{abbrvnat}

\bibliography{example}

\appendix

\section{Intermediate Results}
In order to carry out the steps outlined in Section~\ref{sec:roadmap}, we first need to establish several intermediate results which we do in this section. In particular, we provide a proof for the variance estimation bounds in Lemma \ref{Variancelowerbound}.

First, we focus on the bounds related to the tail condition \eqref{tailcondition}. 
Note that in Section \ref{sec:ATEstabilization}, we showed that the score function $\xi_n$ associated to $nE_n$ is stabilizing with the radius of stabilization $R_n$ given by \eqref{eq:RoS}. Moreover, denote $p_{x,r,1-d}:=\mbb{P}(X \in B(x,r),D=1-d)$. By Assumption A.3 in Section \ref{AssumpA} , we have
\begin{equation}\label{eq:pbd}
    p_{x,r,1-d}=\mbb{P}(X \in B(x,r))\mbb{P}(D=1-d|X \in B(x,r))\in [\eta,1-\eta]\mbb{P}(X \in B(x,r)).
\end{equation}
In what follows, we write $V_m$ for the volume of the unit ball in $\mbb{R}^m$.
\begin{Lemma}\label{lem:tb}
    Under Assumption A.2 in Section \ref{AssumpA}, for all $M \in [n]$, $n \ge 9$ and $\xde{x} =(x,d,\varepsilon) \in \XDE{\X}$,
    \begin{align*}
    \mbb{P}(R_n(\xde{x}, \XDE{\mcal{X}}_{n-8} + \delta_{\xde{x}}) \ge r) \le e^2\cdot \exp\left\{-\frac{V_m g_{\min} \eta}{(2M) \vee 8} nr^{m}\right\}.
\end{align*}
\end{Lemma}

\begin{proof}[Proof of Lemma~\ref{lem:tb}]
Note from the definition \eqref{eq:RoS} that
\begin{align*}
\mbb{P}(R_n(\xde{x}, \XDE{\mcal{X}}_{n-8} + \delta_{\xde{x}}) \ge r) &= \mbb{P}(\XDE{\mcal{X}}_{n-8}^{1-d} (\XDE{B}(x,r)) \le M)
\\& =\mbb{P}(\text{Bin}(n-8,\mbb{P}(X\in B(x,r),D=1-d))<M)
\\& =\mbb{P}(\text{Bin}(n-8,p_{x,r,1-d}) \le M),
\end{align*}
so we need to show the bound in the assertion for $\mbb{P}(\text{Bin}(n-8,p_{x,r,1-d})\le M)$.
If $M \le \frac{1}{2}(n-8)p_{x,r,1-d}$, by the Chernoff bound for binomial random variables (see Lemma 1.1 in \cite{penrose2003random}), we have
\begin{align*}
    \mbb{P}(\text{Bin}(n-8,p_{x,r,1-d})\le M) &\le e^{-(n-8)p_{x,r,1-d} H\left(\frac{M}{(n-8)p_{x,r,1-d}}\right)}\\
    &\le e^{-(n-8)p_{x,r,1-d} H\left(\frac{1}{2}\right)} \le e^{-(n-8)p_{x,r,1-d}/8},
\end{align*}
where $H(x):=1-x+x\log x$ for $x>0$.
On the other hand, if $M>\frac{1}{2}(n-8)p_{x,r,1-d}$, then
\begin{align*}
    e\cdot e^{-\frac{1}{2M}(n-8)p_{x,r,1-d}}\ge e \cdot e^{-1}=1 \ge \mbb{P}(R_n(\xde{x}, \XDE{\mcal{X}}_{n-8} + \delta_{\xde{x}}) \ge r).
\end{align*}
Combining the two bounds above yields
$$
\mbb{P}(R_n(\xde{x}, \XDE{\mcal{X}}_{n-8} + \delta_{\xde{x}}) \ge r) \le e \cdot \exp\left\{-\frac{1}{(2M)\vee 8} (n-8)p_{x,r,1-d} \right\}. 
$$
The result is now a direct consequence of Assumption A.2 in Section \ref{AssumpA} and \eqref{eq:pbd}.
\end{proof}

\begin{Remark}\label{rem:tail}
    Lemma \ref{lem:tb} confirms the tail condition \eqref{tailcondition}. Since $M \ge 1$, we see that $R_n$ satisfies the exponentially decaying condition \eqref{tailcondition} with $C_{1}=e^2$ and 
    $$
    C_2=: V_m g_{\min}\eta/(8M)\ge V_m g_{\min}\eta/((2M)\vee 8)
    $$
    For notational convenience, we will take $C_1=C$ and $C_2=C\eta M^{-1}$ for some $C>0$ depending only on $m$ and $g_{\min}$.

    Also note, in Lemma \ref{lem:tb}, instead of a uniform lower bound on the density of $X$ in Assumption A.2, one can relax it to assuming $
    \inf_{x}~\mbb{P}(X \in B(x,r))\ge Cr^{m}$ for some $C>0$.
\end{Remark}

Throughout this section, we will use the following result from \citet[Lemma 5.1 (b)]{lachieze2019normal} many times. It should be noted here that Condition (2.1) in \citet{lachieze2019normal}, which is required for the proof, is trivially satisfied by our probability measure $\XDE{\mbb{Q}}$ due to Assumption A.2 in Section \ref{AssumpA}.

\begin{Lemma}\label{generalintegral}
    There is a constant $C>0$, depending only on $m$, such that
\begin{equation*}
\int_{\X\backslash B(x,r)} e^{-\beta d(x,y)^{m}} \, \mbb{Q}(dy) \le \frac{C}{\beta} e^{-\beta r^{m}/2},
\end{equation*}
for all $\beta\ge 1$, $x\in\X$ and $r\ge 0$.
\end{Lemma}

Recall $\Gamma_n$ in Theorem \ref{rateswithconstant}. Since $\mbb{P}(D_{\xde{x}}E_n(\XDE{\cX}_{n-1})\neq 0) \le 1$ and $\XDE{\mbb{Q}}$ is a probability measure, we trivially obtain the upper bound
\begin{equation}\label{eq:boundgamman}
\Gamma_n=n\int_{\XDE{\mbb{X}}}\mbb{P}(\mathsf{D}_{\xde{x}}E_n(\XDE{\cX}_{n-1})\neq 0)^{\frac{p}{16+2p}}\XDE{\mbb{Q}}(d\xde{x}) \le n.
\end{equation}

We next bound the integrals of $\psi_n(\xde{x},\xde{x'})$ in Theorem \ref{rateswithconstant} in the following lemma.

\begin{Lemma}\label{boundpsin}
    Let $\psi_n(\xde{x},\xde{x'})$ be as in Theorem \ref{rateswithconstant}. Then for $n\ge 9$,
 \begin{align*}
    n^2\int_{\XDE{\X}^{2}}\psi_{n}(\xde{x},\xde{x'})\XDE{\mbb{Q}}^{2}(d(\xde{x},\xde{x'}))&\le \alpha^{-1}((\eta^{-1}M)^{\alpha+1}\vee 1)n,\\
    n\int_{\XDE{\mbb{X}}}\left(n\int_{\XDE{\mbb{X}}}\psi_{n}(\xde{x},\xde{x'})\XDE{\mbb{Q}}(d\xde{x'})\right)^{2}\XDE{\mbb{Q}}(d\xde{x})&\le \alpha^{-2}((\eta^{-1}M)^{2\alpha+2}\vee 1)n,
\end{align*}
where $\alpha:=p/(16+2p)$ and $p\in(0,1]$.
\end{Lemma}

\begin{proof}[Proof of Lemma~\ref{boundpsin}]
Note that according to Definition \ref{def:cost} and \citet[Lemma 5.2]{lachieze2019normal},
\begin{align*} 
 &\underset{\XDE{\mcal{A}}\subset\XDE{\mbb{X}}:|\XDE{\mcal{A}}|\le 1}{\sup}\mbb{P}(\mathsf{D}^2_{\xde{x},\xde{x'}}E_n(\XDE{\mcal{X}}_{n-2-|\check{\mcal{A}}|}\cup\XDE{\mcal{A}})\neq 0)\\
 &\le \underset{\XDE{\mcal{A}}\subset\XDE{\mbb{X}}:|\XDE{\mcal{A}}|\le 1}{\sup}\mbb{P}(\mathsf{D}_{\xde{x}}\xi_n(\xde{x'},\XDE{\mcal{X}}_{n-2-|\check{\mcal{A}}|}\cup\{\xde{x'}\}\cup\XDE{\mcal{A}})\neq 0)\\
 &\quad + \underset{\XDE{\mcal{A}}\subset\XDE{\mbb{X}}:|\XDE{\mcal{A}}|\le 1}{\sup}\mbb{P}(\mathsf{D}_{\xde{x'}}\xi_n(\xde{x},\XDE{\mcal{X}}_{n-2-|\check{\mcal{A}}|}\cup\{\xde{x}\}\cup\XDE{\mcal{A}})\neq 0)\\
 &\quad + n\int_{\XDE{\X}}\underset{\XDE{\mcal{A}}\subset\XDE{\mbb{X}}:|\XDE{\mcal{A}}|\le 1}{\sup}\mbb{P}(\mathsf{D}^2_{\xde{x},\xde{x'}}\xi_n(\xde{z},\XDE{\mcal{X}}_{n-3-|\check{\mcal{A}}|}\cup\{\xde{z}\}\cup\XDE{\mcal{A}})\neq 0)\XDE{\mbb{Q}}(d\xde{z})\\
 &\quad + \underset{\xde{z}\in\XDE{\X}}{\sup}~\mbb{P}(\mathsf{D}^2_{\xde{x},\xde{x'}}\xi_n(\xde{z},\XDE{\mcal{X}}_{n-3}\cup\{\xde{z}\})\neq 0)\\
 &=: P_1+P_2+P_3+P_4.
\end{align*}
We start by bounding $P_1$ and $P_2$. By the definition \eqref{xin} of $\xi_n$, it follows that 
$$
\mbb{P}(\mathsf{D}_{\xde{x}}\xi_n(\xde{x'},\XDE{\mcal{X}}_{n-2-|\check{\mcal{A}}|}\cup\{\xde{x'}\}\cup\XDE{\mcal{A}})\neq 0) \le  \mbb{P}\left(d_S(\xde{x},\xde{x'}) \le R_n(\xde{x'}, \XDE{\mcal{X}}_{n-2-|\check{\mcal{A}}|}\cup\{\xde{x'}\}\cup\XDE{\mcal{A}}) \right),
$$
since the latter event implies that $\xde{x}$ is not among the $M$-NN's of $\xde{x'}$ in $\XDE{\mcal{X}}_{n-2-|\check{\mcal{A}}|}\cup\{\xde{x},\xde{x'}\}\cup\XDE{\mcal{A}}$. 
Now the tail bound in Remark \ref{rem:tail} and monotonicity of $R_n$ in the second argument (and a similar argument for $P_2$) yields
\begin{align*}
    P_1\le C_1 e^{ -C_2 n d(x,x')^m} \qquad \text{and} \qquad P_2\le C_1 e^{ -C_2 n d(x,x')^m}.
\end{align*}
As for $P_3$ and $P_4$, again by the tail bound in Remark \ref{rem:tail} and \citet[Lemma 5.8, with $K=\XDE{\X}$]{lachieze2019normal}, we have
\begin{align*}
    P_3&\lesssim n\int_{\X}e^{ -C_2 n \max \{ d(x,z)^m, d(x', z)^m \}}\mbb{Q}(dz) \quad \text{and} \quad 
    P_4\lesssim \underset{z\in \X}{\sup}~e^{ -C_2 n \max \{d(x,z)^m, d(x', z)^m\}}.
\end{align*}
Since $\max \{d(x,z), d(x', z)\} \ge \frac{1}{2}d(x,x')$ for any $z \in \X$, we obtain
$$
P_4 \lesssim e^{ -C_2 n d(x,x')^m/2^m}.
$$
For convenience, set $r:=\frac{1}{2}d(x,x')$. We can write
\begin{align*}
    P_3 &\lesssim n\int_{\mbb{X}}e^{ -C_2n\max \{d(x,z)^m, d(z, x')^m \}}\mbb{Q}(dz)\\
  &=n\int_{B(x,r)}e^{ -C_2n\max \{ d(x,z)^m, d(z, x')^m \}}\mbb{Q}(dz)+n\int_{\mbb{X}\backslash B(x,r)}e^{ -C_2 n\max \{ d(x,z)^m, d(z, x')^m \}}\mbb{Q}(dz)\\
  &=: P_{31}+P_{32}.    
\end{align*}
For any $z\in B(x,r)$, by triangle inequality, it holds that $d(z,x')\ge r$, so that
\begin{align*}
P_{31}\le n\int_{B(x,r)}e^{-C_2 nr^m}\mbb{Q}(dz) \le nr^{m}e^{-C_2 nr^m}.
\end{align*}
By noting the fact that $x\le e^{x}$ for all $x\ge 0$, we have
\begin{align*}
    2\left(C_2\cdot\frac{nr^m}{2}\right)\le 2e^{C_2\cdot\frac{nr^m}{2}}. 
\end{align*}
Then, it holds that
\begin{align*}
P_{31}\le nr^{m}e^{-C_2 nr^m}\lesssim
C_2^{-1}e^{-C_2 nr^m/2}.
\end{align*}
On the other hand, by setting $\beta=C_2n$ in Lemma \ref{generalintegral}, we have
\begin{align*}
    P_{32}\le n\int_{\mbb{X}\backslash B(x,r)}e^{ -C_2nd(x,z)^m}\mbb{Q}(dz)\lesssim C_2^{-1} e^{-\beta r^{m}/2}.
\end{align*}
Combining the above parts, we obtain
\begin{align*}
P_3\lesssim C_2^{-1}e^{-\beta r^{m}/2}=C_2^{-1}e^{-C_2 nd(x,x')^{m}/2^{m+1}}.
\end{align*}
Therefore, combining the above bounds on $P_1,P_2,P_3$ and $P_4$, we obtain
\begin{align*}
    P_1+P_2+P_3+P_4\lesssim (1+C_2^{-1})e^{-C_2 nd(x,x')^{m}/2^{m+1}}.
\end{align*}

Now, recall $\psi_n(\xde{x},\xde{x'})$ in Theorem \ref{rateswithconstant}. Then according to Lemma \ref{generalintegral}, we conclude that
\begin{align*}
    n\int_{\XDE{\X}}\psi_n(\xde{x},\xde{x'})\XDE{\mbb{Q}}(d\xde{x'})&\le n\int_{\XDE{\X}}(P_1+P_2+P_3+P_4)^{\alpha}\XDE{\mbb{Q}}(d\xde{x'})\\
    &\lesssim (1+C_2^{-1})^{\alpha}n\int_{\X}e^{-\alpha C_2 nd(x,x')^{m}/2^{m+1}}\mbb{Q}(dx')\\
    &\lesssim (1+C_2^{-1})^{\alpha} (\alpha C_2)^{-1},
\end{align*}
where $\alpha:=p/(16+2p)$. Thus, it yields 
\begin{align*}
    n^2\int_{\XDE{\X}^{2}}\psi_{n}(\xde{x},\xde{x'})\XDE{\mbb{Q}}^{2}(d(\xde{x},\xde{x'}))&\lesssim (1+C_2^{-1})^{\alpha}(\alpha C_2)^{-1}n\\
    &\lesssim \alpha^{-1}((\eta^{-1}M) \vee 1)^{\alpha}(\eta^{-1}M) n=\alpha^{-1}((\eta^{-1}M)^{\alpha+1}\vee 1)n,\\
n\int_{\XDE{\mbb{X}}}\left(n\int_{\XDE{\mbb{X}}}\psi_{n}(\xde{x},\xde{x'})\XDE{\mbb{Q}}(d\xde{x'})\right)^{2}\XDE{\mbb{Q}}(d\xde{x})&\lesssim (1+C_2^{-1})^{2\alpha} (\alpha C_2)^{-2}n\lesssim \alpha^{-2}((\eta^{-1}M)^{2\alpha+2}\vee 1)n.
\end{align*}
\end{proof}

Next, we focus on the estimation of the empirical variance state in 
Lemma \ref{Variancelowerbound}. The proof of the result relies on the density ratio estimation results in \citet{lin2023estimation}. Following the framework therein, for the density ratio
\begin{align*}
    r(x):=\frac{g_1(x)}{g_0(x)}=\frac{\mbb{P}(D=0)}{\mbb{P}(D=1)}\frac{e(x)}{1-e(x)},
\end{align*}
define the density ratio estimator (see \citet[Definition 2.2]{lin2023estimation}) for it as
\begin{align*}
    \hat{r}_{M}(x)=\frac{n_0}{n_1}\frac{K_{M}^{0}(1,\XDE{\cX}_{n})}{M},
\end{align*}
where recall the definition of $K_{M}^{0}(1,\XDE{\cX}_{n})$ in \eqref{matchedtimes}.

\begin{Lemma}\label{lemma:l2convdensityratio}
    Under the assumptions of Theorem \ref{thm:covariate},
    \begin{align*}
        &\mbb{E}\bigg(\left(\frac{n_1}{n_0}\right)^2\mbb{E}\bigg(\bigg(\frac{n_0}{n_1}\frac{K_{M}^{0}(1,\XDE{\cX}_{n})}{M}-\frac{\mbb{P}(D=0)}{\mbb{P}(D=1)}\frac{e(X_1)}{1-e(X_1)}\bigg)^{2}\bigg|\{D_i\}_{i=1}^{n}\bigg)\mathds{1}(n_0>0)\bigg)\\
        &\lesssim \frac{1}{\eta^2}\left(\frac{M}{n\eta}\right)^{1/m}+\delta_{H_1}+(\delta_{H_2}+1)\cdot\frac{1}{\eta^2 M}+\delta_{H_3},
    \end{align*}
and
    \begin{align*}
        &\mbb{E}\bigg(\left(\frac{n_0}{n_1}\right)^2\mbb{E}\bigg(\bigg(\frac{n_1}{n_0}\frac{K_{M}^{1}(1,\XDE{\cX}_{n})}{M}-\frac{\mbb{P}(D=1)}{\mbb{P}(D=0)}\frac{1-e(X_1)}{e(X_1)}\bigg)^{2}\bigg|\{D_i\}_{i=1}^{n}\bigg)\mathds{1}(n_1>0)\bigg)\\
        &\lesssim \frac{1}{\eta^2}\left(\frac{M}{n\eta}\right)^{1/m}+\delta_{H_1}+(\delta_{H_2}+1)\cdot\frac{1}{\eta^2 M}+\delta_{H_3},
    \end{align*}
    where $\delta_{H_1}$-$\delta_{H_3}$ are given in \eqref{delta211}. Moreover, under the assumptions $M^{-1}\log n=o(1), n^{-1}M\log n=o(1)$ and $\eta$ bounded away from $0$, both the upper bounds above simplify to $\left(M/n\right)^{1/m}+M^{-1}.$
\end{Lemma}

\begin{proof}[Proof of Lemma~\ref{lemma:l2convdensityratio}]
    We start by proving the first two assertions. It suffices to prove the first bound, since the second one follows by symmetry. The proof adapts \citet[proof of Theorem B.3 and B.4]{lin2023estimation} with a slightly more careful estimation. Note that \citet[Proof of Theorem B.3]{lin2023estimation} was not stated in a strictly non-asymptotic manner, i.e., they used some asymptotic statements to simplify the result to obtain the convergence rates. In the following, while citing results from them, we also slightly modify such asymptotic parts for our purposes. We first write the expectation as
    \begin{align}\label{div:IandIc}
        &\mbb{E}\bigg(\left(\frac{n_1}{n_0}\right)^2\mbb{E}\bigg(\bigg(\frac{n_0}{n_1}\frac{K_{M}^{0}(1,\XDE{\cX}_{n})}{M}-\frac{\mbb{P}(D=0)}{\mbb{P}(D=1)}\frac{e(X_1)}{1-e(X_1)}\bigg)^{2}\bigg|\{D_i\}_{i=1}^{n}\bigg)\mathds{1}(n_0>0)\bigg)\nonumber\\
        &=\mbb{E}\bigg(\left(\frac{n_1}{n_0}\right)^2\mbb{E}\bigg(\bigg(\frac{n_0}{n_1}\frac{K_{M}^{0}(1,\XDE{\cX}_{n})}{M}-\frac{\mbb{P}(D=0)}{\mbb{P}(D=1)}\frac{e(X_1)}{1-e(X_1)}\bigg)^{2}\bigg|\{D_i\}_{i=1}^{n}\bigg)\mathds{1}(n_0>0,I)\bigg)\nonumber\\
        &\quad+ \mbb{E}\bigg(\left(\frac{n_1}{n_0}\right)^2\mbb{E}\bigg(\bigg(\frac{n_0}{n_1}\frac{K_{M}^{0}(1,\XDE{\cX}_{n})}{M}-\frac{\mbb{P}(D=0)}{\mbb{P}(D=1)}\frac{e(X_1)}{1-e(X_1)}\bigg)^{2}\bigg|\{D_i\}_{i=1}^{n}\bigg)\mathds{1}(n_0>0,I^c)\bigg),
    \end{align}
    where $I=\{|n_0-np_0|<np_0/2\}$, so that on $I$, it holds that $np_0/2< n_0< 3np_0/2$ and according to \eqref{concn1}, we have
    \begin{align}\label{concp0t}
        \mbb{P}(I^c)\le 2e^{-\frac{2t^2}{np_0(1-p_0)}}=2e^{-\frac{np_0}{2(1-p_0)}}.
    \end{align}
    On $I$, with our assumption $M\le C_0 n \eta$ and $\eta\le p_0\le 1$, it holds that
    \begin{align*}
        \left(\frac{4}{g_{0,\min}V_m}\right)^{1/m}\left(\frac{M}{n_0}\right)^{1/m}\le \left(\frac{4}{g_{0,\min}V_m}\right)^{1/m}\left(\frac{C_0n\eta}{n_0}\right)^{1/m}\le \left(\frac{4}{g_{0,\min}V_m}\right)^{1/m}\left(2C_0\right)^{1/m}.
    \end{align*}
    Then, we set $\delta:=2\left(\frac{4}{g_{0,\min}V_m}\right)^{1/m}\left(2C_0\right)^{1/m}$ in \citet[Proof of Theorem B.3]{lin2023estimation} such that on $I$, 
    \begin{align*}
        \delta_n:=\left(\frac{4}{g_{0,\min}V_m}\right)^{1/m}\left(\frac{M}{n_0}\right)^{1/m}\le \delta/2.
    \end{align*}
    With this inequality holding on $I$, we can follow \citet[Proof of Theorem B.3, S3.31]{lin2023estimation} to obtain the upper bound: on $I$,
\begin{align*}
        \mbb{E}\big[\hat{r}_M(x)|\{D_i\}_{i=1}^{n}\big]&\le \frac{g_1(x)+L\delta_n}{g_0(x)-2L\delta_n} \frac{3np_0/2}{3np_0/2+1}+\frac{n}{M} e^{-(1-\log 2)M},
\end{align*}
where we modified $o(n_0^{-\gamma})$ in \citet[Proof of Theorem B.3, S3.31]{lin2023estimation} replaced by a non-asymptotic rate $\frac{n}{M} e^{-(1-\log 2)M}$ according to the first inequality there in \citet[Proof of Theorem B.3, S3.30]{lin2023estimation}, instead of using its second inequality $o(n_0^{-\gamma})$ which is an asymptotic rate. In fact, it may not even be of smaller order as we do not restrict the parameter $p_0\ge \eta$ to be fixed.

Similar modifications take place in \citet[Proof of Theorem B.3, S3.32]{lin2023estimation} for the lower bound, where with $r_{\text{ratio}}$ as in Assumption A.4 in Section \ref{AssumpA}, we replace the asymptotic $o(n_0^{-\gamma})$ term by 
$$
C(1-r_{\text{ratio}})\frac{n}{M}e^{M-r_{\text{ratio}}np_0-M\log M+M\log (r_{\text{ratio}}np_0)}
$$
according to the first inequality derived in their proof, instead of using its second inequality $o(n_0^{-\gamma})$, an asymptotic rate. Therefore, we have the modified lower bound accordingly: on $I$,
\begin{align*}
    \mbb{E}\big[\hat{r}_M(x)|\{D_i\}_{i=1}^{n}\big]&\ge\frac{g_1(x)-L\delta_n}{g_0(x)+2L\delta_n} \frac{np_0/2}{np_0/2+1}-C(1-r_{\text{ratio}})\frac{n}{M}e^{M-r_{\text{ratio}}np_0-M\log M+M\log (r_{\text{ratio}}np_0)}.
\end{align*}
Combining the upper and lower bound, it yields that on $I$, 
\begin{align}\label{mod:bias}
    \left|\mbb{E}\big[\hat{r}_{M}(x)|\{D_i\}_{i=1}^{n}\big]-r(x)\right|&\lesssim \delta_n+\frac{1}{np_0}+\frac{n}{M}\left( e^{-(1-\log 2)M} +e^{M-r_{\text{ratio}}np_0-M\log M+M\log (r_{\text{ratio}}np_0)}\right).
\end{align}

As for the variance, we also need to do similar modifications. By the law of total variance,
\begin{align}\label{eq:ratevar1}
  \Var\big[\hat{r}_M(x)|\{D_i\}_{i=1}^{n}\big] = \mbb{E}\Big[\Var\Big[\hat{r}_M(x) \Biggiven X,\{D_i\}_{i=1}^{n}\Big]\Big] + \Var\Big[\mbb{E}\Big[\hat{r}_M(x) \Biggiven X,\{D_i\}_{i=1}^{n}\Big]\Big].
\end{align}

For the first term in \eqref{eq:ratevar1}, we follow \citet[Proof of Theorem B.3, S3.34]{lin2023estimation} to obtain that on $I$,
\begin{align*}
  \mbb{E}\Big[\Var\Big[\hat{r}_M(x) \Biggiven X,\{D_i\}_{i=1}^{n}\Big]\Big]\le \frac{n_0}{n_1M} \E\big[\hat{r}_M(x)|\{D_i\}_{i=1}^{n}\big],
\end{align*}
where again we only used its first inequality and bound the expected value according to our modified version \eqref{mod:bias}. As for the second part, according to \citet[Proof of Theorem B.3, S3.36 and S3.37]{lin2023estimation}, we have that on $I$,
\begin{align*}
    \Var\Big[\mbb{E}\Big[\hat{r}_M(x) \Biggiven X,\{D_i\}_{i=1}^{n}\Big]\Big]\lesssim \frac{1}{M}+ \left(\frac{n}{M}\right)^2 e^{-(1-\log 2)M}.
\end{align*}
Here, again, we did not choose to use their asymptotic rate $o(n_0^{-\gamma})$ in the second inequality in \citet[Proof of Theorem B.3, S3.36]{lin2023estimation}. Instead, we used their first inequality there in the equation and it is proven to be bounded by $\left(\frac{n}{M}\right)^2 e^{-(1-\log 2)M}$ in \citet[Proof of Theorem B.3, S3.30]{lin2023estimation}.

Then, combining the bias and variance convergence rates, we have on $I$,
\begin{align*}
    &\mbb{E}\left(|\hat{r}_{M}(x)-r(x)|^2|\{D_i\}_{i=1}^{n}\right)\\
    &\lesssim \delta_n^2+\frac{1}{n^2p_0^2}+\left(\frac{n}{M}\right)^2\left(e^{-(1-\log 2)M}+e^{M-r_{\text{ratio}}np_0-M\log M+M\log (r_{\text{ratio}}np_0)}\right)^2\\
    &\quad+\frac{n}{n_1M}\Bigg(1+\delta_n+\frac{1}{np_0}+\frac{n}{M}\left(e^{-(1-\log 2)M}+e^{M-r_{\text{ratio}}np_0-M\log M+M\log (r_{\text{ratio}}np_0)}\right) \Bigg)\\
    &\quad+\frac{1}{M}+\left(\frac{n}{M}\right)^2e^{-(1-\log 2)M}.
\end{align*}
Plugging this pointwise convergence rate in \citet[Proof of Theorem B.4]{lin2023estimation}, we have that on $I$, 
\begin{align*}
    &\mbb{E}\left(|\hat{r}_{M}(X_1)-r(X_1)|^2|\{D_i\}_{i=1}^{n}\right)\\
    &\lesssim \delta_n+\frac{1}{n^2p_0^2}+\left(\frac{n}{M}\right)^2\left(e^{-(1-\log 2)M}+e^{M-r_{\text{ratio}}np_0-M\log M+M\log (r_{\text{ratio}}np_0)}\right)^2\\
    &\quad+\frac{n}{\left(n-\frac{np_0}{2}\right)M}\left(1+\delta_n+\frac{1}{np_0}+\frac{n}{M}\left(e^{-(1-\log 2)M}+e^{M-r_{\text{ratio}}np_0-M\log M+M\log (r_{\text{ratio}}np_0)}\right)\right)\\
    &\quad+\frac{1}{M}+\left(\frac{n}{M}\right)^2 e^{-(1-\log 2)M}.
\end{align*}

Now, coming back to \eqref{div:IandIc}, we plug the above bound into the first term and directly bound the second term according to \eqref{concp0t} and obtain: on $I$,
\begin{align*}
    &\mbb{E}\bigg(\left(\frac{n_1}{n_0}\right)^2\mbb{E}\bigg(\bigg(\frac{n_0}{n_1}\frac{K_{M}^{0}(1,\XDE{\cX}_{n})}{M}-\frac{\mbb{P}(D=0)}{\mbb{P}(D=1)}\frac{e(X_1)}{1-e(X_1)}\bigg)^{2}\bigg|\{D_i\}_{i=1}^{n}\bigg)\mathds{1}(n_0>0,I)\bigg)\\
    &\lesssim \frac{1}{p_0^2}\delta_n+\frac{1}{n^2p_0^4}+\left(\frac{n}{Mp_0}\right)^2\left(e^{-(1-\log 2)M}+e^{M-r_{\text{ratio}}np_0-M\log M+M\log (r_{\text{ratio}}np_0)}\right)^2\\
    &\quad+\frac{n}{p_0^2\left(n-\frac{np_0}{2}\right)M}\left(1+\delta_n+\frac{1}{np_0}+\frac{n}{M}\left(e^{-(1-\log 2)M}+e^{M-r_{\text{ratio}}np_0-M\log M+M\log (r_{\text{ratio}}np_0)}\right)\right)\\
    &\quad+\frac{1}{p_0^2M}+\left(\frac{n}{Mp_0}\right)^2 e^{-(1-\log 2)M},
\end{align*}
and on $I^c$, since $\eta\le p_0\le 1$ and $C_1\le n\eta^2$ by our assumption, we have
\begin{align*}
    &\mbb{E}\bigg(\left(\frac{n_1}{n_0}\right)^2\mbb{E}\bigg(\bigg(\frac{n_0}{n_1}\frac{K_{M}^{0}(1,\XDE{\cX}_{n})}{M}-\frac{\mbb{P}(D=0)}{\mbb{P}(D=1)}\frac{e(X_1)}{1-e(X_1)}\bigg)^{2}\bigg|\{D_i\}_{i=1}^{n}\bigg)\mathds{1}(n_0>0,I^c)\bigg)\\
    &\lesssim \frac{n^2}{Mp_0^2}e^{-\frac{np_0}{2(1-p_0)}} \le \frac{n^2}{\eta^2}e^{-\frac{n\eta}{2(1-\eta)}} = \frac{n^6\eta^6}{n^4\eta^8}e^{-\frac{n\eta}{2}} \lesssim \frac{1}{n^2\eta^4} \le \delta_{H_1}.
\end{align*}
Noting that on $I$,
\begin{align*}
    \delta_n=\left(\frac{4}{g_{0,\min}V_m}\right)^{1/m}\left(\frac{M}{n_0}\right)^{1/m}\le \left(\frac{4}{g_{0,\min}V_m}\right)^{1/m}\left(\frac{2M}{np_0}\right)^{1/m},
\end{align*}
Combining the above two results gives the desired bound.

Moreover, if we assume $M^{-1}\log n=o(1), n^{-1}M\log n=o(1)$ and that $\eta$ is bounded away from $0$, it is straightforward to see that the terms $\delta_{H_1}$-$\delta_{H_3}$ are of smaller order compared to the other summands in the first two assertions. Indeed, considering $\delta_{H_1}$ for instance, using that $\eta$ is bounded below, we have
\begin{align*}
    \delta_{H_1}&\lesssim \frac{1}{n^2}+\left(\frac{n}{M}\right)^2\left(e^{-2(1-\log 2)M}+e^{2(M-r_0n\eta-M\log M+M\log (r_0n\eta))}\right)\\
    &=\frac{1}{n^2}+\left(\frac{n}{M}\right)^2e^{-2(1-\log 2)\frac{M}{\log n}\log n}+\left(\frac{n}{M}\right)^2e^{-2n\left(r_0\eta+\frac{M\log M}{n}-\frac{M}{n}-\frac{M\log(r_0\eta n)}{n}\right)}.
\end{align*}
Now, since $M^{-1}\log n=o(1)$ and $n^{-1}M\log n=o(1)$, it follows that $\delta_{H_1}\lesssim \frac{1}{n^2}$, which is indeed a smaller order term compared to $(M/n)^{1/m}$. We can argue similarly for $\delta_{H_2}$ and $\delta_{H_3}$, yielding the final assertion.
\end{proof}

We are now ready to prove the lower bound on the variance of $E_n$, stated in Lemma~\ref{Variancelowerbound}.

\begin{proof}[Proof of Lemma \ref{Variancelowerbound}]
Recall $E_n$ from \eqref{eq2:F} and write
\begin{align*}
    E_n &= \frac{1}{n}\sum_{i=1}^{n}(\mu_{1}(X_i)-\mu_{0}(X_i))+\frac{1}{n}\sum_{i=1}^{n}(2D_i-1)\bigg(1+\frac{K^{D_i}_{M}(i,\XD{\cX_n})}{M}\bigg)\bm{\varepsilon_i}=:E_{n,1}+E_{n,2},
\end{align*}
where the first term is conditionally independent of the second term given $\XD{\cX}_n$. Thus, by the law of total variance and recalling that $\sigma_\omega(x)^2:=\mbb{E} [(Y(\omega)-\mu_{\omega}(X))^2 | X=x]$ for $\omega \in \{0,1\}$, we have
\begin{align}\label{eq:var}
    \Var E_n&=\mbb{E}\Var (E_{n,1}+E_{n,2}|\XD{\cX}_n)+\Var \mbb{E}(E_{n,1}+E_{n,2}|\XD{\cX}_n)
    \nonumber\\& =\mbb{E}\bigg(\frac{1}{n^2}\sum_{i=1}^{n}(2D_i-1)^2\bigg(1+\frac{K^{D_i}_{M}(i,\XD{\cX}_n)}{M}\bigg)^2\sigma_{D_i}(X_i)^2\bigg)+\Var E_{n,1}\nonumber\\
    &=\mbb{E}\bigg(\frac{1}{n^2}\sum_{i=1}^{n}\bigg(1+\frac{K^{D_i}_{M}(i,\XD{\cX}_n)}{M}\bigg)^2\sigma_{D_i}(X_i)^2 \bigg) + \frac{1}{n}\Var (\mu_1(X)-\mu_0(X)).
\end{align}
Note in particular that this yields the lower bound in the first assertion.

To obtain the upper bound, note that for the first term in \eqref{eq:var}, following \cite{lin2023estimation}, we can write
\begin{align*}
    &\frac{1}{n}\sum_{i=1}^{n}\bigg(1+\frac{K^{D_i}_{M}(i,\XD{\cX}_n)}{M}\bigg)^2\sigma_{D_i}(X_i)^2
    \\& = \frac{1}{n}\sum_{i=1,D_i=1}^{n}\bigg(1+\frac{K^{1}_{M}(i,\XD{\cX}_n)}{M}\bigg)^2\sigma_{1}^2(X_i) + \frac{1}{n}\sum_{i=1,D_i=0}^{n}\bigg(1+\frac{K^{0}_{M}(i,\XD{\cX}_n)}{M}\bigg)^2\sigma_{0}^2(X_i)
    \\& = \bigg(\frac{1}{n}\sum_{i=1,D_i=1}^{n}\bigg(\frac{1}{e(X_i)}\bigg)^2\sigma_{1}^2(X_i) + \frac{1}{n}\sum_{i=1,D_i=0}^{n}\bigg(\frac{1}{1-e(X_i)}\bigg)^2\sigma_{0}^2(X_i)\bigg)
    \\&\qquad +\frac{1}{n}\sum_{i=1,D_i=1}^{n}\bigg(\bigg(1+\frac{K^{1}_{M}(i,\XD{\cX}_n)}{M}\bigg)^2-\left(\frac{1}{e(X_i)}\right)^2\bigg)\sigma_{1}^2(X_i)
    \\&\qquad + \frac{1}{n}\sum_{i=1,D_i=0}^{n}\bigg(\bigg(1+\frac{K^{0}_{M}(i,\XD{\cX}_n)}{M}\bigg)^2-\bigg(\frac{1}{1-e(X_i)}\bigg)^2\bigg)\sigma_{0}^2(X_i)
    \\&=: J_1+J_2+J_3.
\end{align*}
Since $\{(X_i,D_i,Y_i)\}_{i=1}^{n}$ are i.i.d., we have
\begin{align*}
    \mbb{E}J_1= \mbb{E}\left(D \frac{\sigma_1^2(X)}{e(X)^2}+(1-D)\frac{\sigma^2_0(X)}{(1-e(X))^2}\right)=\mbb{E}\left(\frac{\sigma_1^2(X)}{e(X)}+\frac{\sigma^2_0(X)}{1-e(X)}\right).
\end{align*}
Thus from \eqref{eq:var}, with $\sigma^2$ defined at \eqref{def:sigma}, we obtain
\begin{equation}\label{eq:vardiff}
|n\Var F_n-\sigma^2|\le \mbb{E}|J_2+J_3|.    
\end{equation}
Following \citet[S4.1.\ Proof of Lemma C.1]{lin2023estimationsupp} and Assumption \ref{AssumpD}(2), we have
\begin{align}\label{J2}
    \mbb{E}|J_2|\lesssim \mbb{E}\left(D_1\left(J_{21}J_{22}\right)^{\frac{1}{2}}\right) \lesssim \mbb{E}\left(D_1\left(\sqrt{J_{21}} + J_{21}\right)\right),
\end{align}
where
\begin{align*}
    J_{21}=\mbb{E}\bigg(\bigg(\frac{K_{M}^{1}(1,\XDE{\cX}_{n})}{M}-\frac{1-e(X_1)}{e(X_1)}\bigg)^{2}\bigg|\{D_i\}_{i=1}^{n}\bigg)\mathds{1}(n_1>0),
\end{align*}
and
\begin{align*}
    J_{22}=\mbb{E}\bigg(\bigg(2+\frac{K_{M}^{1}(1,\XDE{\cX}_{n})}{M}-\frac{1-e(X_1)}{e(X_1)}\bigg)^{2}\bigg|\{D_i\}_{i=1}^{n}\bigg)\mathds{1}(n_1>0).
\end{align*}
Now using the simple inequality that $(a+b)^2 \le 2(a^2+b^2)$ for $a,b \in \mathbb{R}$, it holds that 
\begin{align*}
    J_{21}&=\left(\frac{n_0}{n_1}\right)^2\mbb{E}\bigg(\bigg(\frac{n_1}{n_0}\frac{K_{M}^{1}(1,\XDE{\cX}_{n})}{M}-\frac{n_1}{n_0}\frac{1-e(X_1)}{e(X_1)}\bigg)^{2}\bigg|\{D_i\}_{i=1}^{n}\bigg)\mathds{1}(n_1>0)\\
    &\lesssim \left(\frac{n_0}{n_1}\right)^2\mbb{E}\bigg(\bigg(\frac{n_1}{n_0}\frac{K_{M}^{1}(1,\XDE{\cX}_{n})}{M}-\frac{\mbb{P}(D=1)}{\mbb{P}(D=0)}\frac{1-e(X_1)}{e(X_1)}\bigg)^{2}\bigg|\{D_i\}_{i=1}^{n}\bigg)\mathds{1}(n_1>0)\\
    &\quad + \left(\frac{n_0}{n_1}\right)^2\left(\left(\frac{\mbb{P}(D=1)}{\mbb{P}(D=0)}\frac{1-e(X_1)}{e(X_1)}-\frac{n_1}{n_0}\frac{1-e(X_1)}{e(X_1)}\right)^{2}\right)\mathds{1}(n_1>0)\\
    &=:J_{211}+J_{212}.
\end{align*}
According to \citet{lin2023estimation}, the expression inside the expectation in $J_{211}$, i.e.,
\begin{align*}
    \frac{n_1}{n_0}\frac{K_{M}^{1}(1,\XDE{\cX}_{n})}{M}-\frac{\mbb{P}(D=1)}{\mbb{P}(D=0)}\frac{1-e(X_1)}{e(X_1)}=\frac{n_1}{n_0}\frac{K_{M}^{1}(1,\XDE{\cX}_{n})}{M}-\frac{g_{0}(X_1)}{g_{1}(X_1)}
\end{align*}
is the difference between the density ratio $g_0/g_1$ and the density ratio estimator based on two samples with sizes $n_0$ and $n_1$, respectively, with $n_0+n_1=n$; here $g_0$ and $g_1$ are defined in Assumption \ref{AssumpA}. Also, we make a convention that when $n_1=0$ or $n_0=0$, $0/0=0$ to make the bound well defined. 

Recall that $p_0:=\mbb{P}(D=0)$ and set $p_1:=\mbb{P}(D=1)$. Noting that $n_1$ is a binomal random variable with parameters $n$ and $p_1$ and $n_0$ is also a binomial random variable with parameters $n$ and $p_0$, according to Hoeffding bound for sub-Gaussian random variables (see \cite[Proposition 2.5]{wainwright2019high}), we have for all $t\ge 0$,
\begin{align}\label{concn1}
    \mbb{P}(|n_1-np_1|\ge t)&\le 2 e^{-\frac{2t^2}{np_1(1-p_1)}},\nonumber\shortintertext{or equivalently}
    \mbb{P}(|n_0-np_0|\ge t)&\le 2 e^{-\frac{2t^2}{np_0(1-p_0)}}.
\end{align}

According to Lemma \ref{lemma:l2convdensityratio}, we have
\begin{align*}
    \mbb{E}J_{211}\lesssim \frac{1}{\eta^2}\left(\frac{M}{n\eta}\right)^{1/m}+\delta_{H_1}+(\delta_{H_2}+1)\cdot\frac{1}{\eta^2 M}+\delta_{H_3}.
\end{align*}
Moreover, we have
\begin{align}\label{boundJ212}
    \mbb{E}J_{212}&=\mbb{E}\bigg(\left(\frac{n_0}{n_1}\right)^2\bigg(\left(\frac{\mbb{P}(D=1)}{\mbb{P}(D=0)}\frac{1-e(X_1)}{e(X_1)}-\frac{n_1}{n_0}\frac{1-e(X_1)}{e(X_1)}\right)^{2}\bigg)\mathds{1}(n_1>0,|n_1-np_1|<t)\bigg)\nonumber\\
    &\quad+\mbb{E}\bigg(\left(\frac{n_0}{n_1}\right)^2\bigg(\left(\frac{\mbb{P}(D=1)}{\mbb{P}(D=0)}\frac{1-e(X_1)}{e(X_1)}-\frac{n_1}{n_0}\frac{1-e(X_1)}{e(X_1)}\right)^{2}\bigg)\mathds{1}(n_1>0,|n_1-np_1|\ge t)\bigg).
\end{align}
Noting from the boundedness of the density ratio that $(1-e(X_1))/e(X_1) \le \eta^{-1}$, on $\{|n_1-np_1|<t\}$,
by setting $t=np_1/2 \wedge (np_1)^{2/3}$, we have
\begin{align*}
    &\left(\frac{n_0}{n_1}\right)^2\left(\frac{\mbb{P}(D=1)}{\mbb{P}(D=0)}\frac{1-e(X_1)}{e(X_1)}-\frac{n_1}{n_0}\frac{1-e(X_1)}{e(X_1)}\right)^{2}\\
    &=\left(\frac{n_0}{n_1}\right)^2\left(\frac{1-e(X_1)}{e(X_1)}\right)^2\left(\frac{\mbb{P}(D=1)}{\mbb{P}(D=0)}-\frac{n_1}{n_0}\right)^{2}\\
    &\lesssim \frac{1}{\eta^4 n_1^2}\left(p_1n_0 - p_0n_1\right)^2\\
    &\lesssim \frac{1}{\eta^4 n_1^2} \left[p_1^2(n_0-np_0)^2 + p_0^2(n_1-np_1)^2\right] \lesssim \frac{(np_1)^{4/3}}{\eta^4 (np_1/2)^2} \lesssim \frac{1}{\eta^6 n^{2/3}}.
\end{align*}
Also, on $\{|n_1-np_1|\ge t\}$, which occurs with probability not larger than $2e^{-\frac{2(np_1)^{1/3}}{1-p_1}}\le 2e^{-\frac{2(n\eta)^{1/3}}{1-\eta}}$ by \eqref{concn1}, we simply bound
\begin{align*}
    \left(\frac{n_0}{n_1}\right)^2\left(\frac{\mbb{P}(D=1)}{\mbb{P}(D=0)}\frac{1-e(X_1)}{e(X_1)}-\frac{n_1}{n_0}\frac{1-e(X_1)}{e(X_1)}\right)^{2}\lesssim \frac{n^2}{\eta^4}.
\end{align*}
Together, we have
\begin{align*}
    \mbb{E}J_{212}\lesssim \frac{1}{\eta^{6}n^{2/3}}+\frac{n^2}{\eta^4}e^{-\frac{2(n\eta)^{1/3}}{1-\eta}}.
\end{align*}
Note that since $C_1\le n\eta^2$, we have that 
\begin{align}\label{deltaH1inJ212}
    \frac{n^2}{\eta^4}e^{-\frac{2(n\eta)^{1/3}}{1-\eta}} = \frac{n^8\eta^8}{n^6\eta^{12}}e^{-2(n\eta)^{1/3}} \lesssim \frac{1}{n^2\eta^4} \le \delta_{H_1}.
\end{align}
Consequently, combining these bounds
, we obtain
\begin{align*}
    \mbb{E}J_{21}\lesssim \frac{1}{\eta^2}\left(\frac{M}{n\eta}\right)^{1/m}+\delta_{H_1}+(\delta_{H_2}+1)\cdot\frac{1}{\eta^2 M}+\delta_{H_3}+\frac{1}{\eta^{6}n^{2/3}}.
\end{align*}
By our assumption, the density ratio $\frac{\mbb{P}(D=1)}{\mbb{P}(D=0)}\frac{1-e(X_1)}{e(X_1)}$ is bounded and hence $J_{21}\lesssim \sqrt{J_{21}}$. Plugging the bounds in \eqref{J2} yields
\begin{align}\label{J2boundtotal}
    \mbb{E}|J_2|\lesssim \frac{1}{\eta}\left(\frac{M}{n\eta}\right)^{1/(2m)}+\delta_{H_1}^{1/2}+(\delta_{H_2}^{1/2}+1)\cdot\frac{1}{\eta M^{1/2}}+\delta_{H_3}^{1/2}+\frac{1}{\eta^{3}n^{1/3}}.
\end{align}
By symmetry, the same bound also holds for $\mbb{E}|J_3|$, which together with \eqref{eq:vardiff} yields the first upper bound.

Finally, as pointed out in Lemma \ref{lemma:l2convdensityratio}, under the assumptions $M^{-1}\log n=o(1), n^{-1}M\log n=o(1)$ and $\eta$ bounded away from $0$, the terms involving $\delta_{H_1}$-$\delta_{H_3}$ are indeed of smaller order compared to the other terms, and the bound thus simplifies to
\begin{align*}
     \left(\frac{M}{n}\right)^{1/(2m)}+\frac{1}{M^{1/2}}+\frac{1}{n^{1/3}},
\end{align*}
completing the proof.
\end{proof}

\section{Proofs for Gaussian Approximation Bounds in Section~\ref{sec:mainresult}}
\subsection{Proof of Theorem \ref{thm:covariate}}\label{pf:covariate} We follow the steps outlined in Section \ref{sec:gauss}.

    \textbf{Step (1)}. We apply Theorem \ref{rateswithconstant} to bound $I_0$ as in Section \ref{sec:gauss}. For this, we need to bound the terms $S_{i}$, $i\in [5]$.
    First, for $S_1$, according to \eqref{def:cMetap}, Lemma \ref{boundpsin} and Lemma \ref{Variancelowerbound}, we have
    \begin{align*}
        S_1&\lesssim (((\zeta\eta)^{-1}M)^{\frac{20}{8+p}}\vee 1)\cdot \alpha^{-\frac{1}{2}}((\eta^{-1}M)^{\frac{\alpha+1}{2}}\vee 1)n^{-\frac{1}{2}}\\
        &=\alpha^{-\frac{1}{2}}\cdot (((\zeta\eta)^{-1}M)^{\frac{20}{8+p}}\vee 1)\cdot ((\eta^{-1}M)^{\frac{16+3p}{32+4p}}\vee 1)\cdot n^{-\frac{1}{2}}.
    \end{align*}
    Similarly, we also have
    \begin{align*}
        S_2&\lesssim (((\zeta\eta)^{-1}M)^{\frac{20}{8+p}}\vee 1)\cdot \alpha^{-1}((\eta^{-1}M)^{\alpha+1}\vee 1)n^{-\frac{1}{2}}\\
        &=  \alpha^{-1}\cdot (((\zeta\eta)^{-1}M)^{\frac{20}{8+p}}\vee 1)\cdot ((\eta^{-1}M)^{\frac{16+3p}{16+2p}}\vee 1)\cdot n^{-\frac{1}{2}}.
    \end{align*}

    For the rest of $S_i$'s, using \eqref{eq:boundgamman} we have
    \begin{align*}
        S_3\lesssim (((\zeta\eta)^{-1}M)^{\frac{20}{8+p}}\vee 1)\cdot n^{-\frac{1}{2}}, \;\quad \;
        S_4\lesssim (((\zeta\eta)^{-1}M)^{\frac{40}{8+p}}\vee 1)\cdot n^{-\frac{1}{2}},
    \end{align*}
    and
    \begin{align*}
        S_5\lesssim (((\zeta\eta)^{-1}M)^{\frac{30}{8+p}}\vee 1)\cdot n^{-\frac{1}{2}}.
    \end{align*}
    Putting the bounds together, Theorem \ref{rateswithconstant} yields
    \begin{align*}
        I_0&={\sf d}_{K}\left(\frac{E_n-\mbb{E}E_n}{\sqrt{\Var E_n}},\mcal{N}(0,1)\right)\\
        &\lesssim \frac{\alpha^{-1} (((\zeta\eta)^{-1}M)^{\frac{20}{8+p}}\vee 1)\cdot ((\eta^{-1}M)^{\frac{16+3p}{16+2p}}\vee 1)}{n^{\frac{1}{2}}}+ \frac{(((\zeta\eta)^{-1}M)^{\frac{40}{8+p}}\vee 1)}{n^{\frac{1}{2}}}.
    \end{align*}

\textbf{Step (2).} To bound $I_1$ as in Section \ref{sec:gauss}, note that by the triangle inequality, we have
\begin{align*}
    I_1={\sf d}_{K}\left(\sqrt{n}\frac{E_{n}-\mbb{E}E_n}{\sigma},\mcal{N}(0,1)\right)\le I_0+{\sf d}_{K}\left(\frac{\sigma}{\sqrt{n\Var E_n}}\mcal{N}(0,1),\mcal{N}(0,1)\right),
\end{align*}
where $\sigma^2$ is defined at \eqref{def:sigma}. Now observing that for $a > 0$,
\begin{equation}\label{eq:dkbd1}
    \mathsf{d}_{K}(a \mcal{N},\mcal{N})\le a\vee a^{-1}-1,
\end{equation} 
we have
\begin{align*}
    I_1\le I_0+\left|\frac{\sigma}{\sqrt{n\Var E_n}}-1\right|+\left|\frac{\sqrt{n\Var E_n}}{\sigma}-1\right|.
\end{align*}
Note that if $\sigma \ge \sqrt{n\Var E_n}$, then using that $n \Var E_n$ is bounded below by a constant, we obtain
\begin{align}\label{boundratioofvariance}
\frac{\sigma}{\sqrt{n\Var E_n}}-1 = \frac{\sigma^2 - n\Var E_n}{n\Var E_n (1+\sigma/\sqrt{n\Var E_n})} \lesssim\sigma^2 - n\Var E_n\lesssim \sigma^2 - n\Var E_n.
\end{align}
The same conclusion also holds when $\sigma < \sqrt{n\Var E_n}$. Thus, according to Lemma \ref{Variancelowerbound},
\begin{align*}
    I_1\lesssim I_0+\frac{1}{\eta}\left(\frac{M}{n\eta}\right)^{1/(2m)}+\delta_{H_1}^{1/2}+(\delta_{H_2}^{1/2}+1)\cdot\frac{1}{\eta M^{1/2}}+\delta_{H_3}^{1/2}+\frac{1}{\eta^{3}n^{1/3}}.
\end{align*}

\textbf{Step (3).} Finally we bound 
$$
I_2={\sf d}_{K}\left(\sqrt{n}\frac{\hat{\tau}_{M}^{bc}-\tau}{\sigma},\mcal{N}(0,1)\right).
$$
Recall from \eqref{eq2:F} that $\hat{\tau}_{M}^{bc}=E_n+(B_M-\hat{B}_{M})=:E_n+B$. Then, for any $t\in\mbb{R}$ and $\epsilon_0>0$, we have
\begin{align*}
    &\left|\mbb{P}\left(\sqrt{n}(\hat{\tau}_{M}^{bc}-\tau)\le t\right)-\mbb{P}\left(\sqrt{n}(E_n-\tau)\le t\right)\right|\\&\le \left|\mbb{P}\left(\sqrt{n}(E_n-\tau)+\sqrt{n}B\le t\right)-\mbb{P}\left(\sqrt{n}(E_n-\tau)\le t\right)\right|
    \\&\le \mbb{P}\left(\sqrt{n}|B|> \epsilon_0\right)+2I_1+\mbb{P}(\mcal{N}(0,\sigma^2)\in (t-\epsilon_0, t-\epsilon_0]).
\end{align*}
Following \citet[Proof of Lemma C.3]{lin2023estimationsupp} under Assumption B.4 in Section \ref{AssumpB}, it holds that
\begin{align}\label{biasboundcovariate}
    \mbb{E}|B|&\lesssim \bigg(\mbb{E}\bigg(\left(\frac{n}{n_0}\right)^{k/m}\mathds{1}(n_0>0)\bigg)\vee \mbb{E}\bigg(\left(\frac{n}{n_1}\right)^{k/m}\mathds{1}(n_1>0)\bigg)\bigg)\nonumber\\
    &\qquad\bigg(\left(\frac{M}{n}\right)^{k/m}+\max_{l\in [k-1]}\bigg(n^{-\gamma_{l}}\bigg(\frac{M}{n}\bigg)^{l/m}\bigg)\bigg),
\end{align}
where $k=\lfloor m/2 \rfloor+1$ and $\gamma_l$ is defined in Assumption B.4 in Section \ref{AssumpB} for $l\in[k-1]$. Moreover, similar to bounding $J_{212}$ in \eqref{boundJ212}, setting $t=np_0/2\wedge (np_0)^{2/3} $ in \eqref{concn1}, we have
\begin{align}\label{adderrorinbias}
    \mbb{E}\bigg(\left(\frac{n}{n_0}\right)^{k/m}\mathds{1}(n_0>0)\bigg)\lesssim \eta^{-k/m}+n^{k/m}e^{-\frac{2(n\eta)^{1/3}}{1-\eta}}\lesssim \eta^{-k/m}+\delta_{H_1},
\end{align}
where we applied \eqref{deltaH1inJ212} for the last inequality. Similarly, we can have the same bound for $\mbb{E}\big(\left(n/n_1\right)^{k/m}\mathds{1}(n_1>0)\big)$ by symmetry. Then, by Markov's inequality, we have
\begin{align*}
    \mbb{P}\left(\sqrt{n}|B|> \epsilon_0\right)\lesssim (\eta^{-k/m}+\delta_{H_1})\big(M^{k/m}n^{-k/m+1/2}+\max_{l\in [k-1]}\big(n^{-\gamma_{l}-l/m+1/2}M^{l/m}\big)\big)\epsilon_0^{-1}.
\end{align*}
On the other hand, by boundedness of the Gaussian density, we have
\begin{align*}
    \mbb{P}(\mcal{N}(0,\sigma^2)\in (t-\epsilon_0,t+\epsilon_0] \lesssim \epsilon_0.
\end{align*}
Choosing $\epsilon_0>0$ optimally, we thus have that have for any $t\in\mbb{R}$,
\begin{align*}
    &\left|\mbb{P}\left(\sqrt{n}(\hat{\tau}_{M}^{bc}-\tau)\le t\right)-\mbb{P}\left(\sqrt{n}(E_n-\tau)\le t\right)\right|\\
    &\lesssim I_1+(\eta^{-k/(2m)}+\delta_{H_1}^{1/2})\big(M^{k/(2m)}n^{-k/(2m)+1/4}+\max_{l\in [k-1]}\big(n^{-\gamma_{l}/2-l/(2m)+1/4}M^{l/(2m)}\big)\big).
\end{align*}
Therefore, taking supremum over all $t \in \mbb{R}$, we can conclude that
\begin{align*}
    I_2 &\le {\sf d}_{K}\left(\sqrt{n}\frac{\hat{\tau}_{M}^{bc}-\tau}{\sigma},\sqrt{n}\frac{E_{n}-\tau}{\sigma}\right) + I_1\\
    &\lesssim I_1+(\eta^{-k/(2m)}+\delta_{H_1}^{1/2})\big(M^{k/(2m)}n^{-k/(2m)+1/4}+\max_{l\in [k-1]}\big(n^{-\gamma_{l}/2-l/(2m)+1/4}M^{l/(2m)}\big)\big)\\
    &\lesssim I_0+(\eta^{-k/(2m)}+\delta_{H_1}^{1/2})\big(M^{k/(2m)}n^{-k/(2m)+1/4}+\max_{l\in [k-1]}\big(n^{-\gamma_{l}/2-l/(2m)+1/4}M^{l/(2m)}\big)\big)\\
    &\quad +\frac{1}{\eta}\left(\frac{M}{n\eta}\right)^{1/(2m)}+\delta_{H_1}^{1/2}+(\delta_{H_2}^{1/2}+1)\cdot\frac{1}{\eta M^{1/2}}+\delta_{H_3}^{1/2}+\frac{1}{\eta^{3}n^{1/3}}\\
    &\lesssim \bigg(\frac{\alpha^{-1} (((\zeta\eta)^{-1}M)^{\frac{20}{8+p}}\vee 1)\cdot ((\eta^{-1}M)^{\frac{16+3p}{16+2p}}\vee 1)}{n^{\frac{1}{2}}}+ \frac{(((\zeta\eta)^{-1}M)^{\frac{40}{8+p}}\vee 1)}{n^{\frac{1}{2}}}\bigg)\\
    &\quad + (\eta^{-k/(2m)}+\delta_{H_1}^{1/2})\big(M^{k/(2m)}n^{-k/(2m)+1/4}+\max_{l\in [k-1]}\big(n^{-\gamma_{l}/2-l/(2m)+1/4}M^{l/(2m)}\big)\big)\\
    &\quad + \frac{1}{\eta}\left(\frac{M}{n\eta}\right)^{1/(2m)}+\delta_{H_1}^{1/2}+(\delta_{H_2}^{1/2}+1)\cdot\frac{1}{\eta M^{1/2}}+\delta_{H_3}^{1/2}+\frac{1}{\eta^{3}n^{1/3}}.
\end{align*}

\subsection{Proof of Corollary \ref{coro:covariate}}
Under the assumptions that $M^{-1}\log n=o(1)$, $n^{-1}M\log n=o(1)$ and that $\eta$ is bounded away from zero, arguing as in the proof of Lemma \ref{lemma:l2convdensityratio}, the terms involving $\delta_{H_1}$-$\delta_{H_3}$ are smaller order terms so that we can suppress them in Theorem \ref{thm:covariate}. This yields the desired simplified bound.

\subsection{Proof of Theorem \ref{thm:rank}}\label{proofofrankmain}
Now, we turn to the $\phi$-transformation rank based ATE estimator.
Following \citet[Proof of Theorem 5.1(ii), Section A.7]{cattaneo2023rosenbaum}, we can write
\begin{align*}
    \hat{\tau}_{\phi,M}^{bc}=E_{\phi,n}+(B_{\phi,n}-\hat{B}_{\phi,n}),
\end{align*} 
where
\begin{align*}
    E_{\phi,n}=\frac{1}{n}\sum_{i=1}^{n}(\mu_{\phi,1}(L_{\phi,1,i})-\mu_{\phi,0}(L_{\phi,0,i}))+\frac{1}{n}\sum_{i=1}^{n}(2D_i-1)\left(1+\frac{K_{\phi}(i)}{M}\right)\bm{\varepsilon}_{\phi,i},
\end{align*}
with $L_{\phi,\omega,i}:={\phi}_{\omega}(X_i)$ for $\omega \in \{0,1\}$ and $\bm{\varepsilon_{\phi,i}}=Y_i-\mu_{\phi,D_i}(L_{\phi,D_i,i})$ for $i\in[n]$. However, since $K_{\phi}(i)$ is defined through the ranks of $\hat{L}_{\phi,\omega,i}$'s (recall \eqref{def:kphii}), which depend on the whole data, it is not possible to express it as a sum of scores which are exponentially stabilizing, as given in \eqref{FoB}, unlike in the covariate based case. We thus consider the following modification. For $i\in[n]$, define $K_{\phi}^*(i)$ similarly as $K_{\phi}(i)$ in \eqref{def:kphii} with $\mcal{J}_{\phi,M}(i)$ replaced by the analogously defined $M$-NN matches in $L_{\phi,\omega,i}$ (instead of $\hat{L}_{\phi,\omega,i}$). We can now re-express $\hat{\tau}_{\phi,M}^{\text{bc}}$ as
\begin{align*}
    \hat{\tau}_{\phi,M}^{\text{bc}}=E_{\phi,n}^*+(B_{\phi,n}-\hat{B}_{\phi,n})+\Delta E_{\phi,n},
\end{align*}
where
\begin{align*}
    E_{\phi,n}^*=\frac{1}{n}\sum_{i=1}^{n}(\mu_{\phi,1}(L_{\phi,1,i})-\mu_{\phi,0}(L_{\phi,0,i}))+\frac{1}{n}\sum_{i=1}^{n}(2D_i-1)\bigg(1+\frac{K_{\phi}^*(i)}{M}\bigg)\bm{\varepsilon}_{\phi,i},
\end{align*}
and
\begin{align*}
    \Delta E_{\phi,n}=\frac{1}{n}\sum_{i=1}^{n}(2D_i-1)\bigg(\bigg(1+\frac{K_{\phi}(i)}{M}\bigg)-\bigg(1+\frac{K^*_{\phi}(i)}{M}\bigg)\bigg)\bm{\varepsilon}_{\phi,i}.
\end{align*}
As opposed to $E_{\phi,n}$, due to the fact that $\{L_{\phi,1,i}\}_{i=1}^{n}$ and $\{L_{\phi,0,i}\}_{i=1}^{n}$ are i.i.d.\ collection of random variables, the functional $E_{\phi,n}^*$ can indeed be written as a sum of exponentially stabilizing score functions similarly as in the covariate based case. Therefore, we can still follow a similar three-step procedure to prove Theorem \ref{thm:rank} as outlined in Section \ref{sec:gauss} and presented in the Proof of Theorem \ref{thm:covariate} for $E_{\phi,n}^*$. Thus following identical arguments as in \textbf{Step (1)} in Section \ref{pf:covariate}, we obtain that
\begin{align*}
    &{\sf d}_{K}\bigg(\frac{E_{\phi,n}^*-\mbb{E}E_{\phi,n}^*}{\sqrt{\Var E_{\phi,n}^*}},\mcal{N}(0,1)\bigg)\\
        &\lesssim \bigg(\frac{\alpha^{-1} (((\zeta\eta)^{-1}M)^{\frac{20}{8+p}}\vee 1)\cdot ((\eta^{-1}M)^{\frac{16+3p}{16+2p}}\vee 1)}{n^{\frac{1}{2}}}+ \frac{(((\zeta\eta)^{-1}M)^{\frac{40}{8+p}}\vee 1)}{n^{\frac{1}{2}}}\bigg).
\end{align*}
Also, we again have $\mbb{E}E_{\phi,n}^*=\tau$ so that
\begin{align*}
    {\sf d}_{K}\bigg(\sqrt{n}\frac{E_{\phi,n}^*-\mbb{E}E_{\phi,n}^*}{\sigma_{\phi}},\mcal{N}\bigg)={\sf d}_{K}\bigg(\sqrt{n}\frac{E_{\phi,n}^*-\tau}{\sigma_{\phi}},\mcal{N}\bigg),
\end{align*}
where $\sigma_{\phi}$ is given in \eqref{def:sigmaphi}. We thus need to only consider \textbf{Step (2)}, which is related to convergence of variance, and \textbf{Step (3)}, which involves bounding the bias term $(B_{\phi,n}-\hat{B}_{\phi,n})+\Delta E_{\phi,n}$. In order to complete these steps, it suffices to find the counterparts of Lemma \ref{lemma:l2convdensityratio} (Lemma \ref{Variancelowerbound} is a direct consequence of it) in \textbf{Step (2)} and the bound \eqref{biasboundcovariate} in \textbf{Step (3)}, which we present in the following two lemmas with their proofs included in Appendix \ref{sec:ranklemma}.

\begin{Lemma}[Counterpart of Lemma \ref{lemma:l2convdensityratio} and Lemma \ref{Variancelowerbound}]\label{rank_variancelowerbound}
Under the assumptions of Theorem \ref{thm:rank},
    \begin{align*}
    &\mbb{E}\bigg(\left(\frac{K_{\phi}(1)}{M}-\frac{1-e(X_1)}{e(X_1)}\right)^{2}\bigg|\{D_i\}_{i=1}^{n}\bigg)\mathds{1}(n_1>0)\\
    &\lesssim \frac{1}{\eta^2}\left(\frac{M}{n\eta}\right)^{\frac{1}{m'}}+\delta_{H_1}+(\delta_{H_2}+1)\cdot\frac{1}{\eta^2 M}+\delta_{H_3} +\left(\frac{M}{n}\right)^{\frac{4m}{m'}}\bigg(\frac{n^2}{M^2}\sup_{x_1,x_2\in\X}\|\hat{\phi}_{\omega}(\cdot;x_1,x_2)-\phi_{\omega}\|_{\infty}^{2m}\bigg),
    \end{align*}
    and
    \begin{align*}
    &\mbb{E}\bigg(\bigg(\frac{K_{\phi}^*(1)}{M}-\frac{1-e(X_1)}{e(X_1)}\bigg)^{2}\bigg|\{D_i\}_{i=1}^{n}\bigg)\mathds{1}(n_1>0)\lesssim \frac{1}{\eta^2}\left(\frac{M}{n\eta}\right)^{\frac{1}{m'}}+\delta_{H_1}+(\delta_{H_2}+1)\cdot\frac{1}{\eta^2 M}+\delta_{H_3},
\end{align*}
where $\hat{\phi}_{\omega}(\cdot;x_1,x_2)$ stands for the estimator constructed by inserting two new points $x_1,x_2\in \X$ into the point cloud with $D=1-\omega$. Moreover,
\begin{align*}
    |n\Var E_{\phi,n}^*-\sigma_{\phi}^2|\lesssim\frac{1}{\eta}\left(\frac{M}{n\eta}\right)^{\frac{1}{2m'}}+\delta_{H_1}^{1/2}+(\delta_{H_2}^{1/2}+1)\cdot\frac{1}{\eta M^{1/2}}+\delta_{H_3}^{1/2}+\frac{1}{\eta^{3}n^{1/3}}.
\end{align*}
\end{Lemma}

\begin{Lemma}[Counterpart of the bias bound \eqref{biasboundcovariate}]\label{rank_biasbound}
    Under the assumptions of Theorem \ref{thm:rank},
    \begin{align*}
    \mbb{E}|B_{\phi,n}-\hat{B}_{\phi,n}|&\lesssim (\eta^{-k/m'}+\delta_{H_1})\bigg(\left(\frac{M}{n}\right)^{k/m'}+n^{-k/2}+\max_{l\in [k-1]}\bigg(n^{-\gamma_{\phi,l}}\bigg(\left(\frac{M}{n}\right)^{l/m'}+n^{-l/2}\bigg)\bigg)\nonumber\\
    &\qquad \qquad+\lim_{\delta\rightarrow 0}\mbb{E}\sup_{x,y\in\mbb{X},\|\phi(x)-\phi(y)\|\le \delta}\|(\hat{\phi}-\phi)(x)-(\hat{\phi}-\phi)(x)\|_{\infty}\bigg),
\end{align*}
where $k:=\lfloor m'/2 \rfloor\vee 1 +1$ and $\gamma_{\phi,l}$'s are given in Assumption D.4 in Section \ref{AssumpD}. Additionally,
\begin{align*}
    \sqrt{n}\mbb{E}|\Delta E_{\phi,n}|
    &\lesssim \left(\frac{M}{n}\right)^{2m/m'}\cdot \bigg(\frac{n^2}{M^2}\sup_{x_1,x_2\in \X}\|\hat{\phi}_{\omega}(\cdot;x_1,x_2)-\phi_{\omega}\|_{\infty}^{2m}\bigg)^{1/2}.
\end{align*}
\end{Lemma}

Putting all the bounds obtained so far together, we have indeed derived everything needed in the three steps prescribed in Section \ref{sec:gauss}. Following similar arguments as in Section \ref{pf:covariate}, we then obtain
\begin{align*}
    &{\sf d}_{K}\left(\sqrt{n}\frac{\hat{\tau}_{\phi,M}^{bc}-\tau}{\sigma_{\phi}},\mcal{N}(0,1)\right)\\
    &\lesssim \frac{\alpha^{-1} (((\zeta\eta)^{-1}M)^{\frac{20}{8+p}}\vee 1)\cdot ((\eta^{-1}M)^{\frac{16+3p}{16+2p}}\vee 1)}{n^{\frac{1}{2}}}+ \frac{(((\zeta\eta)^{-1}M)^{\frac{40}{8+p}}\vee 1)}{n^{\frac{1}{2}}}\\
    &\quad+\frac{1}{\eta}\left(\frac{M}{n\eta}\right)^{1/(2m')}+\delta_{H_1}^{1/2}+(\delta_{H_2}^{1/2}+1)\cdot\frac{1}{\eta M^{1/2}}+\delta_{H_3}^{1/2}+\frac{1}{\eta^{3}n^{1/3}}\\
    &\quad+\left(\frac{n}{M}\right)^{m/m'}\cdot \bigg(\frac{n^2}{M^2}\mbb{E}\Big(\sup_{\omega\in\{0,1\}}\sup_{x_1,x_2\in \X}\|\hat{\phi}_{\omega}(\cdot;x_1,x_2)-\phi_{\omega}\|_{\infty}^{2m}\Big)\bigg)^{1/4}\\
    &\quad +(\eta^{-k/(2m')}+\delta_{H_1}^{1/2})\Bigg(M^{k/(2m')}n^{-k/(2m')+1/4}+\max_{l\in [k-1]}\bigg(n^{-\gamma_{\phi,l}/2+1/4}\bigg(\bigg(\frac{M}{n}\bigg)^{l/(2m')}+n^{-l/4}\bigg)\bigg)\\
    &\quad +n^{-k/4+1/4}+n
    ^{1/4}(\sup_{\omega\in\{0,1\}}\lim_{\delta\rightarrow 0}\mbb{E}\sup_{x,y\in\mbb{X},\|\phi_{\omega}(x)-\phi_{\omega}(y)\|\le \delta}\|(\hat{\phi}_{\omega}-\phi_{\omega})(x)-(\hat{\phi}_{\omega}-\phi_{\omega})(y)\|_{\infty})^{1/2}\Bigg)
\end{align*}
yielding the assertion in Theorem \ref{thm:rank}.

\subsection{Proof of Corollary \ref{coro:cdf}}\label{proofofcdf}
Recall the definition of the $\phi$-transformation rank-based ATE estimator \eqref{phirankATEest}. Setting $\phi_0=\phi_1=\mbf{F}$ and $\hat{\phi}_0=\hat{\phi}_1=\hat{\mbf{F}}_n$ as the (multivariate) distribution function and the (multivariate) empirical distribution function respectively, we obtain the CDF-rank-based ATE estimator \eqref{rankATEest}. 

According to Theorem \ref{thm:rank}, it suffices to bound the rate of convergence of $\hat{\phi}$ in $B_5$ and $B_6$. Firstly, considering $B_5$, noting that by defining $\mbf{T}_i=\mathds{1}(s\mbf{1}_m\le X_i\le t\mbf{1}_m)-\mathds{1}(s\mbf{1}_m\le X\le t\mbf{1}_m)$ coordinate-wise and using the results in \citet[Proof of Theorem 3.1, Part II]{cattaneo2023rosenbaum} for the inequality, we obtain
\begin{align*}
    &\sup_{\omega\in\{0,1\}}\lim_{\delta\rightarrow 0}\mbb{E}\sup_{x,y\in\mbb{X},\|\mbf{F}(x)-\mbf{F}(y)\|\le \delta}\|(\hat{\mbf{F}}_n-\mbf{F})(x)-(\hat{\mbf{F}}_n-\mbf{F})(y)\|_{\infty}\\
    &=\sup_{\omega\in\{0,1\}}\lim_{\delta\rightarrow 0}\int_{0}^{\epsilon}\mbb{P}\bigg(\sup_{x,y\in\mbb{X},\|\mbf{F}(x)-\mbf{F}(y)\|\le \delta}\|(\hat{\mbf{F}}_n-\mbf{F})(x)-(\hat{\mbf{F}}_n-\mbf{F})(y)\|_{\infty}>\epsilon\bigg)d\epsilon\\
    &\lesssim \lim_{\delta\rightarrow 0}\int_{0}^{\infty}e^{-\frac{\epsilon^2 n^2}{C\delta n+\epsilon n}}d\epsilon\lesssim n^{-1}.
\end{align*}
Plugging the above bound in $B_5$ yields the term $B_5'$. 

Finally, to bound $B_6$, we directly apply \citet[Proof of Theorem 3.1, Part I]{cattaneo2023rosenbaum} to obtain
\begin{align*}
    \mbb{E}\sup_{\omega\in\{0,1\}}\sup_{x_1,x_2\in \X}\|\hat{\mbf{F}}_{n}(\cdot;x_1,x_2)-\mbf{F}\|_{\infty}^{2m}\lesssim n^{-2m}.
\end{align*}
For $m\ge 3$, plugging the above result in $B_6$ yields $B_6'$. As for $m=1,2$, simply plugging the above bound in $B_6$ does not yield a decaying bound. This is because setting $\epsilon\asymp\epsilon'\asymp\delta\asymp(M/n)^{1/m'}$ in \eqref{L2x} and \eqref{withdeltabiasbound} fails to yield a tight bound for this CDF-specific case. We instead keep and re-pick $\epsilon,\epsilon',\delta$ in \eqref{L2x} and \eqref{withdeltabiasbound} with the assumption that $\eta$ is bounded away from $0$ in the following. From \eqref{L2x}, \eqref{generaleps} and \eqref{withdeltabiasbound}, we have
\begin{align*}
    B_6'&=|n\Var E_n-\sigma_{\phi}^2|+(\sqrt{n}\mbb{E}|\Delta E_{\phi,n}|)^{1/2} \\
    &\lesssim \left(\frac{M}{n}\right)^{1/(2m)}+\epsilon+\epsilon'+\frac{1}{n^{1/3}}+M^{-1/2} +(\delta\epsilon')^{-m/2}\cdot \left(\frac{n^2}{M^2}n^{-2m}\right)^{1/4}.
\end{align*}
Here, recall that
\begin{align*}
\delta\gtrsim \left(\frac{M}{np_0}\right)^{1/m}\ge \left(\frac{M}{n\eta}\right)^{1/m}.
\end{align*}
Moreover, by definition, we have $\epsilon,\epsilon'\asymp \delta$. For the particular choice of $\phi$ and $\hat{\phi}$ as CDF and eCDF respectively, we can take the infimum over all $\delta\gtrsim \left(\frac{M}{n\eta}\right)^{1/m}$ and obtain the bound
\begin{align*}
    B_6'\lesssim \left(\frac{M}{n}\right)^{1/(2m)}+\frac{1}{n^{1/3}}+M^{-1/2}+\inf_{\delta\gtrsim \left(\frac{M}{n}\right)^{1/m}}\left(\delta+\delta^{-m}\cdot M^{-1/2}n^{(1-m)/2}\right).
\end{align*}
The minimizer is the solution to the equation
\begin{align*}
    1-m\delta^{-m-1}M^{-1/2}n^{(1-m)/2}=0,
\end{align*}
i.e.,
\begin{align*}
    \delta^{*}\asymp  \left(\frac{1}{Mn^{m-1}}\right)^{\frac{1}{2m+2}},
\end{align*}
if it is in the domain. For $m=1$, since $B_4'$ is bounded by our assumption, we have
\begin{align*}
    \delta^{*}\asymp \frac{1}{M^{1/4}}\gtrsim \frac{M}{n}.
\end{align*}
Thus the infimum is attained at $\delta^{*}$ so that
\begin{align*}
    B_6'\lesssim \left(\frac{M}{n}\right)^{1/2}+\frac{1}{n^{1/3}}+M^{-1/2}+\frac{1}{M^{1/4}}.
\end{align*}
For $m=2$, again since $B_4'$ is bounded, we have
\begin{align*}
    \delta^{*}\asymp \frac{1}{M^{1/6}n^{1/6}}\gtrsim \left(\frac{M}{n}\right)^{1/2},
\end{align*}
so that the infimum is attained at $\delta^{*}$ yielding
\begin{align*}
    B_6'\lesssim \left(\frac{M}{n}\right)^{1/4}+\frac{1}{n^{1/3}}+M^{-1/2}+\left(\frac{1}{Mn}\right)^{1/6}.
\end{align*}
For $m\ge 3$, we can directly set $\delta\asymp \left(\frac{M}{n}\right)^{1/m}$. Combining all the cases above, we obtain the desired bound.

\section{Proof of Theorem \ref{thm:boot}}
We start with the case for bootstrapping the covariate based ATE estimator.
Recall the bootstrap estimator given in Section \ref{sec:boostrapresults} given by 
\begin{align*}
    \hat\tau_{M}^{\text{boot}}= \widebar{\Delta\hat{\mu}} + \frac{1}{n}\sum_{i=1}^{n}(\Delta\hat{\mu}(X_i)-\widebar{\Delta\hat{\mu}})V_i+\frac{1}{n}\sum_{i=1}^{n}(2D_i-1)\bigg(1+\frac{K^{D_i}_{M}(i,\XD{\cX_n})}{M}\bigg)\hat{R}_iW_{i},
\end{align*}
where the residual $\hat{R}_{i}$ satisfies
\begin{align*}
    Y_i=\mu_{D_i}(X_i)+\bm{\varepsilon_i}=\hat{\mu}_{D_i}(X_i)+\hat{R}_{i},\ i\in [n].
\end{align*}
Recall also that the bias-corrected estimator can be written as
$$\hat{\tau}_{M}^{\text{bc}}=\widebar{\Delta\hat{\mu}}+\frac{1}{n}\sum_{i=1}^{n}(2D_i-1)\bigg(1+\frac{K^{D_i}_{M}(i,\XD{\cX}_n)}{M}\bigg)\hat R_i.$$
Hence, we obtain
\begin{align*}
    \hat\tau_{M}^{\text{boot}}-\hat{\tau}_{M}^{\text{bc}}=\frac{1}{n}\sum_{i=1}^{n}(\Delta\hat{\mu}(X_i)-\widebar{\Delta\hat{\mu}})V_i+\frac{1}{n}\sum_{i=1}^{n}(2D_i-1)\bigg(1+\frac{K^{D_i}_{M}(i,\XD{\cX_n})}{M}\bigg)\hat{R}_i(W_{i}-1).
\end{align*}
Now observe that conditional on $\XDE{\cX}_n$, $\hat{\tau}^{\text{boot}}_{M}-\hat{\tau}_{M}^{bc}$ has zero mean and moreover, it is an average of independent normal random variables as $\{V_i\}_{i=1}^{n}$ and $\{W_{i}\}_{i=1}^{n}$ are i.i.d.\ $\mcal{N}(0,1)$ and $\mcal{N}(1,1)$ random variables, respectively, independent of each other. Thus it holds that
\begin{align}\label{K1K2K3}
    &\Var(\sqrt{n}(\hat{\tau}^{\text{boot}}_{M}-\hat{\tau}_{M}^{bc})|\XDE{\cX}_n)\nonumber\\
    &=\frac{1}{n}\sum_{i=1}^{n}(\Delta\hat{\mu}(X_i)-\widebar{\Delta\hat{\mu}})^2+\frac{1}{n}\sum_{i=1,D_i=1}^{n}\bigg(1+\frac{K^{1}_{M}(i,\XD{\cX}_n)}{M}\bigg)^2\hat{R}_{i}^2\nonumber\\
    &\quad+\frac{1}{n}\sum_{i=1,D_i=0}^{n}\bigg(1+\frac{K^{0}_{M}(i,\XD{\cX}_n)}{M}\bigg)^2\hat{R}_{i}^2\nonumber\\
    &=\frac{1}{n}\sum_{i=1}^{n}(\Delta\hat{\mu}(X_i)-\widebar{\Delta\hat{\mu}})^2+\frac{1}{n}\sum_{i=1,D_i=1}^{n}\bigg(1+\frac{K^{1}_{M}(i,\XD{\cX}_n)}{M}\bigg)^2(\bm{\varepsilon}_i+\mu_{1}(X_i)-\hat{\mu}_1(X_i))^2\nonumber\\
    &\quad +\frac{1}{n}\sum_{i=1,D_i=0}^{n}\bigg(1+\frac{K^{0}_{M}(i,\XD{\cX}_n)}{M}\bigg)^2(\bm{\varepsilon}_i+\mu_0(X_i)-\hat{\mu}_0(X_i))^2\nonumber\\[8pt]
    &=:K_1+K_2+K_3.
\end{align}
We compare this variance with 
\begin{align*}
    \sigma^2=\Var (\mu_1(X)-\mu_0(X))+\mbb{E}\left(\frac{\sigma_1^2(X)}{e(X)}+\frac{\sigma^2_0(X)}{1-e(X)}\right)
\end{align*}
from  \eqref{def:sigma}. First, for $K_1$ note that
\begin{align*}
    &|K_1-\Var (\mu_1(X)-\mu_0(X))|\\
    &=\bigg|\frac{1}{n}\sum_{i=1}^{n}(\hat{\mu}_0(X_i)-\hat{\mu}_1(X_i))^2-(\widebar{\Delta\hat{\mu}})^2-\mbb{E}(\mu_1(X)-\mu_0(X))^2+(\mbb{E}(\mu_1(X)-\mu_0(X)))^2\bigg|\\
    &\le\bigg|\frac{1}{n}\sum_{i=1}^{n}(\hat{\mu}_0(X_i)-\hat{\mu}_1(X_i))^2-\mbb{E}(\mu_1(X)-\mu_0(X))^2\bigg|+|(\widebar{\Delta\hat{\mu}})^2-(\mbb{E}(\mu_1(X)-\mu_0(X)))^2|\\
    &=:K_{11}+K_{12}.
\end{align*}
For $K_{11}$, for $i\in[n]$, writing $a_i=\hat{\mu}_0(X_i)-\mu_0(X_i)$, $b_i=\mu_0(X_i)-\mu_1(X_i)$ and $c_i=\mu_1(X_i)-\hat{\mu}_1(X_i)$ for simplicity, we have
\begin{align*}
    K_{11}&=\bigg|\frac{1}{n}\sum_{i=1}^{n}(\hat{\mu}_0(X_i)-\mu_0(X_i)+\mu_0(X_i)-\mu_1(X_i)+\mu_1(X_i)-\hat{\mu}_1(X_i))^2-\mbb{E}(\mu_1(X)-\mu_0(X))^2\bigg|\\
    &\le \bigg|\frac{1}{n}\sum_{i=1}^{n}(\mu_{1}(X_i)-\mu_{0}(X_i))^2-\mbb{E}(\mu_1(X)-\mu_0(X))^2\bigg|\\
    &\qquad+\bigg|\frac{1}{n}\sum_{i=1}^{n}a_i^2\bigg|+\bigg|\frac{1}{n}\sum_{i=1}^{n}c_i^2\bigg|+2\bigg|\frac{1}{n}\sum_{i=1}^{n}(a_ib_i+a_ic_i+b_ic_i)\bigg|.
\end{align*}
Further note that 
\begin{align*}
    \frac{1}{n}\left|\sum_{i=1}^{n}a_ib_i\right|&\le \|\hat{\mu}_0-\mu_0\|_{\infty}\cdot\frac{1}{n}\sum_{i=1}^{n}|\mu_0(X_i)-\mu_1(X_i)|,\\
    \frac{1}{n}\left|\sum_{i=1}^{n}b_ic_i\right|&\le \|\hat{\mu}_1-\mu_1\|_{\infty}\cdot\frac{1}{n}\sum_{i=1}^{n}|\mu_0(X_i)-\mu_1(X_i)|,\\
    \frac{1}{n}\left|\sum_{i=1}^{n}a_ic_i\right|&\le \|\hat{\mu}_0-\mu_0\|_{\infty}\|\hat{\mu}_1-\mu_1\|_{\infty}.
\end{align*}
We thus have that
\begin{align*}
    K_{11}&\le \left|\frac{1}{n}\sum_{i=1}^{n}(\mu_{1}(X_i)-\mu_{0}(X_i))^2-\mbb{E}(\mu_1(X)-\mu_0(X))^2\right|\\
    &\qquad+\left|\frac{1}{n}\sum_{i=1}^{n}(\hat{\mu}_{0}(X_i)-\mu_{0}(X_i))^2\right|+\left|\frac{1}{n}\sum_{i=1}^{n}(\hat{\mu}_{1}(X_i)-\mu_{1}(X_i))^2\right|\\
    &\qquad+4 \max_{\omega=0,1}\|\hat{\mu}_{\omega}-\mu_{\omega}\|_{\infty}\cdot\frac{1}{n}\sum_{i=1}^{n}|\mu_0(X_i)-\mu_1(X_i)|+2\max_{\omega=0,1}\|\hat{\mu}_{\omega}-\mu_{\omega}\|_{\infty}^2.
\end{align*}
Since both $\mu_1$ and $\mu_0$ are uniformly bounded by our assumptions, we apply the Hoeffding bound by setting $t=\|\mu_1 - \mu_0\|_\infty^2 \sqrt{n\log n}$ in \citet[Proposition 2.5]{wainwright2019high}) to obtain that with probability at least $1-2n^{-1}$, we have
\begin{align*}
    \bigg|\frac{1}{n}\sum_{i=1}^{n}(\mu_{1}(X_i)-\mu_{0}(X_i))^2-\mbb{E}(\mu_1(X)-\mu_0(X))^2\bigg|\le\|\mu_1 - \mu_0\|_\infty^2\sqrt{n^{-1}\log n},
\end{align*}
whence with probability at least $1-2n^{-1}$,
\begin{align*}
    K_{11}\le \|\mu_1 - \mu_0\|_\infty^2\sqrt{n^{-1}\log n}+4 \|\mu_1 - \mu_0\|_\infty\max_{\omega=0,1}~\|\hat{\mu}_{\omega}-\mu_{\omega}\|_{\infty}+4 \max_{\omega=0,1}~\|\hat{\mu}_{\omega}-\mu_{\omega}\|_{\infty}^2.
\end{align*}
On the other hand, again recalling the definitions of $a_i,b_i,c_i$ above, we have
\begin{align*}
    K_{12}&=\bigg|\bigg(\frac{1}{n}\sum_{i=1}^{n}(\hat{\mu}_0(X_i)-\hat{\mu}_1(X_i))\bigg)^2-(\mbb{E}(\mu_0(X)-\mu_1(X)))^2\bigg|\\
    &=\bigg|\bigg(\frac{1}{n}\sum_{i=1}^{n}(\hat{\mu}_0(X_i)-\mu_0(X_i)+\mu_0(X_i)-\mu_1(X_i)+\mu_1(X_i)-\hat{\mu}_1(X_i))\bigg)^2-(\mbb{E}(\mu_0(X)-\mu_1(X)))^2\bigg|\\
    &\le \bigg|\bigg(\frac{1}{n}\sum_{i=1}^{n}(\mu_0(X_i)-\mu_1(X_i))\bigg)^2-(\mbb{E}(\mu_0(X)-\mu_1(X)))^2\bigg|+\bigg(\frac{1}{n}\sum_{i=1}^{n}a_i\bigg)^2+\bigg(\frac{1}{n}\sum_{i=1}^{n}c_i\bigg)^2\\
    &\quad+2\bigg|\frac{1}{n}\sum_{i=1}^{n}a_i\frac{1}{n}\sum_{i=1}^{n}b_i\bigg|+2\bigg|\frac{1}{n}\sum_{i=1}^{n}b_i\frac{1}{n}\sum_{i=1}^{n}c_i\bigg|+2\bigg|\frac{1}{n}\sum_{i=1}^{n}a_i\frac{1}{n}\sum_{i=1}^{n}c_i\bigg|.
\end{align*}
Arguing again similarly as above, using the Hoeffding bound along with uniform boundedness of $\mu_0$ and $\mu_1$, it follows that with probability at least $1-2n^{-1}$,
\begin{align*}
    \bigg|\left(\frac{1}{n}\sum_{i=1}^{n}(\mu_0(X_i)-\mu_1(X_i))\right)^2-(\mbb{E}(\mu_0(X)-\mu_1(X)))^2\bigg| \le\sqrt{2}\|\mu_0-\mu_1\|_{\infty}^2 \sqrt{n^{-1}\log n}.
\end{align*}
Also, it holds almost surely that
\begin{align*}
    &\bigg(\frac{1}{n}\sum_{i=1}^{n}a_i\bigg)^2+\bigg(\frac{1}{n}\sum_{i=1}^{n}c_i\bigg)^2+2\bigg|\frac{1}{n}\sum_{i=1}^{n}a_i\frac{1}{n}\sum_{i=1}^{n}b_i\bigg|+2\bigg|\frac{1}{n}\sum_{i=1}^{n}b_i\frac{1}{n}\sum_{i=1}^{n}c_i\bigg|+2\bigg|\frac{1}{n}\sum_{i=1}^{n}a_i\frac{1}{n}\sum_{i=1}^{n}c_i\bigg|\\
    &\le 4\|\mu_0-\mu_1\|_{\infty}\max_{\omega=0,1}~\|\hat{\mu}_{\omega}-\mu_{\omega}\|_{\infty}+4 \max_{\omega=0,1}\|\hat{\mu}_{\omega}-\mu_{\omega}\|_{\infty}^2.
\end{align*}
Together, we obtain that with probability at least $1-2n^{-1}$,
\begin{align*}
    K_{12}\le \sqrt{2}\|\mu_0-\mu_1\|_{\infty}^2 \sqrt{n^{-1}\log n}+4\|\mu_0-\mu_1\|_{\infty}\max_{\omega=0,1}~\|\hat{\mu}_{\omega}-\mu_{\omega}\|_{\infty}+4\max_{\omega=0,1}\|\hat{\mu}_{\omega}-\mu_{\omega}\|_{\infty}^2,
\end{align*}
which in turn yields, putting the bounds on $K_{11}$ and $K_{12}$ together and noting the fact that $\|\mu_0-\mu_1\|_{\infty}$ is bounded by our assumption, that with probability at least $1-4n^{-1}$,
\begin{align*}
    &|K_1-\Var (\mu_1(X)-\mu_0(X))|\nonumber\\
    &\le 
    3 \|\mu_0-\mu_1\|_{\infty}^2 \sqrt{n^{-1}\log n}+8\|\mu_1 - \mu_0\|_\infty\max_{\omega=0,1}~\|\hat{\mu}_{\omega}-\mu_{\omega}\|_{\infty}+8\max_{\omega=0,1}~\|\hat{\mu}_{\omega}-\mu_{\omega}\|_{\infty}^2\nonumber\\
    &\lesssim\sqrt{n^{-1}\log n}+\max_{\omega=0,1}~(\|\hat{\mu}_{\omega}-\mu_{\omega}\|_{\infty}+\|\hat{\mu}_{\omega}-\mu_{\omega}\|_{\infty}^2).
\end{align*}
For $K_2$, we have
\begin{align*}
    &\left|K_2-\mbb{E}\left(\frac{\sigma_1^2(X)}{e(X)}\right)\right|\\
    &\le \bigg|\frac{1}{n}\sum_{i=1,D_i=1}^{n}\bigg(1+\frac{K^{1}_{M}(i,\XD{\cX}_n)}{M}\bigg)^2\bm{\varepsilon}_i^2-\mbb{E}\left(\frac{\sigma_1^2(X)}{e(X)}\right)\bigg|+\frac{1}{n}\sum_{i=1,D_i=1}^{n}\bigg(1+\frac{K^{1}_{M}(i,\XD{\cX}_n)}{M}\bigg)^2\|\mu_1-\hat{\mu}_1\|_{\infty}^2\\
    &\quad+\frac{2}{n}\sum_{i=1,D_i=1}^{n}\bigg(1+\frac{K^{1}_{M}(i,\XD{\cX}_n)}{M}\bigg)^2|\bm{\varepsilon}_i|\|\mu_1-\hat{\mu}_1\|_{\infty}\\
    &=:K_{21}+K_{22}+K_{23}.
\end{align*}
We can further bound $K_{21}$ as
\begin{align*}
    K_{21}&\le \bigg|\frac{1}{n}\sum_{i=1,D_i=1}^{n}\bigg(\bigg(1+\frac{K^{1}_{M}(i,\XD{\cX}_n)}{M}\bigg)^2-\left(\frac{1}{e(X_i)}\right)^2\bigg)\bm{\varepsilon}_i^2\bigg|\\
    &\quad+\bigg|\frac{1}{n}\sum_{i=1,D_i=1}^{n}\left(\frac{1}{e(X_i)}\right)^2\bm{\varepsilon}_i^2-\mbb{E}\left(\frac{\sigma_1^2(X)}{e(X)}\right)\bigg|
    =:K_{211}+K_{212}.
\end{align*}
According to the bound for $\mbb{E}|J_2|$ in the proof of Lemma \ref{Variancelowerbound}, we have
\begin{align*}
    \mbb{E}K_{211}\lesssim B_3^2.
\end{align*}
Then, it implies with probability at least $1-(B_3 \wedge 1)$, we have
\begin{align*}
    K_{211}\lesssim B_3.
\end{align*}
Since $\mbb{E}|\bm{\varepsilon}|^{4+p}$ is bounded by our assumption, by Markov's inequality, it holds that for any $\epsilon>0$,
\begin{align*}
    \mbb{P}(K_{212}\ge\epsilon)\le \frac{\mbb{E}K_{212}^2}{\epsilon^2}\le \frac{\Var(e(X_1)^{-2}\bm{\varepsilon}_1^2\mathds{1}(D_1=1))}{n\epsilon^2}\lesssim \frac{1}{n\eta^4\epsilon^2}.
\end{align*}
We thus have that with probability at least $1-n^{-1/2}$,
\begin{align*}
    K_{212}\lesssim \eta^{-2}n^{-1/4}.
\end{align*}
Combining the bounds for $K_{211}$ and $K_{212}$, it holds that with probability at least $1-(B_3+n^{-1/2}) \wedge 1\ge 1-(2B_3)\wedge 1$,
\begin{align*}
    K_{21}\lesssim B_3+\eta^{-2}n^{-1/4}.
\end{align*}
Next, for $K_{22}$, note that
\begin{align*}
    K_{22}
    &\le \bigg|\frac{1}{n}\sum_{i=1,D_i=1}^{n}\bigg(\bigg(1+\frac{K^{1}_{M}(i,\XD{\cX}_n)}{M}\bigg)^2-\bigg(\frac{1}{e(X_i)}\bigg)^{2}\bigg)\bigg| \|\mu_1-\hat{\mu}_1\|_{\infty}^2\\
    &
    \qquad+\frac{1}{n}\sum_{i=1,D_i=1}^{n}\left(\frac{1}{e(X_i)}\right)^2\|\mu_1-\hat{\mu}_1\|_{\infty}^2.
\end{align*}
Similar to $K_{21}$, we have that with probability at least $1-(B_3 \wedge 1)$,
\begin{align*}
    K_{22}\lesssim (B_3+\eta^{-2})\|\mu_1-\hat{\mu}_1\|_{\infty}^2.
\end{align*}
As for $K_{23}$, arguing similar to $K_{21}$, we obtain that with probability at least $1-(B_3+n^{-1/2}) \wedge 1\ge 1-(2B_3)\wedge 1$,
\begin{align*}
    K_{23}\lesssim (B_3+\eta^{-1}+\eta^{-2}n^{-1/4})\|\mu_1-\hat{\mu}_1\|_{\infty}.
\end{align*}
Putting all the above bounds together yields that with probability at least $1- (5B_3) \wedge 1$,
\begin{align*}
    \left|K_2-\mbb{E}\left(\frac{\sigma_1^2(X)}{e(X)}\right)\right|&\lesssim (\max_{\omega=0,1}\|\mu_{\omega}-\hat{\mu}_{\omega}\|_{\infty}^2+1)B_3+\eta^{-2}(n^{-1/4}\vee\max_{\omega=0,1}\|\mu_{\omega}-\hat{\mu}_{\omega}\|^2)\\
&\qquad+\eta^{-1}\max_{\omega=0,1}\|\mu_{\omega}-\hat{\mu}_{\omega}\|.
\end{align*}
By symmetry, the same bound also holds for $K_3$ with high probability. We thus have that with probability at least $1-(10B_3+ 4n^{-1/2}) \wedge 1\ge 1-(14B_3)\wedge 1$,
\begin{align}\label{bootvariancebound}
    |\Var(\sqrt{n}(\hat{\tau}^{\text{boot}}_{M}-\hat{\tau}_{M}^{bc})|\XDE{\cX}_n)-\sigma^2|
    &\lesssim \eta^{-1}\max_{\omega=0,1}\|\hat{\mu}_{\omega}-\mu_{\omega}\|_{\infty}+(\max_{\omega=0,1}\|\mu_{\omega}-\hat{\mu}_{\omega}\|_{\infty}^2+1)B_3\nonumber\\
    &\quad+\eta^{-2}(n^{-1/4}+\max_{\omega=0,1}\|\mu_{\omega}-\hat{\mu}_{\omega}\|_{\infty}^2).
\end{align}
Finally, we can write
\begin{align}\label{eq:varbootsum}
    &\mathsf{d}_{K}(\sqrt{n}(\hat{\tau}^{\text{boot}}_{M}-\hat{\tau}_{M}^{bc})|\XDE{\cX}_n,\sqrt{n}(\hat{\tau}_{M}^{bc}-\tau))\nonumber\\
    &=\sup_{t\in\mbb{R}}~|\mbb{P}(\sqrt{n}(\hat{\tau}^{\text{boot}}_{M}-\hat{\tau}_{M}^{bc})\le t|\XDE{\cX}_n)-\mbb{P}(\sqrt{n}(\hat{\tau}_{M}^{bc}-\tau)\le t)|\nonumber\\
    &\le \mathsf{d}_{K}(\mcal{N}(0,\Var(\sqrt{n}(\hat{\tau}^{\text{boot}}_{M}-\hat{\tau}_{M}^{bc})|\XDE{\cX}_n)),\mcal{N}(0,\sigma^2))+\mathsf{d}_{K}(\mcal{N}(0,\sigma^2),\sqrt{n}(\hat{\tau}_{M}^{bc}-\tau)).
\end{align}
Now, in order to apply \eqref{eq:dkbd1} and \eqref{boundratioofvariance}, we need to obtain a lower bound to the conditional variance above. From the the variance decomposition in \eqref{K1K2K3}, we can almost surely lower bound
\begin{multline*}
    \Var(\sqrt{n}(\hat{\tau}^{\text{boot}}_{M}-\hat{\tau}_{M}^{bc})|\XDE{\cX}_n)
    \ge K_2+K_3\ge \frac{1}{n}\sum_{i=1}^{n}\hat{R}_i^{2}\\
    =\frac{1}{n}\sum_{i=1}^{n}(\bm{\varepsilon}_i+\mu_{D_i}(X_i)-\hat{\mu}_{D_i}(X_i))^2
    \ge\frac{1}{n}\sum_{i=1}^{n}\bm{\varepsilon}_i^2-\frac{2}{n}\bigg|\sum_{i=1}^{n}\bm{\varepsilon_i}(\mu_{D_i}(X_i)-\hat{\mu}_{D_i}(X_i))\bigg|.
\end{multline*}
By our Assumption B.2 in Section \ref{AssumpB}, assuming without loss of generality that $M_{u,p} \ge 1$, we have $\mbb{E}|\bm{\varepsilon}_i|^{4}\le M_{u,p}$ for all $i\in [n]$. By Markov's inequality, we thus have that for any $\epsilon>0$,
\begin{align*}
    \mbb{P}\bigg(\bigg|\frac{1}{n}\sum_{i=1}^{n}(\bm{\varepsilon}_i^2-\mbb{E}\bm{\varepsilon}_i^2)\bigg|\ge \epsilon\bigg)\le \frac{M_{u,p}}{n\epsilon^2}.
\end{align*}
Therefore, by setting $\epsilon$ such that $\frac{M_{u,p}}{n\epsilon^2}=n^{-1/3}$, it yields that with probability at least $1-n^{-1/3}\ge 1-(B_3\wedge 1)$,
\begin{align*}
    \frac{1}{n}\sum_{i=1}^{n}\bm{\varepsilon}_i^2\ge M_{l}-\sqrt{M_{u,p}}\ n^{-1/3}.
\end{align*}
Similarly, one has that with probability at least $1-n^{-1/3}\ge 1-(B_3 \wedge 1)$,
\begin{align*}
    \frac{1}{n}\sum_{i=1}^{n}|\bm{\varepsilon}_i|\le M_{u,p}+(2M_{u,p})^{1/4}n^{-5/12}.
\end{align*}
Together, it yields that with probability at least $1-(2B_3)\wedge 1$,
\begin{align*}
    \Var(\sqrt{n}(\hat{\tau}^{\text{boot}}_{M}-\hat{\tau}_{M}^{bc})|\XDE{\cX}_n)\ge L(\hat{\mu},\mu,n),
\end{align*}
where
\begin{align*}
    L(\mu,\hat{\mu},n)&:=\big(M_{l}-\sqrt{M_{u,p}}\ n^{-1/3}-2\max_{\omega=0,1}\|\hat{\mu}_{\omega}-\mu_{\omega}\|_{\infty}\big(M_{u,p}+(2M_{u,p})^{1/4}n^{-5/12}\big)\big)\vee 0.
\end{align*}

Now using \eqref{eq:varbootsum}, Theorem \ref{thm:covariate} as well as \eqref{eq:dkbd1} and \eqref{boundratioofvariance} along with \eqref{bootvariancebound} and the variance lower bound above, we obtain that with probability at least $1-(16B_3)\wedge 1$,
\begin{align*}
    &\mathsf{d}_{K}(\sqrt{n}(\hat{\tau}^{\text{boot}}_{M}-\hat{\tau}_{M}^{bc})|\XDE{\cX}_n,\sqrt{n}(\hat{\tau}_{M}^{bc}-\tau))\\
    &\lesssim B_1+B_2+\frac{(1+\max_{\omega=0,1}\|\mu_{\omega}-\hat{\mu}_{\omega}\|_{\infty}^2) B_3}{L(\mu,\hat{\mu},n)}+\frac{\eta^{-1}\max_{\omega=0,1}\|\hat{\mu}_{\omega}-\mu_{\omega}\|_{\infty}+\eta^{-2}n^{-1/4}}{L(\mu,\hat{\mu},n) },
\end{align*}
yielding the first assertion, where we have used the fact that the Kolmogorov distance is always smaller than or equal to $1$ to simplify the bound.

The proof for the rank-based case follows mutatis mutandis the above proof and is, therefore, omitted.

\section{Proof of Theorem \ref{rateswithconstant}}
As mentioned in Remark \ref{rmk:relationwithgeneral}, the proof closely follows \citet[Proof of Theorem 4.2]{lachieze2019normal} while additionally keeping track of the quantity $c(M,\eta,p)$ (recall Remark~\ref{rmk:relationwithgeneral}). We first utilize the following bound modified from \citet[Lemma 5.5]{bhattacharjee2022gaussian}; see also \citet[Lemma 5.6]{lachieze2019normal}.
\begin{Lemma}\label{MBofFn}
    For $E_n$ defined at \eqref{eq:E_n}, there exists a constant $C>0$ such that for $p\in (0,1]$ and $\zeta=p/(40+10p)$,
    \begin{align*}
        \mbb{E}|\mathsf{D}_{\xde{x}}(nE_n)(\XDE{\mcal{X}}_{n-1-|\XDE{\mcal{A}}|}\cup \XDE{\mcal{A}})|^{4+p/2}
        \le C M_n(\xde{x},\XDE{\mcal{A}})^{4+p/2} (1+(\zeta\eta)^{-5}M^5)
    \end{align*}
    for all $\xde{x} \in\XDE{\mbb{X}}$ with $\XDE{\mcal{A}}\subset\XDE{\X}$, $|\XDE{\mcal{A}}|\le 1$ and $n\ge 9$, where we let
    $$
    M_{n}(\xde{x},\XDE{\mcal{A}}):=1+|\varepsilon|+\sum_{(x_k, d_k, \varepsilon_k)\in\mcal{A}}|\varepsilon_k|.
    $$
\end{Lemma}

\begin{proof}[Proof of Lemma~\ref{MBofFn}]
We argue as in \cite{lachieze2019normal}. Let $\XDE{\cX}_{n,\XDE{A}}=\XDE{\mcal{X}}_{n-1-|\XDE{\mcal{A}}|}\cup \XDE{\mcal{A}}$. Recalling the score function $\xi_n$ in \eqref{xin} and using the definition of the add-one cost operator followed by an application of Jensen's inequality, we have
\begin{align}\label{eq:momsum}
	\mbb{E}\big|\mathsf{D}_{\xde{x}}(nE_n)(\XDE{\cX}_{n,\XDE{\mcal{A}}})\big|^{4+p/2}
	&=\mbb{E}\bigg|\xi_n(\xde{x},\XDE{\cX}_{n,\XDE{\mcal{A}}}\cup\{\xde{x}\})
	+\sum_{\xde{y}\in \XDE{\cX}_{n,\XDE{\mcal{A}}}} \mathsf{D}_{\xde{x}} \xi_{n}(\xde{y}, \XDE{\cX}_{n,\XDE{\mcal{A}}})\bigg|^{4+p/2} \nonumber\\
	&\leq 3^{3+p/2} \mbb{E}\bigg|\xi_{n}(\xde{x}, \XDE{\cX}_{n,\XDE{\mcal{A}}}\cup\{\xde{x}\})\bigg|^{4+p/2}+3^{3+p/2}\mbb{E}\bigg|\sum_{\xde{y} \in \XDE{\cX}_{n-1-|\XDE{\mcal{A}}|}} \mathsf{D}_{\xde{x}} \xi_{n}(\xde{y},\XDE{\cX}_{n,\XDE{\mcal{A}}})\bigg|^{4+p/2}\nonumber\\
    &\qquad +3^{3+p/2}\sum_{\xde{y} \in \XDE{\mcal{A}}} \mbb{E}|\mathsf{D}_{\xde{x}} \xi_{n}(\xde{y},\XDE{\cX}_{n,\XDE{\mcal{A}}})|^{4+p/2}.
\end{align}
Let us first verify the moment condition \eqref{momentcondition} for $\xi_n$. Note from \eqref{xin} that for $i\in[n]$,
\begin{align*}
    \xi_n(\xde{X_j}, \XDE{\cX_n})&:=(\mu_1(X_j)-\mu_0(X_j))+(2D_j-1)\bm{\varepsilon}_{j} +\frac{1}{M}(1-2D_j)\hspace{-.2cm}\sum_{i=1,D_i=1-D_j}^{n}\bm{\varepsilon}_i \mathds{1}(i\in \mcal{J}^{1-D_j}_{M}(j,\XDE{\cX_n})).
\end{align*}
According to Assumption A.1 in Section \ref{AssumpA} and Assumption B.3 in Section \ref{AssumpB}, the functions $\mu_0$ and $\mu_1$ being continuous on compact supports,  are uniformly bounded. 
Thus, Assumption B.2 in Section \ref{AssumpB} yields that for any $p\in (0,1]$, $\XDE{\mcal{A}}\subset \XDE{\X}$ with $|\XDE{\mcal{A}}|\le 7$ and $\xde{x} \in \XDE{X}$,
\begin{align}\label{boundxinwithA}
    \big(\mbb{E}|\xi_n(\xde{x}, \XDE{\cX}_{n,\XDE{\mcal{A}}}\cup\{\xde{x}\})|^{4+p}\big)^{\frac{1}{4+p}} \lesssim 1+|\varepsilon|+\sum_{(x_k, d_k, \varepsilon_k)\in\XDE{\mcal{A}}}|\varepsilon_k|=:M_{n}(\xde{x},\XDE{\mcal{A}}).
\end{align}
Here, we observe that the bound $M_{n}(\xde{x},\XDE{\mcal{A}})$ for $\xde{x}=(x,d,\varepsilon)$ in the moment condition depends only on $\varepsilon$ and $\varepsilon_k$'s associated with the point $\xde{X}$ and the points in the set $\XDE{\mcal{A}}$, respectively, and it is non-decreasing in its second argument $\XDE{\mcal{A}}$. This yields that the first summand on the right-hand side of \eqref{eq:momsum} is bounded by
$C \cdot M_n(\xde{x},\XDE{\mcal{A}})^{4+p/2}$ for some $C>0$. Arguing analogously as in
\citet[Lemma~5.6]{lachieze2019normal}, the second summand can be bounded as
\begin{align*}
    3^{3+p/2}\mbb{E}\bigg|\sum_{\xde{y} \in \XDE{\cX}_{n-1-|\XDE{\mcal{A}}|}} \mathsf{D}_{\xde{x}} \xi_{n}(\xde{y},\XDE{\cX}_{n,\XDE{\mcal{A}}})\bigg|^{4+p/2}\le 3^{3+p/2}(I_1 + 15I_2 + 25I_3 + 10I_4 + I_5),
\end{align*}
where for $i \in \{1,\dots,5\}$, we let

\begin{align*}
	I_{i}=\mbb{E} \sum_{(\xde{y}_{1},\dots,\xde{y}_{i} )\in \XDE{\cX}_{n-1-|\XDE{\mcal{A}}|}^{i,\neq}}
	\mathds{1}_{\textsf{D}_{\xde{x}}\xi_{n}(\xde{y}_{j}, \XDE{\cX}_{n,\XDE{\mcal{A}}})\neq 0, j \in [i]}
	\big|\mathsf{D}_{\xde{x}} \xi_{n}(\xde{y}_{1}, \XDE{\cX}_{n,\XDE{\mcal{A}}})\big|^{4+p/2}.
\end{align*}
Here $\XDE{\cX}_{n-1-|\XDE{\mcal{A}}|}^{i, \neq}$ stands for the set of all $i$-tuples of distinct
points from $\XDE{\cX}_{n-1-|\XDE{\mcal{A}}|}$, where multiple points at the same
location are considered to be different ones. Then an application of H\"{o}lder's inequality yields
\begin{align*}
	I_i &\le n^{i} \int_{\XDE{\mbb{X}}^{i}}\mbb{E}\Big[\mathds{1}_{\textsf{D}_{\xde{x}}
		\xi_{n}(\xde{y}_{j}, \XDE{\cX}_{n-i,\XDE{\mcal{A}}}\cup\{\xde{y}_1,\ldots,\xde{y}_i\})\neq 0, j\in [i]} 
	\big|\mathsf{D}_{\xde{x}} \xi_n (\xde{y}_1, \XDE{\cX}_{n-i,\XDE{\mcal{A}}}\cup\{\xde{y}_1,\ldots,\xde{y}_i\})\big|^{4+p/2} \Big] \XDE{\mbb{Q}}^{i}(d(\xde{y}_{1},\dots,\xde{y}_{i}))\\
	&\lesssim n^{i} \int_{\XDE{\mbb{X}}^{i}} M_n(\xde{y}_1,\XDE{\mcal{A}}\cup\{\xde{x}, \xde{y}_2,\ldots,\xde{y}_i\})^{4+p/2} 
	\prod_{j=1}^{i} \mbb{P}\left(\mathsf{D}_{\xde{x}} \xi_n (\xde{y}_j, \XDE{\cX}_{n-i,\XDE{\mcal{A}}}\cup\{\xde{y}_1,\ldots,\xde{y}_i\})\neq 0\right)^{\frac{p-p/2}{4 i+p i}}\\
    &\qquad\qquad \quad \XDE{\mbb{Q}}^{i}(d(\xde{y}_{1},\dots,\xde{y}_{i}))\\
	& \lesssim n^{i} \int_{\XDE{\mbb{X}}^{i}} M_n(\xde{y}_1,\XDE{\mcal{A}}\cup\{\xde{x}, \xde{y}_2,\ldots,\xde{y}_i\})^{4+p/2} 
	\\
    & \qquad \qquad \times \prod_{j=1}^{i} \mbb{P}\left( d(x, {y}_j) \le R_n(\xde{y}_j,\XDE{\cX}_{n-i,\XDE{\mcal{A}}}\cup\{\xde{y}_j\})  \right)^{\frac{p-p/2}{4 i+p i}}
	\XDE{\mbb{Q}}^{i}(d(\xde{y}_{1},\dots,\xde{y}_{i})),
\end{align*}
where in the last step, we have used that the radius of stabilization in \eqref{eq:RoS} is non-increasing in its second argument and satisfies (see \citet[Lemma 5.3]{lachieze2019normal})
$$
\mathsf{D}_{\xde{x}} \xi_n (\xde{y}, \mu \cup\{\xde{y}\})\neq 0 \implies d_S(\xde{x},\xde{y}) = d(x,y) \le R_n(\xde{y},\mu \cup\{\xde{y}\}).
$$
Now by \eqref{boundxinwithA}, we have
\begin{align*}
    M_n(\xde{y}_1,\XDE{\mcal{A}}\cup\{\xde{x}, \xde{y}_2,\ldots,\xde{y}_i\})\lesssim 1+ |\varepsilon|+\sum_{j=1}^{i}|\varepsilon_j|+\sum_{(x_k,d_k,\varepsilon_k)\in\mcal{A}}|\varepsilon_k|.
\end{align*}
Note from Lemma \ref{lem:tb} that $\mbb{P}\left(d({x}, {y}_j) \le R_n(\xde{y}_j,\XDE{\cX}_{n-i,\XDE{\mcal{A}}}\cup\{\xde{y}_j\})\right)$ can be upper bounded by a quantity that does not involve $(d_j,\varepsilon_j)_{j=1}^i$. In addition, we can integrate over $(d_j,\varepsilon_j)_{j=1}^i$ and due to our Assumption B.2 in Section \ref{AssumpB} obtain that
\begin{align*}
    I_i&\lesssim n^{i} \int_{\mbb{X}^{i}} \bigg(1+|\varepsilon|+\sum_{j=1}^{i}|\varepsilon_j|+\sum_{(x_k,d_k,\varepsilon_k)\in\mcal{A}}|\varepsilon_k|\bigg)^{4+p/2} 
	\prod_{j=1}^{i} \exp\Big\{-C_2\frac{p/2}{4 i+p i}
	nd({y}_j,{x})^m\Big\}\\
    &\qquad\qquad\mbb{Q}^{i}(d(y_{1},\dots,y_{i}))\\
    & \lesssim n^{i} \int_{\mbb{X}^{i}} M_n(\xde{x},\XDE{\mcal{A}})^{4+p/2} 
	\prod_{j=1}^{i} \exp\Big\{-C_2\frac{p/2}{4 i+p i}
	nd({y}_j,{x})^m\Big\}\mbb{Q}^{i}(d(y_{1},\dots,y_{i}))\\
    &\le M_n(\xde{x},\XDE{\mcal{A}})^{4+p/2} (1+g_n(\xde{x})^5),
\end{align*}
where we have defined
\begin{align*}
    g_n(\xde{x}):=n\int_{\mbb{X}} 
	\exp\Big\{-\zeta C_2
	nd({y},{x})^m\Big\} \mbb{Q}(dy).
\end{align*}
Moreover, by Lemma \ref{generalintegral}, we have $g_{n}(\xde{x}) \lesssim (\zeta C_2)^{-1} \lesssim (\zeta\eta)^{-1}M$.
Combining with the definition of $M_n(\xde{x},\XDE{\mcal{A}})$, the above bound yields that for $i=1,\ldots,5$,
\begin{align*}
    I_i\lesssim M_n(\xde{x},\XDE{\mcal{A}})^{4+p/2} (1+(\zeta\eta)^{-5}M^5).
\end{align*}
This implies the second summand in \eqref{eq:momsum} is bounded by 
\begin{align*}
    3^{3+p/2}\mbb{E}\bigg|\sum_{\xde{y} \in \XDE{\cX}_{n-1-|\XDE{\mcal{A}}|}} \mathsf{D}_{\xde{x}} \xi_{n}(\xde{y},\XDE{\cX}_{n,\XDE{\mcal{A}}})\bigg|^{4+p/2}\lesssim M_n(\xde{x},\XDE{\mcal{A}})^{4+p/2} (1+(\zeta\eta)^{-5}M^5).
\end{align*}
Lastly for the third summand in \eqref{eq:momsum}, by Jensen inequality and H\"{o}lder's inequality, we have for $\xde{y} \in \XDE{\mcal{A}}$ that
\begin{align*}
    &\mbb{E}|\mathsf{D}_{\xde{x}} \xi_{n}(\xde{y},\XDE{\cX}_{n,\XDE{\mcal{A}}})|^{4+p/2}\\
    &\lesssim \mbb{E}(|\xi_{n}(\xde{y},\XDE{\cX}_{n,\XDE{\mcal{A}}}\cup\{\xde{x}\})|^{4+p/2}+|\xi_{n}(\xde{y},\XDE{\cX}_{n,\XDE{\mcal{A}}})|^{4+p/2})
    \le M_{n}(\xde{x} ,\XDE{\mcal{A}})^{4+p/2}.
\end{align*}
Combining all these three bounds yields via \eqref{eq:momsum} that
\begin{align*}
    &\mbb{E}|\mathsf{D}_{\xde{x}}(nE_n)(\XDE{\mcal{X}}_{n-1-|\XDE{\mcal{A}}|}\cup \XDE{\mcal{A}})|^{4+p/2}\\
    &\lesssim M_n(\xde{x},\XDE{\mcal{A}})^{4+p/2} (1+(\zeta\eta)^{-5}M^5)\lesssim \bigg(\sum_{(x_k,d_k,\varepsilon_k)\in\XDE{\mcal{A}}}|\varepsilon_k|^{4+p/2} + |\varepsilon|^{4+p/2}+1\bigg)(1+(\zeta\eta)^{-5}M^5).
\end{align*}
\end{proof}

 \citet[Theorem 5.1]{lachieze2017new} along with \citet[Remark 5.2 and Proposition 5.3]{lachieze2017new} (see also \citet[Theorem 4.3]{lachieze2019normal}) provides the following theorem, which serves as a key ingredient in deriving Gaussian approximation bounds for i.i.d.\ input (binomial point process) and can be viewed as a counterpart of the well-known second order Poincar\'{e} inequality for functionals of Poisson point processes. Indeed, we will use this result to prove Theorem \ref{rateswithconstant}. To state it, we first need to introduce some notation. 
 
 For an $n$-dimensional random vector $U$ with i.i.d.\ coordinates, let $U',U''$ be independent copies of $U$. We say a random vector $V$ is a recombination of $\{U,U',U''\}$ if for $i\in[n]$, $V_i\overset{a.s.}{=}U_i\ \text{or}\ U_{i}'\ \text{or}\ U''_{i}$. Also, for a vector $u=(u_1,\ldots,u_n)$ and distinct indices $I=(i_1,\ldots,i_n)\subset[n]$, denote by $u^{I}$ the subvector with the coordinates corresponding to $I$ removed. For a symmetric function $f$ defined on a point cloud $\{u_1,\ldots,u_n\}$, we extend the notation $f(u_1,\ldots,u_n):=f(\{u_1,\ldots,u_n\})$. We write $\mcal{U}_{n}:=\{U_1,\ldots,U_n\}$ and for $i,j\in [n]$, we define the index derivatives
\begin{align*}
    \mathsf{\bf D}_{i}f(U)&:=f(U)-f(U^{i}),\\
    \mathsf{\bf D}_{i,j}^2f(U)&:=f(U)-f(U^{i})-f(U^{j})+f(U^{i,j}).
\end{align*}
Note that the derivatives $\mathsf{D}$ and $\mathsf{\bf D}$ obey the relation $\mathsf{\bf D}_{i}f(U)=\mathsf{D}_{U_i}f(\mcal{U}_{n}^{i})$. Also, for random vectors $V,V'$ and $W$, we denote
\begin{align*}
    \gamma_{V,W}(f)&:=\mbb{E}\left(\mathds{1}_{\mbf{D}_{1,2}^{2}f(V)\neq 0}\mbf{D}_{2}f(W)^4\right),\\
    \gamma'_{V,V',W}(f)&:=\mbb{E}\left(\mathds{1}_{\mbf{D}_{1,2}^{2}f(V)\neq 0,\mbf{D}_{1,3}^{2}f(V')\neq 0}\mbf{D}_{2}f(W)^4\right),\\
    B_{n}(f)&:=\sup\{\gamma_{V,W}(f);V,W\ \text{recombinations of}\ \{U,U',U''\}\},\\
    B_{n}'(f)&:=\sup\{\gamma'_{V,V',W}(f);V,V',W\ \text{recombinations of}\ \{U,U',U''\}\}.
\end{align*}

\begin{Theorem}\label{Poincarebinomial}
    Let $n\ge 2$ and $F:=f(\mcal{U}_n)$ be a symmetric function of binomial process with $\mbb{E}f^{2}(\mcal{U}_n)<\infty$. Then, there exists a constant $c_0>0$ not depending on $f$ or $n$ such that
    \begin{align*}
        \mathsf{d}_{K}\left(\frac{F-\mbb{E}F}{\sqrt{\Var F_n}},\mcal{N}(0,1)\right)&\le c_0\Bigg(\frac{\sqrt{n}}{\Var F}\left(\sqrt{nB_{n}(f)}+\sqrt{n^2B_{n}'(f)}+\sqrt{\mbb{E}\mbf{D}_{1}f(U)^4}\right)\\
        &\quad +\sup_{V}~\frac{n}{(\Var F)^2}\mbb{E}|(f(U)-\mbb{E}F)(\mbf{D}_{1}f(V))^3|+\frac{n}{(\Var F)^{\frac{3}{2}}}\mbb{E}|\mbf{D}_{1}f(U)|^3\Bigg),
    \end{align*}
    where the supremum runs over all recombinations $V$ of $\{U,U',U''\}$. 
\end{Theorem} 

We are now ready to prove Theorem~\ref{rateswithconstant}.
\begin{proof}[Proof of Theorem \ref{rateswithconstant}]
    We will apply Theorem \ref{Poincarebinomial} with $f=nE_n$ to obtain the result. We start by bounding $\gamma_{V,W}(f)$ in Theorem \ref{Poincarebinomial}. Following \citet[Proof of Theorem 4.2]{lachieze2019normal}, we obtain
    $$
        \gamma_{V,W}(f)
        \le \int_{\XDE{\mbb{X}}^{4}}\mbb{P}(\mathsf{D}^2_{\xde{v}_1,\xde{v}_2}f(V^{1,2})\neq 0)^{\frac{p/2}{4+p/2}}(\mbb{E}(\mathsf{D}_{\xde{w}_2}f(W^{1,2}\cup\{\xde{w}_1\}))^{4+p/2})^{\frac{4}{4+p/2}}\XDE{\mbb{Q}}^4(d(\xde{v}_1,\xde{v}_2,\xde{w}_1,\xde{w}_2)).
    $$
    Note that in contrast to assuming a bounded $(4+p)$-th moment as in \citet[Theorem 4.2]{lachieze2019normal}, we use Lemma \ref{MBofFn} to bound $\mbb{E}(\mathsf{D}_{\hat{w}_2}f(W^{1,2}\cup\{\hat{w}_1\}))^{4+p/2}$ keeping track of its dependence on $M,\eta,p$ and the $\varepsilon$'s associated with $\xde{\omega}_1$ and $\xde{\omega}_2$. 
    
    As in the proof of Lemma \ref{MBofFn}, the upper bound is a sum of these $\varepsilon$'s which can thus be integrated over due to our Assumption B.2 in Section \ref{AssumpB}.
We thus obtain
    \begin{align*}
        \gamma_{V,W}(f)\lesssim \left(\Big(\frac{M}{\zeta\eta}\Big)^{5}\vee 1\right)^{\frac{4}{4+p/2}}\int_{\XDE{\X}^{2}}\psi_{n}(\xde{x},\xde{x'})\XDE{\mbb{Q}}^{2}(d(\xde{x},\xde{x'})),
    \end{align*}
    where recall the definition of $\psi_n$ in Theorem \ref{rateswithconstant}. Therefore, for the first term of the bound in Theorem \ref{Poincarebinomial}, we have
    \begin{align*}
        \frac{\sqrt{n}}{\Var (nE_n)}\sqrt{nB_n(nE_n)}\lesssim c(M,\eta,p)^{\frac{2}{4+p/2}}\frac{1}{n\Var E_n}\sqrt{\int_{\XDE{\X}^{2}}\psi_{n}(\xde{x},\xde{x'})\XDE{\mbb{Q}}^{2}(d(\xde{x},\xde{x'}))},
    \end{align*}
    with
    \begin{align*}
        c(M,\eta,p)\asymp\Big(\frac{M}{\zeta\eta}\Big)^{5}\vee 1.
    \end{align*}
 This gives the desired bound $S_1$ in Theorem \ref{rateswithconstant}. For the rest of the terms of the bound in Theorem \ref{Poincarebinomial}, following an almost identical argument as in \citet[Proof of Theorem 4.2]{lachieze2019normal} by specifying $c$ there as $c(M,\eta,p)$ above, we obtain the bounds $S_{i}$'s for $i=2,\ldots,5$.
\end{proof}

\section{Proof of Lemma \ref{rank_variancelowerbound} and Lemma \ref{rank_biasbound}}\label{sec:ranklemma}
The proof of both of these results involve the convergence rate of the density ratio estimation similar to Proof of Lemma \ref{Variancelowerbound} and Lemma \ref{lemma:l2convdensityratio}, which we first present in the following.

According to \citet[Proof of Theorem 5.2, Section A.8]{cattaneo2023rosenbaum}, 
letting 
\begin{align*}
    \hat{r}_{\phi}(x):=\frac{n_0}{n_1}\frac{K_{\phi}(x)}{M},
\end{align*}
it can be viewed as an estimator for the density ratio $r:=g_{\phi,\omega,1}/g_{\phi,\omega,0}$ (see also below). From now on, for simplicity, we will drop $\omega$ from the subscripts noting that all statements here hold for any $\omega\in\{0,1\}$. 

Abusing the notation  restricted only to this section, let $\{X_i\}_{i=1}^{n_0}$ and $\{Z_{i}\}_{i=1}^{n_1}$ be two i.i.d.\ sample from the control group ($D=0$) and the treatment group ($D=1$), respectively. Define the catchment area of $x$ as
\begin{align*}
    A_{\phi}(x):=\{z\in \mbb{R}^d:\|\hat{\phi}(x)-\hat{\phi}(z)\|\le \hat{\Phi}_{M}(z)\},
\end{align*}
where $\hat{\Phi}_{M}(z)$ is the $M$-th order statistics of $\{\|\hat{\phi}(X_i)-\hat{\phi}(z)\|\}_{i=1}^{n_0}$. Note that
\begin{align*}
    K_{\phi}(x)=\sum_{j=1}^{n_1}\mathds{1}(Z_j\in A_{\phi}(x)).
\end{align*}
Then, the density ratio estimator can be written as
\begin{align}\label{def:densityratioest_rank}
    \hat{r}_{\phi}(x)=\frac{n_0}{M n_1}\sum_{j=1}^{n_1}\mathds{1}(Z_j\in A_{\phi}(x)) .
\end{align}
Observe now that provided there are no ties in $\{\|\hat{\phi}(X_i)-\hat{\phi}(z)\|\}_{i=1}^{n_0}$, we have $M = \sum_{j=1}^{n_1}\mathds{1}(X_j\in A_{\phi}(x))$ almost surely, so that $$\hat{r}_{\phi}(x)= \frac{\frac{1}{n_1}\sum_{j=1}^{n_1}\mathds{1}(Z_j\in A_{\phi}(x))}{\frac{1}{n_0}\sum_{j=1}^{n_1}\mathds{1}(X_j\in A_{\phi}(x))}.$$
In \citet[Theorem A.2]{cattaneo2023rosenbaum}, it has been shown that
\begin{align*}
    \lim_{n_0\rightarrow\infty} \mbb{E}|\hat{r}_{\phi}(X)-r(\phi(X))|^2=0.
\end{align*}
In the following, we focus on the rate for such a convergence. The arguments will largely follow \citet[Proof of Lemma B.1]{lin2023estimationsupp} and \citet[Proof of Lemma A.3, Section B.3]{cattaneo2023rosenbaum}. However, we will quantify their convergence arguments under our Assumption C.6 in Section \ref{AssumpC}. Following \citet[Proof of Theorem B.3]{lin2023estimationsupp}, in \citet[Proof of Lemma A.3, Section B.3]{cattaneo2023rosenbaum} Part I Case I, we set 
\begin{align*}
3\delta\ge \delta_n=\left(\frac{4}{g_{\min}V_{m'}}\right)^{1/m'}\left(\frac{M}{n_0}\right)^{1/m'},
\end{align*}
where $V_{m'}$ is the volume of the unit ball in $\mbb{R}^{m'}$, and according to Assumption C.6 in Section \ref{AssumpC}, we set 
\begin{align*}
\epsilon=\frac{3\delta L_{\phi}}{g_{\phi,0}(\phi(x))\wedge g_{\phi,1}(\phi(x))}.
\end{align*}
Also, set $\epsilon'=\frac{\delta}{4\text{diam}(S_0)}$. 
Again, similar to Lemma \ref{Variancelowerbound}, we still have the same smaller order errors $\delta_{H_1}$- $\delta_{H_3}$ here due to Lemma \ref{lemma:l2convdensityratio}. For the convenience of presentation, we omit these smaller order terms in the following part of the proof, and only include them in the final step, i.e., all the arguments in the following happen conditional on the event $I=\{|n_0-np_0|<np_0/2\}$ in Lemma \ref{lemma:l2convdensityratio}. For the upper bound, according to \citet[Proof of Lemma A.3, Equation B.6]{cattaneo2023rosenbaum}, we have
\begin{align*}
    \mbb{P}(Z_1\in A_{\phi}(x))&\le \mbb{P}(\|\phi(x)-\phi(Z_1)\|-4\|\hat{\phi}-\phi\|_{\infty}\le \Phi_{M}(Z_1), 4\|\hat{\phi}-\phi\|_{\infty}\le \epsilon'\|\phi(x)-\phi(Z_1)\|)\\
    &\quad + \mbb{P}(\|\phi(x)-\phi(Z_1)\|-4\|\hat{\phi}-\phi\|_{\infty}\le \Phi_{M}(Z_1), 4\|\hat{\phi}-\phi\|_{\infty}> \epsilon'\|\phi(x)-\phi(Z_1)\|)\\
    &=:\Phi_1+\Phi_2,
\end{align*}
where $\Phi_{M}(z)$ is the $M$-th order statistics of $\{\|\phi(X_i)-\phi(z)\|\}_{i=1}^{n_0}$. From \cite{cattaneo2023rosenbaum}, one obtains
\begin{align*}
    \Phi_1\le \mbb{P}\left(\left(\frac{1-\epsilon}{1+\epsilon}-\frac{1+\epsilon}{1-\epsilon}m\epsilon'\right)\frac{g_{\phi,0}(\phi(x))}{g_{\phi,1}(\phi(x))}U\le U_{(M)}\right)+\mbb{P}\left(U_{(M)}>\eta_{n}\frac{M}{n_0}\right),
\end{align*}
where $\eta_n:=4\log(n_0/M)$, $U\sim \operatorname{Unif}[0,1]$ and $U_{(M)}$ is the $M$-th order statistic of $n_0$ independent random variables from $\operatorname{Unif}[0,1]$. According to \citet[Proof of Lemma B.1, Equation S3.5]{lin2023estimationsupp} and \citet[Proof of Lemma A.3, Section B.3]{cattaneo2023rosenbaum}, we can further bound $\Phi_1$ by
\begin{align*}
    \frac{n_0}{M}\Phi_1\le \left(\frac{1-\epsilon}{1+\epsilon}-\frac{1+\epsilon}{1-\epsilon}m\epsilon'\right)^{-1}\frac{g_{\phi,1}(\phi(x))}{g_{\phi,0}(\phi(x))}\frac{n_0}{n_0+1}+\left(\frac{n_0}{M}\right)^{1-2M}.
\end{align*}
As for $\Phi_2$, according to \citet[Proof of Lemma A.3, Section B.3]{cattaneo2023rosenbaum}, we have
\begin{align*}
    \Phi_2\lesssim (\delta\epsilon')^{-m}\mbb{E}(\sup_{z\in S_1}\|\hat{\phi}_{Z_1\rightarrow z}-\phi\|_{\infty}^{m}),
\end{align*}
where for any positive integer $\ell>0$, we write $\hat{\phi}_{(Z_1,\ldots,Z_{\ell})\rightarrow z}$ as the estimator replacing $(Z_1,\ldots,Z_{\ell})$ by $z$ for $z\in S_1^{\ell}$. Combining the bounds above yields
\begin{align*}
    \frac{n_0}{M}\mbb{P}(Z_1\in A_{\phi}(x))-\frac{g_{\phi,1}(\phi(x))}{g_{\phi,0}(\phi(x))}
    &\le \frac{n_0}{M}(\Phi_1+\Phi_2)-\frac{g_{\phi,1}(\phi(x))}{g_{\phi,0}(\phi(x))}\\
    &\lesssim \left(\frac{\epsilon}{1+\epsilon}+\frac{1+\epsilon}{1-\epsilon}m\epsilon'\right)\frac{g_{\phi,1}(\phi(x))}{g_{\phi,0}(\phi(x))}+\frac{g_{\phi,1}(\phi(x))}{g_{\phi,0}(\phi(x))}\frac{1}{n_0+1}+\left(\frac{n_0}{M}\right)^{1-2M}\\
    &\qquad \qquad + \frac{n_0}{M}(\delta\epsilon')^{-m}\mbb{E}(\sup_{z\in S_1}\|\hat{\phi}_{Z_1\rightarrow z}-\phi\|_{\infty}^{m})\\
    &\lesssim \epsilon+\epsilon'+\frac{1}{n_0}+\left(\frac{n_0}{M}\right)^{1-2M}+\frac{n_0}{M}(\delta\epsilon')^{-m}\mbb{E}(\sup_{z\in S_1}\|\hat{\phi}_{Z_1\rightarrow z}-\phi\|_{\infty}^{m}).
\end{align*}
Similarly, following \citet[Proof of Lemma A.3, Section B.3]{cattaneo2023rosenbaum}, one can obtain the lower bound
\begin{align*}
    \frac{n_0}{M}\mbb{P}(Z_1\in A_{\phi}(x))-\frac{g_{\phi,1}(\phi(x))}{g_{\phi,0}(\phi(x))}\gtrsim \epsilon+\epsilon'+\frac{1}{n_0}+\left(\frac{n_0}{M}\right)^{1-2M}-\frac{n_0}{M}(\delta\epsilon')^{-m}\mbb{E}(\sup_{z\in S_1}\|\hat{\phi}_{Z_1\rightarrow z}-\phi\|_{\infty}^{m}).
\end{align*}
Therefore, we have
\begin{align*}
    &\left|\frac{n_0}{M}\mbb{P}(Z_1\in A_{\phi}(x))-\frac{g_{\phi,1}(\phi(x))}{g_{\phi,0}(\phi(x))}\right|\lesssim \epsilon+\epsilon'+\frac{1}{n_0}+\frac{n_0}{M}(\delta\epsilon')^{-m}\mbb{E}(\sup_{z\in S_1}\|\hat{\phi}_{Z_1\rightarrow z}-\phi\|_{\infty}^{m}).
\end{align*}

As for Part I Case II in \citet[Proof of Lemma A.3, Section B.3]{cattaneo2023rosenbaum}, we set $\delta$ and $\epsilon'$ the same as before and let 
$$
\epsilon=\frac{3\delta L_{\phi}}{g_{\phi,0}(\phi(x))\wedge 1}.
$$ 
Note in this case that $r(\phi(x))=g_{\phi,1}(\phi(x))/g_{\phi,0}(\phi(x))=0$. Then, similar to the bound for $\Phi_1$ and $\Phi_2$, following \citet[Proof of Lemma A.3, Section B.3]{cattaneo2023rosenbaum}, we have
\begin{align*}
    \frac{n_0}{M}\mbb{P}(Z_1\in A_{\phi}(x))-\frac{g_{\phi,1}(\phi(x))}{g_{\phi,0}(\phi(x))}&\lesssim \epsilon (1-\epsilon')^{-m}(1-\epsilon')^{-1}\frac{1}{g_{\phi,0}(\phi(x))}\frac{n_0}{n_0+1}+\left(\frac{n_0}{M}\right)^{1-2M}\\
    &\qquad + \frac{n_0}{M}\epsilon(\epsilon')^{-m}\mbb{E}(\sup_{z\in S_1}\|\hat{\phi}_{Z_1\rightarrow z}-\phi\|_{\infty}^{m})\\
    &\lesssim \epsilon+\epsilon'+\frac{1}{n_0}+\epsilon(\epsilon')^{-m} \left(\frac{n_0}{M}\mbb{E}(\sup_{z\in S_1}\|\hat{\phi}_{Z_1\rightarrow z}-\phi\|_{\infty}^{m})\right).
\end{align*}
Combining these two cases above, we conclude that
\begin{align}\label{boundsummand}
    \left|\frac{n_0}{M}\mbb{P}(Z_1\in A_{\phi}(x))-r(\phi(x))\right|\lesssim \epsilon+\epsilon'+\frac{1}{n_0}+(\delta\epsilon')^{-m} \left(\frac{n_0}{M}\mbb{E}(\sup_{z\in S_1}\|\hat{\phi}_{Z_1\rightarrow z}-\phi\|_{\infty}^{m})\right).
\end{align}
Note that \cite{cattaneo2023rosenbaum} proposed the assumption (see Assumption A.1 therein)
\begin{align*}
    \frac{n_0}{M}\mbb{E}(\sup_{z\in S_1}\|\hat{\phi}_{Z_1\rightarrow z}-\phi\|_{\infty}^{m})\rightarrow 0,
\end{align*}
for their consistency result without any rates. Here, we instead obtain a bound where this term features in the bound itself. Now, since
\begin{align*}
    \mbb{E}\hat{r}_{\phi}(x)=\frac{n_0}{M n_1}\sum_{j=1}^{n_1}\mbb{E}\mathds{1}(Z_j\in A_{\phi}(x)),
\end{align*}
by \eqref{boundsummand}, we have
\begin{align}\label{biasofr}
    \left|\mbb{E}\hat{r}_{\phi}(x)-r(\phi(x))\right|\lesssim \epsilon+\epsilon'+\frac{1}{n_0}+(\delta\epsilon')^{-m} \left(\frac{n_0}{M}\mbb{E}(\sup_{z\in S_1}\|\hat{\phi}_{Z_1\rightarrow z}-\phi\|_{\infty}^{m})\right).
\end{align}
Here, we again recall that we have so far omitted those smaller order terms $\delta_{H_1}$-$\delta_{H_3}$ appearing in Lemma \ref{lemma:l2convdensityratio} and all our arguments so far are conditional on the event $I=\{|n_0-np_0|<np_0/2\}$, as in Lemma \ref{lemma:l2convdensityratio}. We will only including those error terms back at the end of our argument.

For the variance estimate in the second assertion, note according to the law of total variance that $\Var \hat{r}_{\phi}(x)=\mbb{E}\Var (\hat{r}_{\phi}(x)|X)+\Var \mbb{E}(\hat{r}_{\phi}(x)|X)$. The first term can be directly bounded via \eqref{biasofr} as
\begin{align*}
    \mbb{E}\Var (\hat{r}_{\phi}(x)|X)&=\frac{n_0^2}{M^2 n_1^2}\mbb{E}\Var\left(\sum_{j=1}^{n_1}\mathds{1}(Z_j\in A_{\phi}(x))|X\right)\nonumber\\
    &=\frac{n_0^2}{M^2 n_1}\mbb{E}\Var\left(\mathds{1}(Z\in A_{\phi}(x))|X\right)\nonumber\\
    &\le \frac{n_0^2}{M^2 n_1}\mbb{P}(Z\in A_{\phi}(x))\le \frac{n_0}{Mn_1}\mbb{E}\hat{r}_{\phi}(x)\lesssim \frac{n_0}{Mn_1}.
\end{align*}
For the second term, it holds that
\begin{align*}
    \Var \mbb{E}(\hat{r}_{\phi}(x)|X)&=\frac{n_0^2}{M^2 n_1^2}\Var\mbb{E}\bigg(\sum_{j=1}^{n_1}\mathds{1}(Z_j\in A_{\phi}(x))|X\bigg)\nonumber\\
    &=\frac{n_0^2}{M^2}\Var\mathds{1}(Z\in A_{\phi}(x))=\frac{n_0^2}{M^2}(\mbb{E}\mathds{1}(Z\in A_{\phi}(x))-(\mbb{E}\mathds{1}(Z\in A_{\phi}(x)))^2).
\end{align*}
Following \citet[Proof of Theorem B.3(ii)]{lin2023estimationsupp} and \citet[Proof of Lemma A.3, Part II]{cattaneo2023rosenbaum}, we have
\begin{align*}
    \frac{n_0^2}{M^2}\Var\mathds{1}(Z\in A_{\phi}(x))\asymp \frac{1}{M}+(\delta\epsilon')^{-2m} \bigg(\frac{n_0^2}{M^2}\mbb{E}(\sup_{z\in S_1^2}\|\hat{\phi}_{(Z_1,Z_2)\rightarrow z}-\phi\|_{\infty}^{2m})\bigg),
\end{align*}
where the second term comes from the third term in \citet[Proof of Lemma A.3, Equation B.8]{cattaneo2023rosenbaum} when replacing $\hat{\phi}$ by $\phi$. Note here, \citet[Assumption A.1]{cattaneo2023rosenbaum} requires that
\begin{align*}
    \lim_{n_0\rightarrow\infty}\frac{n_0^2}{M^2}\mbb{E}(\sup_{z\in S_1^2}\|\hat{\phi}_{(Z_1,Z_2)\rightarrow z}-\phi\|_{\infty}^{2m})=0,
\end{align*}
which does not offer any rates of convergence. Instead, we keep this term in our bound. Combining the two bounds above, we obtain via the law of total variance that
\begin{align}\label{varofr}
    \Var \hat{r}_{\phi}(x)\lesssim \frac{n_0}{M n_1}+\frac{1}{M}+(\delta\epsilon')^{-2m} \bigg(\frac{n^2}{M^2}\mbb{E}(\sup_{z\in S_1^2}\|\hat{\phi}_{(Z_1,Z_2)\rightarrow z}-\phi\|_{\infty}^{2m})\bigg).
\end{align}
Therefore, putting \eqref{biasofr} and \eqref{varofr} together, we have
\begin{align}\label{l2x}
    \mbb{E}|\hat{r}_{\phi}(x)-r(\phi(x))|^2
    &= |\mbb{E}\hat{r}_{\phi}(x)-r(\phi(x))|^2+\Var \hat{r}_{\phi}(x) \nonumber\\
    &\lesssim \epsilon^2+\epsilon'^2+\frac{1}{M}+\frac{1}{n^{2/3}}+(\delta\epsilon')^{-2m} \bigg(\frac{n^2}{M^2}\mbb{E}(\sup_{z\in S_1^2}\|\hat{\phi}_{(Z_1,Z_2)\rightarrow z}-\phi\|_{\infty}^{2m})\bigg).
\end{align}
Finally, plugging the bound \eqref{l2x} in \citet[Proof of Theorem A.2, Section B.4]{cattaneo2023rosenbaum}, we obtain the following bound for the $L^2$ convergence of $\hat{r}_{\phi}(X)$:
\begin{align}\label{L2x}
    &\mbb{E}|\hat{r}_{\phi}(X)-r(\phi(X))|^2 \lesssim \left(\frac{M}{n_0}\right)^{1/m'}\hspace{-.3cm}+\epsilon^2+(\epsilon')^2+\frac{1}{M}+(\delta\epsilon')^{-2m} \left(\frac{n^2}{M^2}\mbb{E}(\sup_{z\in S_1^2}\|\hat{\phi}_{(Z_1,Z_2)\rightarrow z}-\phi\|_{\infty}^{2m})\right).
\end{align}
The bound above involves an additional term related to the approximation error for $\phi$ by $\hat{\phi}$, as well as the error terms $\epsilon,\epsilon',\delta$. According to the definition of $\delta$ above, it is only required that 
\begin{align*}
\delta\gtrsim \left(\frac{M}{np_0}\right)^{1/m'}\ge \left(\frac{M}{n\eta}\right)^{1/m'}.
\end{align*}
Moreover, by definition, we have $\epsilon,\epsilon'\asymp \delta$. When we take $\delta\asymp (M(n\eta)^{-1})^{1/m'}$, it yields that
\begin{align*}
    &\mbb{E}|\hat{r}_{\phi}(X)-r(\phi(X))|^2 \lesssim \left(\frac{M}{n_0}\right)^{1/m'}+\frac{1}{M}+\left(\frac{M}{n}\right)^{4m/m'} \bigg(\frac{n^2}{M^2}\mbb{E}(\sup_{z\in S_1^2}\|\hat{\phi}_{(Z_1,Z_2)\rightarrow z}-\phi\|_{\infty}^{2m})\bigg).
\end{align*}
Recall the smaller order terms $\delta_{H_1}$-$\delta_{H_3}$ in Lemma \ref{lemma:l2convdensityratio}. Here, we have so far omitted all those terms and have only focused on bounding conditional on the event $I=\{|n_0-np_0|<np_0/2\}$, as in Lemma \ref{lemma:l2convdensityratio} so far for simplicity that also appeared in Lemma \ref{Variancelowerbound}. Now adding those terms back, we obtain the bound
\begin{align*}
    \mbb{E}|\hat{r}_{\phi}(X)-r(\phi(X))|^2
    &\lesssim 
    \frac{1}{\eta^2}\left(\frac{M}{n\eta}\right)^{1/m'}+\delta_{H_1}+(\delta_{H_2}+1)\cdot\frac{1}{\eta^2 M}+\delta_{H_3}\nonumber\\
    &\qquad +\left(\frac{M}{n}\right)^{4m/m'} \bigg(\frac{n^2}{M^2}\sup_{x_1,x_2\in \X}\|\hat{\phi}_{\omega}(\cdot;x_1,x_2)-\phi_{\omega}\|_{\infty}^{2m}\bigg).
\end{align*}
We are now ready to prove the bounds in Lemmas \ref{rank_variancelowerbound} and \ref{rank_biasbound}. Focusing first on Lemma \ref{rank_variancelowerbound}, plugging in the definition of the density ratio in \eqref{def:densityratioest_rank}, it becomes
\begin{align}\label{densityratioconv_rank}
    &\mbb{E}\bigg(\bigg(\frac{K_{\phi}(1)}{M}-\frac{1-e(X_1)}{e(X_1)}\bigg)^{2}\bigg|\{D_i\}_{i=1}^{n}\bigg)\mathds{1}(n_1>0)\nonumber\\
    &\lesssim \frac{1}{\eta^2}\left(\frac{M}{n\eta}\right)^{1/m'}\hspace{-.3cm}+\delta_{H_1}+(\delta_{H_2}+1)\cdot\frac{1}{\eta^2 M}+\delta_{H_3}+\left(\frac{M}{n}\right)^{4m/m'} \bigg(\frac{n^2}{M^2}\sup_{x_1,x_2\in \X}\|\hat{\phi}_{\omega}(\cdot;x_1,x_2)-\phi_{\omega}\|_{\infty}^{2m}\bigg),
\end{align}
proving the first bound.

For the variance estimation, following the corresponding arguments in the Proof of Lemma \ref{Variancelowerbound} in Section \ref{pf:covariate} yields
\begin{align*}
    |n\Var E_{\phi,n}^*-\sigma_{\phi}^2|&\lesssim \mbb{E}\bigg(\frac{1}{n}\sum_{i=1,D_i=1}^{n}\bigg(\bigg(1+\frac{K_{\phi}^*(i)}{M}\bigg)^2-\left(\frac{1}{e(X_i)}\right)^2\bigg)\sigma_{\phi,1}^2(X_i)\bigg)
    \nonumber\\&\quad + \mbb{E}\bigg(\frac{1}{n}\sum_{i=1,D_i=0}^{n}\bigg(\bigg(1+\frac{K_{\phi}^*(i)}{M}\bigg)^2-\left(\frac{1}{1-e(X_i)}\right)^2\bigg)\sigma_{\phi,0}^2(X_i)\bigg)\\
    &\coloneqq \mbb{E}J_2'+\mbb{E}J_3',
\end{align*}
where
\begin{align}\label{def:sigmaphi}
        \sigma_{\phi}^2:=\Var (\mu_{\phi,1}(L_{\phi,1})-\mu_{\phi,0}(L_{\phi,0}))+\mbb{E}\bigg(\frac{\sigma_{\phi,1}^2(X)}{e(X)}+\frac{\sigma^2_{\phi,0}(X)}{1-e(X)}\bigg),
\end{align}
with $e(x):=\mbb{P}(D=1|X=x)$ and $\sigma_{\phi,\omega}(x)^2:= \mbb{E} [(Y(\omega)-\mu_{\phi,\omega}(L_{\phi,\omega}))^2 | X=x]$. 

Similar to bounding $\mbb{E}|J_2|$ in \eqref{J2} in the Proof of Lemma \ref{Variancelowerbound}, we have
\begin{align*}
    \mbb{E}|J_2'|\lesssim \mbb{E}\bigg(\mbb{E}\bigg(\bigg(\frac{K_{\phi}^*(1)}{M}-\frac{1-e(X_1)}{e(X_1)}\bigg)^{2}\Bigg|\{D_i\}_{i=1}^{n}\bigg)\mathds{1}(n_1>0)\bigg)^{1/2}.
\end{align*}
Comparing to the bound in \eqref{densityratioconv_rank}, a bound for $\mbb{E}|J_2'|$ above can be obtained by replacing $K_{\phi}(1)$ (defined by using $\{\hat{L}_{\phi,\omega,i}=\hat{\phi}_{\omega}(X_i)\}$) by $K^*_{\phi}(1)$ (defined by replacing $\hat{\phi}_{\omega}(X_i)$ by $\phi_{\omega}(X_i)$) in the proof. Then, without the additional term involving the convergence of $\hat{\phi}$ to $\phi$ in \eqref{densityratioconv_rank}, we have
\begin{align*}
    \mbb{E}|J_2'|\lesssim \frac{1}{\eta}\left(\frac{M}{n\eta}\right)^{1/(2m')}+\delta_{H_1}^{1/2}+(\delta_{H_2}^{1/2}+1)\cdot\frac{1}{\eta M^{1/2}}+\delta_{H_3}^{1/2}.
\end{align*}
This proves the second assertion. The same bound also holds for $\mbb{E}|J_3'|$ by symmetry, yielding
\begin{align}\label{generaleps}
    |n\Var E_{\phi,n}^*-\sigma_{\phi}^2|\lesssim\frac{1}{\eta}\left(\frac{M}{n\eta}\right)^{1/(2m')}+\delta_{H_1}^{1/2}+(\delta_{H_2}^{1/2}+1)\cdot\frac{1}{\eta M^{1/2}}+\delta_{H_3}^{1/2}+\frac{1}{\eta^{3}n^{1/3}},
\end{align}
where the last term $\frac{1}{\eta^{3}n^{1/3}}$ appears due to bounding $\mbb{E}J_{212}$ similar to \eqref{J2boundtotal} outside the event $I$. This proves the third assertion and completes the proof of Lemma \ref{rank_variancelowerbound}.

Next, we prove Lemma \ref{rank_biasbound}. According to \citet[Proof of Lemma A.2, Section B.1]{cattaneo2023rosenbaum} along with similar arguments used in \eqref{adderrorinbias}, we have
\begin{align*}
    \mbb{E}|B_{\phi,n}-\hat{B}_{\phi,n}|&\lesssim (\eta^{-k/m'}+\delta_{H_1})\Bigg(\bigg(\frac{M}{n}\bigg)^{k/m'}+n^{-k/2}+\max_{l\in [k-1]}\bigg(n^{-\gamma_{\phi,l}}\bigg(\left(\frac{M}{n}\right)^{l/m'}+n^{-l/2}\bigg)\bigg)\nonumber\\
    &\qquad+\lim_{\delta\rightarrow 0}\mbb{E}\sup_{x,y\in\mbb{X},\|\phi(x)-\phi(y)\|\le \delta}\|(\hat{\phi}-\phi)(x)-(\hat{\phi}-\phi)(x)\|_{\infty}\Bigg),
\end{align*}
where $k:=\lfloor m'/2 \rfloor\vee 1 +1$ and $\gamma_{\phi,l}$'s are given in Assumption D.4 in Section \ref{AssumpD}. This proves the first assertion. 
Moreover,
\begin{align*}
    \sqrt{n}\mbb{E}|\Delta E_{\phi,n}|
    &=\sqrt{n}\mbb{E}\bigg|\frac{1}{n}\sum_{i=1}^{n}(2D_i-1)\bigg(\bigg(1+\frac{K_{\phi}(i)}{M}\bigg)-\bigg(1+\frac{K^*_{\phi}(i)}{M}\bigg)\bigg)\bm{\varepsilon}_{\phi,i}\bigg|\\
    &\le \sqrt{n}\bigg(\mbb{E}\bigg(\frac{1}{n}\sum_{i=1}^{n}(2D_i-1)\bigg(\bigg(1+\frac{K_{\phi}(i)}{M}\bigg)-\left(1+\frac{K^*_{\phi}(i)}{M}\right)\bigg)\bm{\varepsilon}_{\phi,i}\bigg)^2\bigg)^{\frac{1}{2}}\\
    &=\bigg(\mbb{E}\bigg(\frac{1}{n}\sum_{i=1}^{n}\bigg(\bigg(1+\frac{K_{\phi}(i)}{M}\bigg)-\bigg(1+\frac{K^*_{\phi}(i)}{M}\bigg)\bigg)^2\sigma^2_{\phi,D_i}(X_i)\bigg)\bigg)^{\frac{1}{2}},
\end{align*}
where in the last step, we have used the independence and conditional mean zero properties of $\bm{\varepsilon}_{\phi,i}$. Now, we can further bound
\begin{align}\label{H1234}
    &\mbb{E}\bigg(\frac{1}{n}\sum_{i=1}^{n}\bigg(\bigg(1+\frac{K_{\phi}(i)}{M}\bigg)-\bigg(1+\frac{K^*_{\phi}(i)}{M}\bigg)\bigg)^2\sigma^2_{\phi,D_i}(X_i)\bigg)\nonumber\\
    &\lesssim \mbb{E}\bigg(\frac{1}{n}\sum_{i=1,D_i=1}^{n}\bigg(\frac{K_{\phi}(i)}{M}-\frac{K_{\phi}^*(i)}{M}\bigg)^2\sigma^2_{\phi,1}(X_i)\bigg)+\mbb{E}\bigg(\frac{1}{n}\sum_{i=1,D_i=0}^{n}\bigg(\frac{K_{\phi}(i)}{M}-\frac{K_{\phi}^*(i)}{M}\bigg)^2\sigma^2_{\phi,0}(X_i)\bigg)\nonumber\\
    &=:H_1+H_2.
\end{align}
We only bound $H_1$, since $H_2$ can be bounded in an identical way. Similar to \eqref{J2}, $H_1$ can be bounded by
\begin{align*}
    H_1&\lesssim\mbb{E}\bigg(\bigg(\frac{K_{\phi}(1)}{M}-\frac{K_{\phi}^*(1)}{M}\bigg)^{2}\bigg|\{D_i\}_{i=1}^{n}\bigg)\mathds{1}(n_1>0).
\end{align*}
Recall that $K_{\phi}(1)$ is defined by using $\{\hat{L}_{\phi,\omega,i}=\hat{\phi}_{\omega}(X_i)\}$ while $K^*_{\phi}(1)$ is defined by replacing $\hat{\phi}_{\omega}(X_i)$ as $\phi_{\omega}(X_i)$ in the definition of $K_{\phi}(1)$. Following the proof of \eqref{densityratioconv_rank}, and noting that the only difference here is the term involving approximation error of $\phi$ by $\hat{\phi}$, we have
\begin{align*}
    H_1\lesssim (\delta\epsilon')^{-2m} \bigg(\frac{n_0^2}{M^2}\mbb{E}(\sup_{x_1,x_2\in \mbb{X}}\|\hat{\phi}_{\omega}(\cdot;x_1,x_2)-\phi_{\omega}\|_{\infty}^{2m})\bigg).
\end{align*}
The same bound also holds for $H_2$ similarly. Together, \eqref{H1234} yields
\begin{align}\label{withdeltabiasbound}
    \sqrt{n}\mbb{E}|\Delta E_{\phi,n}|
    &\lesssim(\delta\epsilon')^{-m} \bigg(\frac{n_0^2}{M^2}\mbb{E}(\sup_{x_1,x_2\in \mbb{X}}\|\hat{\phi}_{\omega}(\cdot;x_1,x_2)-\phi_{\omega}\|_{\infty}^{2m})\bigg)^{1/2}.
\end{align}
Taking again $\epsilon,\epsilon'\asymp \delta\asymp (M(n\eta)^{-1})^{1/m'}$, we obtain
\begin{align*}
    \sqrt{n}\mbb{E}|\Delta E_{\phi,n}|
    \lesssim \left(\frac{M}{n}\right)^{2m/m'} \bigg(\frac{n^2}{M^2}\sup_{x_1,x_2\in \X}\|\hat{\phi}_{\omega}(\cdot;x_1,x_2)-\phi_{\omega}\|_{\infty}^{2m}\bigg)^{1/2}
\end{align*}
yielding the second assertion in Lemma \ref{rank_biasbound}.

\end{document}